\documentclass[12pt]{amsart}
\usepackage{amsfonts,latexsym,amsmath, amssymb}
\usepackage{mathrsfs,MnSymbol}
\usepackage{url,color}
\usepackage{upgreek}
\usepackage{fancyhdr}
\usepackage{hyperref}
\usepackage{times}
\usepackage{scalefnt}
\usepackage{url}
\usepackage{cite}
\usepackage{amsmath}

\newcommand{\bea}{\begin{eqnarray}}
\newcommand{\eea}{\end{eqnarray}}
\def\beaa{\begin{eqnarray*}}
\def\eeaa{\end{eqnarray*}}
\def\ba{\begin{array}}
\def\ea{\end{array}}
\def\be#1{\begin{equation} \label{#1}}
\def \eeq{\end{equation}}

\def\be{{\beta}}

\def\H{{\mathbb{H}}}
\def\R{{\mathbb{R}}}
\def\C{{\mathbb{C}}}
\def\S{{\mathbb{S}}}

\newcommand{\<}{  \langle   }
\renewcommand{\>}{  \rangle   }

\DeclareMathOperator{\ch}{ch}
\DeclareMathOperator{\sh}{sh}

\newtheorem{theorem}{Theorem}[section]
\newtheorem{lemma}[theorem]{Lemma}
\newtheorem{proposition}[theorem]{Proposition}

\numberwithin{equation}{section}

\numberwithin{equation}{section}

\topmargin - .23 in

\begin{document}

%\scalefont{1.10}

\title{Equivariant Schr\"{o}dinger maps from two dimensional hyperbolic space }
\author{Jiaxi Huang, Youde Wang, Lifeng Zhao}

\thanks{}

\thanks{ }

\thanks{}	

\thanks{}	

\begin{abstract}
  In this article, we consider the equivariant Schr\"odinger map from $\H^2$ to $\S^2$ which converges to the north pole of $\S^2$ at the origin and spatial infinity of the hyperbolic space. If the energy of the data is less than $4\pi$, we show that the local existence of Schr\"{o}dinger map. Furthermore, if the energy of the data sufficiently small, we prove the solutions are global in time. 
\end{abstract}
\maketitle

\setcounter{tocdepth}{2}
\pagenumbering{roman} \tableofcontents \newpage \pagenumbering{arabic}

\section{Introduction}
In this article, we consider the Schr\"{o}dinger map equation
\begin{equation}           \label{Schrodinger map}
\frac{\partial u}{\partial t}=J\tau(u),
\end{equation}
where $u(x,t): [0,T]\times\H^2 \rightarrow \S^2$, $\tau(u)$ is the tension field of $u$ and $J$ is complex structure on $\S^2$. The equation admits the conserved energy
\begin{equation*}
E(u)=\frac{1}{2} \int_{\H^2} |du|^2 {\rm dvol}_g,
\end{equation*}
where ${\rm dvol}_g$ is the volume form of $(\H^2,\ g)$.

The Schr\"{o}dinger maps from Euclidean spaces have been intensely studied in the last decades. The local well-posedness of Schr\"{o}dinger maps was established by Sulem, Sulem and Bardos \cite{SuSuBa} for $S^2$ target, Ding and Wang \cite{DingWang1,DingWang2} and McGahagan \cite{Mc} for general K\"{a}hler manifolds. Ionescu and Kenig \cite{IoKe} obtained the global well-posedness of maps into $S^2$ with small data in the critical Besov spaces $\dot{B}^{\frac{d}{2}}_Q(R^d,S^2)$, $Q\in S^2$ for $d\geq 3$. The global well-posedness for maps $\R^d\rightarrow S^2$, $d\geq 2$ with small critical Sobolev norms was obtained by Bejenaru, Ionescu, Kenig and Tataru \cite{BeIoKeTasmall}. However, the Schr\"{o}dinger map equation with large data is a much more dufficult problem. When the target is $S^2$, there exists a collection of families $\mathcal{Q}^m$ (see \cite{BeTa}) of finite energy stationary solutions for integer $m\geq 1$; When the target is $\H^2$, there is not nontrival equivariant stationary solution with finite energy.  Hence, Bejenaru, Ionescu, Kenig and Tataru \cite{BeIoKeTaS2,BeIoKeTaH2} proved the global well-posedness and scattering for equivariant Schr\"{o}dinger maps $\R^2\rightarrow \S^2$ with energy blow the ground state and equivariant Schr\"{o}dinger maps $\R^2 \rightarrow \H^2$ with finite energy. When the energy of maps is larger than that of ground state, the dynamic behaviors are complicated. The asymptotic stability and blow-up for Schr\"{o}dinger maps have been considered by many authors for instance \cite{GuKaTs1,GuKaTs2,GuNaTs,BeTa,MeRaRo,Pe}. We refer to \cite{KTV} for more open problems in this field.

The above results are restricted on flat domains, naturally, we can consider geomertic flow on curved manifolds. Because the hyperbolic spaces are symmetric and noncompact, geometric flows from hyperbolic spaces are natural starting points. The heat flow between hyperbolic spaces is an interesting model because it is related to the Schoen-Li-Wang conjecture (see Lemm, Markovic \cite{LeMa}). For such heat flow, Li and Tam \cite{LiTam} obtained the sufficient conditions to ensure that the harmonic map between hyperbolic spaces can be solved by solving the heat flow. In recent years, there are many works concerning wave maps on hyperbolic spaces which are expected to have many similar phenomenon to Schr\"{o}dinger maps.  D'Ancona and Qidi Zhang \cite{DAZhang} showed the global existence of equivariant wave maps from hyperbolic spaces $\H^d$ for $d\geq 3$ to general targets for small initial data in $H^{\frac{d}{2}}\times H^{\frac{d}{2}-1}$. The problem was also intensely studied by Lawrie, Oh, Shahshahani \cite{LaOhSha,LaOhShagap,LaOhShalarge,LaOhShaCauchy} and Li, Ma, Zhao \cite{LiMaZhao}. Since the wave maps $\H^2\rightarrow\H^2$ or $\S^2$ have a family of equivariant harmonic maps, \cite{LaOhSha} and \cite{LaOhShagap} proved the stability of stationary $k$-equivariant wave maps by analyzing spectral properties of the linearized operator. \cite{LaOhShalarge} continued to consider this problem and showed the soliton resolution for equivariant wave maps $\H^2\rightarrow\H^2$ with initial data $(\psi_0,\ \psi_1)\in \mathcal{E}_{\lambda}$ for $0\leq\lambda\leq\Lambda$ by profile decomposition. For initial data without any symmetric assumption,  Li, Ma and Zhao \cite{LiMaZhao} proved that the small energy harmonic maps from $\H^2$ to $\H^2$ are asymptotically stable under the wave map recently. \cite{LaOhShaCauchy} established global well-posedness and scattering for wave maps from $\H^d$ for $d\geq 4$ into Riemannian manifolds of bounded geometry for small initial data in the critical Sobolev space. As a geometric flow, Schr\"{o}dinger map is a special case of Landau-Lifshitz flow. Li and Zhao \cite{LiZhao} proved that the solution of Landau-Lifshitz $u(t,x)$ from $\H^2$ to $\H^2$ converges to some harmonic map as $t\rightarrow\infty$ when the Gilbert coefficient is positive.

The Schr\"{o}dinger maps on $\H^2$ exhibits markedly different phenomena from its Euclidean counterpart. First, the most interesting feature is that there is an abundance of equivariant harmonic maps introduced by \cite{LaOhSha}. Precisely, when the target is $\S^2$, there is a family of equivariant harmonic maps with energy $4\pi \frac{\lambda^2}{1+\lambda^2}$ for $\lambda\in[0,+\infty)$; When the target is $\H^2$, we also have a family of equivariant harmonic maps with energy $4\pi\frac{1+\lambda^2}{1-\lambda^2}$ for $\lambda\in[0,1)$. Naturally, the dynamic behaviors of solutions with energy above the harmonic maps are of great interest. Second, the maps still exhibit features of mass critical equation, though it lacks scaling symmetry. Indeed, in the Coulomb gauge, the Schr\"{o}dinger map can be reduced to two coupled Schr\"{o}dinger equations. If the support of initial data is contained in a open ball $B_{\epsilon}(\textbf{0})$ for $\epsilon>0$ small, then the solutions will not exhibit the global geometry of the domain and thus can be approximated by solutions to the corresponding scaling invariant mass critical Schr\"{o}dinger equations $\R^{2}\rightarrow \S^2$. Third, the notable feature of the problem is the better dispersive estimates of the operator $e^{it\Delta_{\H^2}}$ than the Euclidean counterpart. The stronger dispersion are possible due to the more robust geometry at infinity of noncompact symmetric spaces compared to Euclidean spaces. The above features make (\ref{Schrodinger map}) an interesting model for investigating the well-posedness for large data and the stability of stationary solutions.

In this paper, we establish the local well-posedness for large data and global well-posedness for small initial data.

To explain the main results in more detail, we give a more precise account. As both the domain and the target are rotationally symmetric, the map $u$ is called $m$-equivariant, if $u$ satisfies $u \circ \rho =\rho \circ u$ for all rotations $\rho \in SO(2)$. Since $u$ is a map $\H^2 \rightarrow \S^2$ here, in the polar coordinates, $u$ is $m$-equivariant if and only if $u$ can be written as
\begin{equation*}
u(r,\theta)=e^{m\theta R}\bar{u}(r).
\end{equation*}
Here $R$ is the generator of horizontal rotations, which is defined as
\begin{equation*}
R:=\left( \begin{matrix}
0 & -1 & 0\\
1 & 0 & 0\\
0 & 0 & 0
\end{matrix}
\right),\ \ Ru=\vec{k} \times u.
\end{equation*}
where $\vec{k}=(0,\ 0,\ 1)^T$. We denote $\vec{i}:=(1,\ 0,\ 0)^T$ and $\vec{j}:=(0,\ 1,\ 0)^T$.
The energy of $m$-equivariant maps can be expressed as
\begin{equation*}
E(u)=\pi \int_0^{\infty} \left(|\partial_r \bar{u}|^2 + \frac{m^2}{\sinh^2 r} (\bar{u}_1^2+\bar{u}_2^2) \right) \sinh r dr.
\end{equation*}
If $m\ne 0$, then $E(u)<\infty$ implies that $\lim\limits_{r\rightarrow 0}u_1=\lim\limits_{r\rightarrow 0}u_2=0$. Due to the exponential decay of $\sinh^{-1}r$, we assume that $\lim\limits_{r\rightarrow \infty}u_1,\ \lim\limits_{r\rightarrow \infty}u_2 \geq 0$, which gives $\lim\limits_{r\rightarrow \infty}u_3=\frac{1-\lambda ^2}{1+\lambda^2} \geq 1$ for $\lambda \in [0,1)$ by $u_1^2+u_2^2+u_3^2=1$. The equivariant Schr\"{o}dinger map (\ref{Schrodinger map}) admits solitons, which are equivariant harmonic maps $u$ such that $u\times \Delta u=0$. In contrast to the Sch\"{o}dinger maps from Euclidean spaces, the Schr\"{o}dinger maps on $\H^2$ admit harmonic maps with any energy $E(u)<4\pi$ for $\S^2$ target and $E(u)<\infty$ for $\H^2$ target. In fact, for $u:\H^2\rightarrow\S^2$ with endpoint $u_3(\infty)=\frac{1-\lambda^2}{1+\lambda^2}$ for $\lambda\in[0,\infty)$, there exists equivariant stationary solution to (\ref{Schrodinger map})
\begin{equation*}
Q_{\lambda}=(\frac{2\lambda\tanh \frac{r}{2}}{1+(\lambda \tanh \frac{r}{2})^2},\ 0,\ \frac{1-(\lambda \tanh \frac{r}{2})^2}{1+(\lambda \tanh \frac{r}{2})^2}),
\end{equation*}
with energy $E(Q_{\lambda})=4\pi\frac{\lambda^2}{1+\lambda^2}$. For $u:\H^2\rightarrow\H^2$ with endpoint $u_3(\infty)=\frac{1+\lambda^2}{1-\lambda^2}$ for $\lambda\in[0,1)$, there exists equivariant stationary solution to (\ref{Schrodinger map})
\begin{equation*}
P_{\lambda}=(\frac{2\lambda\tanh \frac{r}{2}}{1-(\lambda \tanh \frac{r}{2})^2},\ 0,\ \frac{1+(\lambda \tanh \frac{r}{2})^2}{1-(\lambda \tanh \frac{r}{2})^2}),
\end{equation*}
with energy $E(P_{\lambda})=4\pi\frac{\lambda^2}{1-\lambda^2}$.

This leads us to consider the equivariant Schr\"{o}dinger maps in the classes
\begin{equation*}
\mathcal{E}_{\lambda}=\left\{ u:\H^2\rightarrow \S^2 \big| E(u) <\infty ,\ \lim\limits_{r\rightarrow 0}u_3=1,\ \lim\limits_{r\rightarrow \infty}u_3=\frac{1-\lambda^2}{1+\lambda^2}  \right\},\ \ \lambda\in [0,1),
\end{equation*}
but the case $\lambda >0$ is difficult, we will not consider here. Let $u:\H^2\rightarrow\S^2\subset\R^3$ be a smooth map. The Sobolev norm $\mathfrak{H}^k(\H^2;\S^2)$ are defined by
\begin{equation*}
\left\lVert u\right\rVert_{\mathfrak{H}^k}^2:=\sum\limits_{i=1}^k\int_{\H^2}|\nabla^{i-1}du|_g^2\ {\rm dvol}_g.
\end{equation*}
The main results are the following.
\begin{theorem}   \label{main result 1}
If $u_0\in \mathfrak{H}^3$, then there exists $T>0$, such that (\ref{Schrodinger map}) has a unique solution in $L^{\infty}_t([0,T];\mathfrak{H}^3)$.
\end{theorem}
\begin{theorem}   \label{main result 2}
If $u_0\in \mathfrak{H}^1$ is a 1-equivariant map satisfying $u_0\in \mathcal{E}_0$ and $E(u_0)<4\pi$, then there exists $T>0$, such that (\ref{Schrodinger map}) has a unique solution $u\in L^{\infty}_t([0,T];\mathfrak{H}^1)$ in the class $\mathcal{E}_0$ defined as the unique limit of smooth solution in $\mathfrak{H}^3$. In particular, there exists $\epsilon>0$ such that $E(u_0)<\epsilon$, then for any compact interval $J\subset \R $, there exists a unique solution $u\in L^{\infty}_t(J;\mathfrak{H}^1)$.
\end{theorem}
\noindent\emph{Remark 1.2.} In Theorem \ref{main result 2}, we restrict the map $u$ in the class $\mathcal{E}_0$ for initial data $u_0\in \mathfrak{H}^1$. To obtain the existence of solutions in $\mathfrak{H}^1$, we need to prove the Lipschitz continuity of $u$ with respect to $u_0$. If the map $u\in \mathcal{E}_{\lambda}$ for $\lambda>0$, then the third component $u_3$ of $u$ does not convergence to $1$ as $r\rightarrow\infty$, thus the argument of Lipschitz continuity fails. Therefore, we need to restrict $u$ in $\mathcal{E}_0$.

\noindent\emph{Remark 1.3.} In Theorem \ref{main result 2}, scattering for small data is not expected generally. Represented in the Coulomb gauge, (\ref{Schrodinger map}) can be reduced to the coupled mass-critical Schr\"{o}dinger equations with potentials, i.e $(\psi^+,\ \psi^-)$-system. However, the Schr\"{o}dinger operator admits discrete spectrum in one equation of the system which is in sharp contrast with the Schr\"{o}dinger map from $\R^2$. In fact, we show that the $L^4L^4$-bound for $\psi^{\pm}$ depends on the compact interval $J\subset\R$, which leads to the $L^{\infty}_t(J;\mathfrak{H}^1)$-bound for $u$ depends on interval $J$.

Theorem \ref{main result 1} and \ref{main result 2} is of similar flavor to the result of \cite{Mc,BeIoKeTaS2} in the flat domain $\R^2$. The first step is to prove the local existence for Schr\"{o}dinger map with data $u_0\in \mathfrak{H}^3$ by approximation of wave maps (see \cite{Mc}). The second step is to show the existence for equivariant Schr\"{o}dinger map with data $u_0\in \mathfrak{H}^1$. Since we restrict ourselves to the class of equivariant Schr\"{o}dinger maps, the symmetry allow us to use Coulomb gauge. The Coulomb gauge condition impose some restriction on the connection form $A$ , which allow us to choose the particular solution $A_1=0$. Using the Coulomb gauge as our choice of frame on $T_u\H^2$, we can rewrite the equations for $\partial_r u$ and $\partial_{\theta}u$ which lead to a $(\psi^+,\ \psi^-)$-system of mass-critical Schr\"{o}dinger equations with potentials. Then it suffices to consider the Cauchy problem of the $(\psi^+,\ \psi^-)$-system. In order to establish the well-posedness for data in the space $L^2(\H^2)$, we prove the Strichartz estimates for Schr\"{o}dinger operator with such potentials. In fact, we can get the dispersive estimates for $0<t<1$ with more general potentials $V\geq 0$ and $V\in e^{-\alpha r}L^{\infty}(\H^2)$ for $\alpha \geq 1$. Since our interest lies in the solutions which correspond to the geometric flow, we show that the solutions of the system satisfy the compatibility condition. To construct the Schr\"{o}dinger map $u$ from $\psi^{\pm}$, the key observation is that $\psi^+$ or $\psi^-$ contain all the information of the map as in \cite{BeIoKeTaS2}. Hence, we can recover the map $u(t)$ from $\psi^{\pm}(t)$ for initial data $R_{\pm}\psi^{\pm}(0)\in H^2$. Furthermore by the result in Theorem \ref{main result 1}, we show that the map $u(t)$ is a Schr\"{o}dinger map for data $u_0$ in $\mathfrak{H}^3$. At the same time, we obtain the Lipschitz continuity of $u(t)$ with respect to $u_0$ in $\mathfrak{H}^1$, which gives Theorem \ref{main result 2}.

There are two main obstacles in the above arguments. One is the a priori higher order energy estimates for approximate wave map equations, which guarantees the uniform lifespan $T>0$ for approximate solutions. In order to simplify the computation, the global system of coordinates related to the Iwasawa decomposition is used. Meanwhile the uniformly estimates follows from a bootstrap argument. The other obstacle lies in the establishment of the well-posedness for the coupled Schr\"odinger system with potentials. Indeed, the system is composed of two coupled mass-critical Schr\"{o}dinger equations with potentials. One of the equations admits Schr\"{o}dinger operator with positive potential, which has only purely absolutely continuous spectrum $[\frac{1}{4},\ \infty)$. The dispersive estimate for $t>1$ has been provided by \cite{BoMa}. So we only need to establish the similar estimate for $0<t<1$, namely
\begin{equation}                   \label{dispersive estimate introduction}
\left\lVert e^{it(\Delta_{\H^2}-V)}\right\rVert_{L^1\rightarrow L^{\infty}}\lesssim t^{-1},
\end{equation}
for nonnegative potential $V\in e^{-\alpha r}L^{\infty}(\H^2)$, $\alpha\geq 1$. We make use of the kernel of resolvent introduced by \cite{Ba} frequently. By Birman-Schwinger type resolvent expansion, the resolvent $R_V$ can be expressed as a series with respect to $R_0,\ R_V$ and $V$, then the Schr\"{o}dinger propagator in (\ref{dispersive estimate introduction}) can be written as a series. Since the dominant terms only depend on $R_0$ and $V$, we will use the pointwise bounds for free resolvent kernel and the Lemma \ref{important lemma}. For the remainder term, we use the meromorphic continuity of resolvent $R_V$ in Lemma \ref{R_V}. The other equation admits Schr\"{o}dinger operator with negative potential which  has at least a discrete spectrum $0$ even though it is extremely difficult to describe. Since we are dealing with the small data problem, the potential can be regarded as a perturbation term of the nonlinearity here.

The rest of the paper is organized as follows: In Section 2 we recall the hyperbolic spaces, function spaces, basic inequalities and the Fourier transformation. In Section 3 we use the approximating scheme to prove local well-posedness for Schr\"{o}dinger map (\ref{Schrodinger map}) in $\mathfrak{H}^3$, i.e Theorem \ref{main result 1}. In Sections 4 we introduce the Coulomb gauge, in which the Schr\"{o}dinger map can be written as two coupled Schr\"{o}dinger equations, i.e $(\psi^+,\ \psi^-)$-system. Conversely, if we have $\psi^+\in L^2$, we can reconstruct the Schr\"{o}dinger map $u$. In Sections 5 we provide the Strichartz estimates for operator $-\Delta_{\H^2}+V$, then we get the well-posedness of $(\psi^+,\ \psi^-)$-system for data $\psi^{\pm}_0\in L^2$. Finally,we finish the proof of Theorem \ref{main result 2}.

\section{Preliminaries}
In this section we review the geometry of hyperbolic space and the Fourier transformation.

\subsection{Hyperbolic spaces}
We consider the Minkowski space $\mathbb{R}^{d+1}$ for $d \geq 2$ with the Minkowski metric $(dx^1)^2+\cdots +(dx^d)^2-(dx^{d+1})^2$, and we can define the bilinear form on $\mathbb{R}^{d+1} \times \mathbb{R}^{d+1}$,
$$[x,y]=-x^1 y^1- \cdots -x^d y^d+ x^{d+1} y^{d+1}.$$
Then hyperbolic space $\mathbb{H}^d$ is defined as
$$\mathbb{H}^d=\{x \in \mathbb{R}^{d+1}: [x,x]=1 , x^{d+1}>0\},$$
and the Riemannian metric $g$ on $\mathbb{H}^d$ is induced by the Minkowski metric on $\mathbb{R}^{d+1}$. We take the point $\textbf{0}:=(0,\cdots,0,1) \in \R ^{d+1}$ as the origin in $\H ^d$.

We define $(\mathbb{G},\circ);=(SO(d,1),\circ)$ as the connected Lie group of $(d+1)\times (d+1)$ matrices that leave the bilinear form $[\cdot,\cdot]$ invariant. We have $X\in SO(d,1)$ if and only if
$$X^t \cdot I_{d,1} \cdot X=I_{d,1}, \mbox{   }  \det X=1, \mbox{    }X_{d+1,d+1}>0, $$
where $I_{d,1}$ is the diagonal matrix ${\rm diag}[1,\cdots,1,-1]$. Let $\mathbb{K}=SO(d)$ denote the subgroup of $SO(d,1)$ that fix the origin $\textbf{0}$. Indeed, $\mathbb{K}$ is a compact subgroup of rotations acting on the variables $(x^1,\cdots,x^d)$. We can thus identify $\mathbb{H}^d$ with the symmetric space $\mathbb{G}/\mathbb{K}$. For every $h\in \mathbb{G}$ we can define the map $$L_{h}: \H^d \longrightarrow \H^d, \mbox{ }L_h (x)=h\cdot x.$$
A function $f:\H^d \rightarrow \R$ is called $\mathbb{K}$-invariant or radial, if for all $k\in \mathbb{K}$ and for all $x\in \H^d$ we have $$f(k\cdot x)=f(x).$$
Then we have the Cartan decomposition of $h\in \mathbb{G}$, namely
\begin{align*}
h=k \circ a_r \circ \tilde{k},\ \ \ a_r\in \mathbb{A}_+,\ k,\tilde{k}\in \mathbb{K},
\end{align*}
where
\begin{align*}
a_r:=\left( \begin{matrix}
I_{d-1\times d-1}& 0& 0\\
0& \cosh r& \sinh r \\
0 & \sinh r & \cosh r \\
\end{matrix}
\right), \ \ \mathbb{A}_+:=\{ a_r:r\in [0,\infty) \}.
\end{align*}

We introduce two convenient global systems of coordinates on $\H^d$. One of the systems is geodesic polar coordinates:
\begin{eqnarray}           \label{global coordinate 1}
\phi :\R_+ \times \S^{d-1} \rightarrow \H^d \subset \R^{d+1} ,  \ \ \ \phi(r,\omega) =(\sinh r\cdot \omega,\  \cosh r).
\end{eqnarray}
For $d=2$, $\phi$ can be written explicitly as
\begin{eqnarray*}
\phi :\R_+ \times [0,2\pi) \rightarrow \H^d \subset \R^{d+1} ,  \ \ \ \phi(r,\theta) =(\sinh r \cos \theta,\  \sinh r \sin \theta,\  \cosh r),
\end{eqnarray*}
in these coordinates, the hyperbolic metric $g$ is given by $g=dr^2+ \sinh^2 r d\theta^2$, the volume element $\mu(dx)$ on $\H^2$ is given by $\sinh r dr d\theta$ and the Laplace-Beltrami operator is given by
\begin{eqnarray*}
\Delta_{\H^2}=\partial_r^2+\coth r \partial_r+\frac{\partial_{\theta}^2}{\sinh ^2 r}.
\end{eqnarray*}
The other global system of coordinates is defined as follows \cite{IoSta}:
\begin{eqnarray}          \label{global coordinate 2}
\varphi: \R^{d-1} \times \R \rightarrow \H^d,\ \ \varphi(v,s)=(\ \sinh s+e^{-s}\frac{|v|^2}{2},e^{-s}v_1,\ \cdots \ ,e^{-s}v_{d-1},\ \cosh s+e^{-s}\frac{|v|^2}{2}),
\end{eqnarray}
using these coordinates we have the induced metric
\begin{eqnarray*}
g=e^{-s}[dv_1^2+\cdots + dv_{d-1}^2+e^{2s}ds^2].
\end{eqnarray*}
If we fix the global orthonormal frame
\begin{eqnarray*}
e_{\alpha}=e^s \partial_{v_{\alpha}}\ \ {\rm for}\ \alpha =1,\cdots, d-1,\ {\rm and}\ \ e_d=\partial_s,
\end{eqnarray*}
we compute the commutators
\begin{eqnarray*}
[e_d,e_{\alpha}]=e_{\alpha},[e_{\alpha},e_{\beta}]=[e_d,e_d]=0\ \ {\rm for\ any}\ \alpha,\beta=1,\cdots,d-1,
\end{eqnarray*}
and the covariant derivatives
\begin{eqnarray*}
\nabla_{e_{\alpha}}e_{\beta}=\delta_{\alpha \beta}e_d,\ \nabla_{e_{\alpha}}e_d=-e_{\alpha},\ \nabla_{e_d}e_{\alpha}=\nabla_{e_d}e_d=0,\ \ {\rm for}\ \alpha,\beta=1,\cdots,d-1.
\end{eqnarray*}

\subsection{Function spaces and basic inequalities}
Here we define some relevant function spaces on $\H^2$ and recall some basic inequalities. For smooth function $f:\H^2\rightarrow\R$, the $L^p(\H^2)$-norm for $1\leq p\leq\infty$ are defined by
\begin{equation*}
\left\lVert f\right\rVert_{L^p(\H^2)}:=(\int_{\H^2}|f(x)|^p\ {\rm dvol}_g)^{1/p},\ \ \ \left\lVert f\right\rVert_{L^{\infty}(\H^2)}:=\sup\limits_{x\in\H^2}|f(x)|.
\end{equation*}
Also we can define the Sobolev norm $H^k(\H^2;\R)$ of $f$, namely
\begin{equation*}
\left\lVert f\right\rVert_{H^k}^2:=\sum\limits_{l\leq k}\int_{\H^2}|\nabla^l f|^2\ {\rm dvol}_g,
\end{equation*}
where $\nabla^l f$ is the $l$-th covariant derivative of $f$.
By \cite{LaOhShaCauchy}, we have
\begin{equation}
\left\lVert f\right\rVert_{H^{2l}}\simeq \left\lVert (-\Delta)^l f\right\rVert_{L^2},\ \ \ {\rm for}\ l=0,1,2,\cdots,
\end{equation}
and
\begin{equation}              \label{equivalent of Hl of u}
\left\lVert f\right\rVert_{H^{2l+1}}\simeq  \left\lVert\nabla^{2l+1} f\right\rVert_{L^2}\simeq \left\lVert \nabla(-\Delta)^l f\right\rVert_{L^2},\ \ \ {\rm for}\ l=0,1,2,\cdots.
\end{equation}
We will often use these equivalent definitions.

As a $\R^3$-valued function, we can define the extrinsic Sobolev spaces $\mathcal{H}^k(\H^2;\R^3)$. We say that $u$ has finite $\mathcal{H}^k$-norm with respect to $u(\infty):=\lim\limits_{r\rightarrow\infty}u(r)$ if
\begin{equation*}
u\in\{ u:\H^2\rightarrow\S^2\subset\R^3\big| u_i-u_i(\infty)\in H^k,\ \ {\rm for}\ i=1,\ 2,\ 3\}.
\end{equation*}
Denote
\begin{equation*}
\left\lVert u \right\rVert_{\mathcal{H}^k(\H^2;\R^3)}:=\sum\limits_{i=1}^3\left\lVert  u_i-u_i(\infty) \right\rVert_{H^k}
\end{equation*}
In the polar coordinate (\ref{global coordinate 1}), the equivariant maps are easily reduced to maps of a single variable $r$. For smooth radial function $f$, we define a natural space $\dot{H}^1_e$ by
\begin{equation*}
\left\lVert f\right\rVert_{\dot{H}^1_e}:=\left\lVert \partial_r f\right\rVert_{L^2}+\left\lVert \frac{f}{\sinh r}\right\rVert_{L^2},
\end{equation*}
then for such $f$, we have Sobolev embedding
\begin{equation}                    \label{Sobolev embedding H1e}
\left\lVert f\right\rVert_{L^{\infty}} \lesssim \left\lVert f\right\rVert_{\dot{H}^1_e}.
\end{equation}

We now recall the Sobolev inequalities (see \cite{LiZhao}, \cite{LaOhShaCauchy}).
\begin{lemma}                  \label{Sobolev inequality}
Let $f\in C_c^{\infty}$, then for $1<p<\infty$, $p\leq q\leq \infty$, $0<\theta<1$, $1<r<2$, $r\leq l<\infty$, the following inequalities hold:
\begin{align}                  \label{Sobolev inequality 1}
\left\lVert f\right\rVert_{L^2}\lesssim & \left\lVert\nabla f\right\rVert_{L^2},\\         \label{Sobolev inequality 2}
\left\lVert f\right\rVert_{L^q}\lesssim & \left\lVert f\right\rVert_{L^p}^{1-\theta}\left\lVert \nabla f\right\rVert_{L^2}^{\theta},\ \ \ {\rm for}\ \ \frac{1}{q}=\frac{1}{p}-\frac{\theta}{2},\\          \label{Sobolev inequality 3}
\left\lVert f\right\rVert_{L^l}\lesssim &\left\lVert \nabla f\right\rVert_{L^r},\ \ \ \ \ \ \ \ \ \ \ \ \ \ \ {\rm for}\ \ \frac{1}{l}=\frac{1}{r}-\frac{1}{2},\\           \label{Sobolev inequality 4}
\left\lVert f\right\rVert_{L^{\infty}}\lesssim &\left\lVert (-\Delta)^{\frac{\alpha}{2}} f\right\rVert_{L^2},\ \ \ \ \ {\rm for}\ \ \alpha>1,\\                       \label{Sobolev inequality 5}
\left\lVert \nabla f\right\rVert_{L^2}\sim &\left\lVert (-\Delta)^{\frac{1}{2}}f\right\rVert_{L^2}.
\end{align}
\end{lemma}

We also recall the diamagnetic inequality (see \cite{LiZhao}, \cite{LaOhShaCauchy}).
\begin{lemma}                     \label{diamagnetic inequality lemma}
If $T$ is some $(r,\ s)$-type tension or tension matrix defined on $\H^2$, then in the distribution sense, one has
\begin{equation}                 \label{diamagnetic inequality}
\big|\nabla|T|\big| \leq |\nabla T|.
\end{equation}
\end{lemma}

\begin{lemma}                   \label{equivalent of Sobolev lemma}
Let $u:\H^2\rightarrow\S^2$ be smooth map with $u(\infty)=\lim\limits_{r\rightarrow\infty}u(r)$, then
\begin{equation}
\left\lVert u\right\rVert_{\mathfrak{H}^3}\sim \left\lVert u\right\rVert_{\mathcal{H}^3}
\end{equation}
in the sense that there exist polynomials $P$ and $Q$ such that
\begin{equation}              \label{equivalent of Sobolev}
\left\lVert u\right\rVert_{\mathfrak{H}^3}\leq P(\left\lVert u\right\rVert_{\mathcal{H}^3}),\ \ \ \left\lVert u\right\rVert_{\mathcal{H}^3}\leq Q(\left\lVert u\right\rVert_{\mathfrak{H}^3}).
\end{equation}
\end{lemma}

\begin{proof}
In order to prove (\ref{equivalent of Sobolev}), we use the polar coordinates (\ref{global coordinate 1}). For $k=1$, we have
\begin{equation*}
|du|^2=\<\partial_r u,\partial_r u\>+\frac{1}{\sinh^2 r}\<\partial_{\theta}u,\partial_{\theta}u\>=\sum\limits_{i=1}^{3}|du_i|^2,
\end{equation*}
hence, $\left\lVert u\right\rVert_{\mathfrak{H}^1}\leq \left\lVert u\right\rVert_{\mathcal{H}^1}$. Conversely, by (\ref{Sobolev inequality 1}), we obtain
\begin{align*}
\left\lVert u\right\rVert_{\mathcal{H}^1}\lesssim &\left\lVert u\right\rVert_{\mathfrak{H}^1}+\sum\limits_{i=1}^3\left\lVert u_i-u_i(\infty)\right\rVert_{L^2},\\
\lesssim & \left\lVert u\right\rVert_{\mathfrak{H}^1}.
\end{align*}
For $k=2$, we have
\begin{equation*}
|\nabla du|^2=g^{ii}g^{jj}\big[ |\partial_{ij}u-\Gamma_{ij}^k \partial_k u|^2-(\partial_{ij}u\cdot u)^2 \big],
\end{equation*}
and
\begin{equation}      \label{nabla2 ui}
\sum\limits_{i=1}^3 |\nabla^2 u_i|^2=g^{ii}g^{jj} |\partial_{ij}u-\Gamma_{ij}^k \partial_k u|^2.
\end{equation}
Therefore, $\left\lVert u\right\rVert_{\mathfrak{H}^2}\lesssim \left\lVert u\right\rVert_{\mathcal{H}^2}$ immediately. Conversely, $\partial_j u\cdot u=0$ implies
\begin{align*}
g^{ii}g^{jj}(\partial_{ij}u\cdot u)^2=g^{ii}g^{jj}(\partial_i u \partial_j u)^2=|du|^4,
\end{align*}
then, by (\ref{Sobolev inequality 2}) and (\ref{diamagnetic inequality}), we have
\begin{align*}
\left\lVert u\right\rVert_{\mathcal{H}^2}^2 \lesssim & \left\lVert u\right\rVert_{\mathfrak{H}^2}^2+\left\lVert du\right\rVert_{L^4}^4,\\
\lesssim & \left\lVert u\right\rVert_{\mathfrak{H}^2}^2+\left\lVert \nabla|du|\right\rVert_{L^2}^2\left\lVert du\right\rVert_{L^2}^2,\\
\lesssim & \left\lVert u\right\rVert_{\mathfrak{H}^2}^2+\left\lVert \nabla du\right\rVert_{L^2}^2\left\lVert du\right\rVert_{L^2}^2,\\
\lesssim & \left\lVert u\right\rVert_{\mathfrak{H}^2}^2+\left\lVert u\right\rVert_{\mathfrak{H}^2}^4.
\end{align*}
For $k=3$, we have
\begin{align*}
|\nabla^2 du|^2= & g^{ii}g^{jj}g^{kk}\big|\partial_{ijk}u-\partial_i (\Gamma_{jk}^l \partial_l u)-\Gamma_{ij}^l \partial_{lk}u +\Gamma_{ij}^l\Gamma_{lk}^p\partial_p u\\
&-\Gamma_{ik}^l \partial_{jl}u+\Gamma_{ik}^l\Gamma_{jl}^p\partial_p u +3(\partial_{ij}u\partial_k u -\Gamma_{ij}^l \partial_l u\partial_k u)u  \big|^2,
\end{align*}
and
\begin{align*}
|\nabla^3 u_q|^2=& g^{ii}g^{jj}g^{kk}\big| \partial_{ijk}u_q-\partial_i(\Gamma_{jk}^l \partial_l u_q)-\Gamma_{ij}^l\partial_{lk}u_p\\
&+\Gamma_{ij}^l\Gamma_{lk}^p\partial_p u_q-\Gamma_{ik}^l \partial_{jl}u_q+\Gamma_{ik}^l \Gamma_{jl}^p\partial_pu_q   \big|^2.
\end{align*}
By (\ref{nabla2 ui}), we have
\begin{align*}
 g^{ii}g^{jj}g^{kk}\big|(\partial_{ij}u\partial_k u -\Gamma_{ij}^l \partial_l u\partial_k u)u \big|^2 \lesssim  \sum\limits_{i=1}^3|\nabla^2 u_i|^2|du|^2.
\end{align*}
Then
\begin{align*}
\left\lVert u\right\rVert_{\mathfrak{H}^3}^2\lesssim & \sum\limits_{i=1}^3  \left\lVert \nabla^3 u_i\right\rVert_{L^2}^2+ \sum\limits_{i=1}^3 \left\lVert \nabla^2 u_i\right\rVert_{L^4}^2 \left\lVert du\right\rVert_{L^4}^2,\\
\lesssim &\left\lVert u\right\rVert_{\mathcal{H}^3}^2+ \sum\limits_{i=1}^3 \left\lVert \nabla^3 u_i\right\rVert_{L^2}\left\lVert \nabla^2 u_i\right\rVert_{L^2} \left\lVert \nabla du\right\rVert_{L^2}\left\lVert du\right\rVert_{L^2},\\
\lesssim & \left\lVert u\right\rVert_{\mathcal{H}^3}^2+ \left\lVert u\right\rVert_{\mathcal{H}^3}^4.
\end{align*}
Conversely, by (\ref{Sobolev inequality 2}) and (\ref{diamagnetic inequality}), we have
\begin{align*}
\left\lVert u\right\rVert_{\mathcal{H}^3}^2 \lesssim &\left\lVert u\right\rVert_{\mathfrak{H}^3}^2+ \sum\limits_{i=1}^3 \left\lVert |\nabla^2 u_i| |du|\right\rVert_{L^2}^2,\\
\lesssim & \left\lVert u\right\rVert_{\mathfrak{H}^3}^2+ \left\lVert |\nabla  du| |du|+|du|^3\right\rVert_{L^2}^2,\\
\lesssim & \left\lVert u\right\rVert_{\mathfrak{H}^3}^2+ \left\lVert \nabla  du \right\rVert_{L^4}^2\left\lVert du \right\rVert_{L^4}^2 +\left\lVert du \right\rVert_{L^6}^6,\\
\lesssim & \left\lVert u\right\rVert_{\mathfrak{H}^3}^2+ \left\lVert \nabla^2 du\right\rVert_{L^2}\left\lVert \nabla du\right\rVert_{L^2}^2\left\lVert du\right\rVert_{L^2}+\left\lVert \nabla du\right\rVert_{L^2}^4 \left\lVert du\right\rVert_{L^2}^2,\\
\lesssim & \left\lVert u\right\rVert_{\mathfrak{H}^3}^2+\left\lVert u\right\rVert_{\mathfrak{H}^3}^4+\left\lVert u\right\rVert_{\mathfrak{H}^3}^6.
\end{align*}
Therefore, (\ref{equivalent of Sobolev}) are obtained.
\end{proof}

Finally, we state the following estimates, which are often used for radial functions and obtained by Schur's test easily.
\begin{lemma}             \label{basic inequality}
Let $f\in L^p$ be radial function, we have
\begin{align}           \label{basic inequality1}
\left \lVert \frac{1}{\sinh^2 r}\int _0^r \sinh s f(s) ds \right \lVert _{L^p}\lesssim  \left \lVert f \right \lVert _{L^p}, \ \ \ 1<p \leq \infty,\\           \label{basic inequality2}
\left \lVert \frac{\cosh r}{\sinh^2 r}\int _0^r \sinh s f(s) ds \right \lVert _{L^p}\lesssim \left \lVert f \right \lVert _{L^p},\ \ \ 1<p\leq \infty,\\             \label{basic inequality3}
\left \lVert\int_r^{\infty}  \frac{\cosh s}{\sinh s}f(s)\right \lVert_{L^p}\lesssim \left \lVert f \right \lVert_{L^p},\ \ \ 1 \leq p<\infty,\\             \label{basic inequality4}
\left \lVert\int_r^{\infty}  \frac{1}{\sinh s}f(s)\right \lVert_{L^p}\lesssim \left \lVert f \right \lVert_{L^p},\ \ \ 1 \leq p<\infty,\\              \label{basic inequality5}
 \left \lVert \frac{1}{\sinh r} \int_r^{\infty} e ^{\int _{\rho}^r \sinh^{-1} s ds } f(\rho) d\rho \right \lVert_{L^p}\lesssim \left \lVert f \right \lVert_{L^p},\ \ \ 1 \leq p<\infty, \\         \label{basic inequality6}
\left \lVert \frac{1}{r^2}\int _0^r  f(s) sds \right \lVert _{L^p(\R^2)}\lesssim \left \lVert f \right \lVert _{L^p(\R^2)},\ \ \ 1<p\leq \infty,           
\end{align}
\end{lemma}

\subsection{Fourier transformation}
For $\omega \in \mathbb{S}^{d-1}$ and $\lambda$ a real number, the functions of the type
$$h_{\lambda ,\omega}:\H^d \longrightarrow \C,\mbox{  }x\rightarrow [x,b(\omega)]^{i \lambda-\frac{n-1}{2}},$$
are generalized eigenfunctions of the Laplacian-Beltrami operator. Indeed, we have
$$-\Delta _{\H^d}h_{\lambda,\omega}=(\lambda ^2+\frac{(n-1)^2}{4})h_{\lambda, \omega}.$$
Then we can define the Fourier transformation analogous to the Euclidean case. For $f\in C_c^{\infty} (\H^d)$,
\begin{equation*}
\widehat{f}(\lambda,\omega)=\int_{\H^d}f(x)h_{\lambda,\omega}(x){\rm dvol}_g,
\end{equation*}
and one has the Fourier inversion formula for function on $\H^d$
\begin{equation*}
f(x)=\int^{+\infty}_{-\infty} \int_{\mathbb{S}^{d-1}} \overline{h_{\lambda,\omega}(x)} \widehat{f}(\lambda,\omega) \frac{d \lambda d\omega}{|c(\lambda)|^2},
\end{equation*}
where $c(\lambda)$ is the Harish-Chandra coefficient,
$$c(\lambda)=C\frac{\Gamma (i\lambda)}{\Gamma(\frac{d-1}{2}+i\lambda)}.$$

For the linear Schrodinger equation on $\H^d$,
\begin{equation*}
\left\{
\begin{aligned}
& i\partial_t u+\Delta _{\H^d}u = 0,\\
& u(0) = u_0\in L^2(\H^d).
\end{aligned}
\right.
\end{equation*}
the solution can be written explicitly see \cite{Ba} as
$$u(t,x)=C\exp ^{-it(\frac{d-1}{2})^2}\int_{\H^d} u_0(y)K^d(t,d(x,y))dy,$$
where the kernel $K^d$ is, for $\rho >0$ and $d \geq 3$ odd
$$K^d(t,\rho)=\int^{+\infty}_{-\infty} e ^{-it \lambda^2} (\frac{\partial _{\rho}}{\sinh \rho})^{\frac{d-1}{2}} \cos (\lambda \rho) d \lambda.$$
and for $d\geq 2$ even,
$$K^d(t,\rho)=\int^{+\infty}_{-\infty} e ^{-it \lambda^2} \int ^{+\infty}_{\rho} \frac{\sinh s}{\sqrt{\cosh s-\cosh \rho}}(\frac{\partial _s}{\sinh s})^{\frac{d}{2}} \cos (\lambda s) ds d \lambda.$$
In particular, $d=2$,
$$K^2(t,\rho)=\frac{c}{|t|^\frac{3}{2}} \int^{+\infty}_{\rho} \frac{e ^{i \frac{s^2}{4t}} s}{\sqrt{\cosh s -\cosh \rho}}ds.$$

\section{Local well-posedness for Schr\"{o}dinger maps}
In order to prove the local well-posedness in $\mathfrak{H}^3$, we apply the approximating Scheme introduced by McGahagan \cite{Mc}. For any $\delta>0$, we introduce the wave map model equation:
\begin{equation}
\left\{
\begin{aligned}          \label{approximate equation}
&\delta^2 \nabla_t \partial_t u-J \partial_t u-\tau (u)=0,\\
&u(0,x)=u_0,\ \partial_t u(0,x)=g_0^{\delta}.
\end{aligned}
\right.
\end{equation}
where $u(t,x):[0,T]\times \H^2 \rightarrow \S^2$ and $g_0^{\delta}\in T_{u_0(x)}\S^2$. In this section we use the global coordinates (\ref{global coordinate 2}), denote $\nabla_i=\nabla_{e_i} $ for $i=1,\ 2$. For simplicity, denote $u:=u^{\delta}$. 

Before proving the Theorem \ref{main result 1}, we need the following lemma.

\begin{lemma}           \label{LWP lemma}
For $\tilde{T}>0$, there exists a constant $C>0$ independent of $\delta$, such that for any $u^{\delta}:[0,\tilde{T}]\times \H^2 \rightarrow \S^2$, $u^{\delta}\in C([0,T];\mathfrak{H}^3)$, a solution of the approximate equation, and any $1\leq k \leq 2$, the following estimate holds for $u^{\delta}$:
\begin{equation*}
\left \lVert \nabla^{k-1} \partial_t u^{\delta} \right \rVert_{C([0,T];L^2)} \leq  C \left \lVert  \partial_t u^{\delta}(0) \right \rVert_{\mathfrak{H}^{k-1}}+C \left \lVert u^{\delta} \right \rVert_{C([0,T];\mathfrak{H}^{k+1})}
\end{equation*}
for some $T>0$, depending only on the size of the solution $\left \lVert u^{\delta} \right \rVert_{C([0,\tilde{T}];\mathfrak{H}^{k+1})}$ and on the size of the initial data $ \left \lVert  \partial_t u^{\delta}(0) \right \rVert_{\mathfrak{H}^1}$.
\end{lemma}

\begin{proof}
For $k=1$, we take the inner product of the above wave map equation with $J(u)\nabla_t \partial_t u$, the first term will disappear by orthogonality, we get
\begin{eqnarray}                         \label{partial t u 1}                  
\int_{\H^2} \< \partial_t u, \nabla_t \partial_t u \> & =& \int_{\H^2} \< J\tau (u), \nabla_t \partial_t u \> , \\      \label{partial t u 2}
&=& \frac{d}{dt} \int_{\H^2} \< J\tau (u),\partial_t u \> -\int_{\H^2} \< J\nabla_t \tau(u),\partial_t u \>
\end{eqnarray}
In the system of coordinate, $\tau (u)$ can be written as $\tau(u)=\nabla_i e_i (u)-(\nabla_i e_i)(u)$, then commute $\nabla_i$ and $\nabla_t$, by integration by parts, the second term of (\ref{partial t u 2}) becomes
\begin{eqnarray*}
\int_{\H^2} \< J\nabla_t \tau(u),\partial_t u \> &=&  \int_{\H^2} \< J\nabla_t(\nabla_i e_i (u)-(\nabla_i e_i)(u)), \partial_t u \> , \\
&=& \int_{\H^2} \< J[\nabla_t,\nabla_i  ]e_i u +J\nabla_i \nabla_t e_i u -J\nabla_t e_2 u , \partial_t u  \> ,\\
&=& \int_{\H^2} \<  J[\nabla_t,\nabla_i  ]e_i u, \partial_t u \>  +   \int_{\H^2} e_i \< J\nabla_i \partial_t u ,\partial_t u \>  \\
&& - \int_{\H^2} \< J\nabla_i  \partial_t u, \nabla_i  \partial_t u \> - \int_{\H^2} \< J\nabla_2 \partial_t u, \partial_t u \>,\\
&= &  \int_{\H^2} \<  J[\nabla_t,\nabla_i  ]e_i u, \partial_t u \>  + \int_{\H^2}  \< J\nabla_2 \partial_t u ,\partial_t u \>  - \int_{\H^2} \< J\nabla_2  \partial_t u, \partial_t u \>,\\
&=& \int_{\H^2} \<  J[\nabla_t,\nabla_i  ]e_i u, \partial_t u \> .
\end{eqnarray*}
If we integrate in time,by H\"{o}lder inequality we find that (\ref{partial t u 1}) becomes
\begin{eqnarray*}
\frac{1}{2} \left \lVert \partial_t u (t)\right \rVert_{L^2}^2 &=& \frac{1}{2} \left \lVert \partial_t u (0)\right \rVert_{L^2}^2+ \int_{\H^2} \< J\tau (u),\partial_t u  \> (t)- \int_{\H^2} \< J\tau (u),\partial_t u  \> (0) \\
& & -\int_0^t \int_{\H^2} \< J[\nabla_t, \nabla_i ]e_i u, \partial_t u \>,\\
& \leq & \frac{1}{2} \left \lVert \partial_t u (0)\right \rVert_{L^2}^2+  \left \lVert \tau (u)(t)\right \rVert_{L^2} \left \lVert \partial_t u (t)\right \rVert_{L^2}  + \left \lVert \tau (u)(0)\right \rVert_{L^2} \left \lVert \partial_t u (0)\right \rVert_{L^2}   \\
&&  \int_0^t  \left \lVert\partial_t u \right \rVert_{L^2}^2 \left \lVert d( u)\right \rVert_{L^{\infty}}^2ds,\\
  & \leq & C\left \lVert \partial_t u (0)\right \rVert_{L^2}^2 +C \left \lVert \tau( u )(t)\right \rVert_{L^2}^2 +C \left \lVert \tau( u )(0)\right \rVert_{L^2}^2 \\
  && + \int_0^t  \left \lVert\partial_t u \right \rVert_{L^2}^2 \left \lVert d( u)\right \rVert_{L^{\infty}}^2ds,
\end{eqnarray*}
then
\begin{eqnarray*}
\left \lVert \partial_t u (t)\right \rVert_{L^2}^2 \lesssim \left \lVert \partial_t u (0)\right \rVert_{L^2}^2+ \left \lVert \tau( u )\right \rVert_{C([0,T];L^2)}^2 +\int_0^t  \left \lVert \partial_t u \right \rVert_{L^2}^2 \left \lVert \partial u\right \rVert_{C([0,T];H^2)}^2ds.
\end{eqnarray*}
Therefore, by Gronwall inequality, choose $T$ such that $\left\lVert du\right\rVert_{C(0,\tilde{T};H^2)}T$ small, we have
\begin{eqnarray}               \label{estimate of partial t u}
\left \lVert \partial_t u (t)\right \rVert_{L^2}^2 \lesssim (\left \lVert \partial_t u (0)\right \rVert_{L^2}^2+ \left \lVert \tau( u )\right \rVert_{C([0,T];L^2)}^2) (1+\left \lVert \partial u\right \rVert_{C([0,T];H^2)}^2).
\end{eqnarray}
For $k=2$, we take $\nabla_i $ on the approximate equation (\ref{approximate equation}):
\begin{equation*}
\delta^2 \nabla_i  \nabla_t \partial_t u- J\nabla_i  \partial_t u -\nabla_i \tau (u)=0,
\end{equation*}
then we take the inner product of the above equation with $J(u)\nabla_i  \nabla_t \partial_t u$ and commute $\nabla_i$ and $\nabla_t$, we have
\begin{equation}                   \label{k=2 inner product}
\begin{aligned}
0= & \int_{\H^2} \< J\nabla_i  \partial_t u,J\nabla_i \nabla_t \partial_t u \> {\rm dvol}_g- \int_{\H^2} \< J \nabla_i  \tau, \nabla_i \nabla_t \partial_t u \>{\rm dvol}_g,\\
= & \frac{1}{2}\frac{d}{dt}\left\lVert \nabla_i\partial_t u\right\rVert_{L^2}^2+\int_{\H^2}\< \nabla_i\partial_t u,[\nabla_i,\nabla_t]\partial_t u \>{\rm dvol}_g-\frac{d}{dt}\int_{\H^2}\<J\nabla_i \tau,\nabla_i\partial_t u\>{\rm dvol}_g\\
& +\int_{\H^2}\< J\nabla_t\nabla_i \tau,\nabla_i \partial_t u \>{\rm dvol}_g-\int_{\H^2}\<J\nabla_i \tau,[\nabla_i,\nabla_t]\partial_t u\>{\rm dvol}_g.
\end{aligned}
\end{equation}
Denote
\begin{equation}           \label{denote I II III}
\begin{aligned}
&{\rm I}=\int_{\H^2}\< \nabla_i\partial_t u,[\nabla_i,\nabla_t]\partial_t u \>{\rm dvol}_g,\ {\rm II}=\int_{\H^2}\< J\nabla_t\nabla_i \tau,\nabla_i \partial_t u \>{\rm dvol}_g,\\
&{\rm III}=\int_{\H^2}\<J\nabla_i \tau,[\nabla_i,\nabla_t]\partial_t u\>{\rm dvol}_g.
\end{aligned}
\end{equation}
Then ${\rm II}$ can be rewritten as
\begin{align*}
{\rm II}=&\int \<J[\nabla_t,\nabla_i]\tau, \nabla_i\partial_t u\>+e_i\< J\nabla_t\tau,\nabla_i\partial_t u \>-\<J\nabla_t\tau,\nabla_i\nabla_i\partial_t u\>{\rm dvol}_g,\\
\triangleq & {\rm II}_1+{\rm II}_2+{\rm II}_3,
\end{align*}
by the representation of $\tau(u)$ and $(\nabla_j e_j)u=e_2 u$, ${\rm II}_2$ becomes
\begin{align*}
{\rm II}_2=& -\int_{\H^2} e_2 \< \nabla_t \tau ,J\nabla_i \partial_t u \>{\rm dvol}_g, \\
=& -\int_{\H^2} \< \nabla_t (\nabla_{e_j}e_j u-(\nabla_{e_j}e_j)u ),J\nabla_t e_2 u \>{\rm dvol}_g,\\
=& -\int_{\H^2} \< \nabla_j\nabla_t e_j u,J\nabla_t e_2 u \>+ \< [\nabla_t, \nabla_j]e_j u,J\nabla_t e_2 u \> {\rm dvol}_g,
\end{align*}
for ${\rm II}_3$, by integration by parts, we have
\begin{align*}
{\rm II}_3= &- \int_{\H^2} \< J[\nabla_t,\nabla_j]e_j u,\nabla_i\nabla_i \partial_t u \>+ \< J\nabla_j\nabla_j \partial_t u, \nabla_i\nabla_i \partial_t u \> -\< J\nabla_t e_2 u,\nabla_i\nabla_i \partial_t u \>{\rm dvol}_g,\\
=&- \int_{\H^2} e_i\< J[\nabla_t, \nabla_j]e_j u,\nabla_i \partial_t u \> -\< J\nabla_i[\nabla_t,\nabla_j]e_j u,\nabla_i \partial_t u \>-\< J\nabla_t e_2 u,\nabla_i\nabla_i \partial_t u \>{\rm dvol}_g,\\
=& \int_{\H^2}\< [\nabla_t, \nabla_j]e_j u ,J \nabla_t e_2 u \> +\< J\nabla_i[\nabla_t,\nabla_j]e_j u,\nabla_i \partial_t u \>
+\< J\nabla_t e_2 u,\nabla_i\nabla_i \partial_t u \>{\rm dvol}_g.
\end{align*}
Hence,
\begin{equation*}
{\rm II}=\int_{\H^2} \<J[\nabla_t,\nabla_i] \tau, \nabla_i\partial_t u \> +\< J\nabla_i[\nabla_t ,\nabla_j]e_j u,\nabla_i \partial_t u \>{\rm dvol}_g.
\end{equation*}
Then (\ref{k=2 inner product}) can be written as
\begin{equation*}
\frac{1}{2}\frac{d}{dt} \left \lVert \nabla\partial_t u \right \rVert_{L^2}^2=\frac{d}{dt}\int_{\H^2}\< J\nabla_i \tau, \nabla_i \partial_t u \> -{\rm I}-{\rm II}+{\rm III}.
\end{equation*}
Integrating in time, by H\"{o}lder inequality, it gives
\begin{align*}
\frac{1}{2}\left \lVert \nabla \partial_t u(t) \right \rVert_{L^2}^2  \leq & \frac{1}{2}\left \lVert \nabla\partial_t u(0) \right \rVert_{L^2}^2+\left \lVert \nabla\tau (t) \right \rVert_{L^2}\left \lVert \nabla\partial_t u(t) \right \rVert_{L^2}^2\\
 &  +\left \lVert \nabla\tau (0) \right \rVert_{L^2}\left \lVert \nabla\partial_t u(0) \right \rVert_{L^2}^2+\int_0^t -{\rm I}-{\rm II}+{\rm III}\ ds,
\end{align*}
then
\begin{align*}
\left \lVert \nabla \partial_t u(t) \right \rVert_{L^2}^2  \lesssim \left \lVert \nabla\partial_t u(0) \right \rVert_{L^2}^2+\left \lVert \nabla\tau (t) \right \rVert_{L^2}^2+\left \lVert \nabla\tau (0) \right \rVert_{L^2}^2+\int_0^t |-{\rm I}-{\rm II}+{\rm III}| ds.
\end{align*}
From (\ref{denote I II III}), H\"{o}lder inequality and Lemma \ref{Sobolev inequality}, we have
\begin{align*}
\big|-{\rm I}-{\rm II}+{\rm III}\big|  \leq  & \left \lVert \nabla \partial_t u \right \rVert_{L^2} \Big( \left \lVert [\nabla_t, \nabla_i] \tau \right \rVert_{L^2} +\left \lVert \nabla_i [\nabla_t,\nabla_j]e_j u \right \rVert_{L^2} \Big) \\
 & + \Big ( \left \lVert \partial u \right \rVert_{H^2} +  \left \lVert \nabla \partial_t u \right \rVert_{L^2} \Big) \left \lVert [\nabla,\nabla_t] \partial_t u \right \rVert_{L^2},\\
\leq & \left \lVert \nabla \partial_t u \right \rVert_{L^2} \Big( \left \lVert  \partial_t u \right \rVert_{L^4} \left \lVert \tau(u) \right \rVert_{L^4} \left \lVert du \right \rVert_{L^{\infty}} +  \left \lVert du \right \rVert_{L^{\infty}}^3  \left \lVert \partial_t u \right \rVert_{L^2}  \left \lVert du \right \rVert_{L^{\infty}}^2  \left \lVert \nabla\partial_t u \right \rVert_{L^2}\\
 & + \left \lVert \partial_t u \right \rVert_{L^4} \left \lVert du \right \rVert_{L^{\infty}}  \left \lVert \nabla^2 u \right \rVert_{L^4} \Big) + \Big ( \left \lVert \partial u \right \rVert_{H^2} +  \left \lVert \nabla \partial_t u \right \rVert_{L^2} \Big) \left \lVert \partial_t u \right \rVert_{L^4}^2 \left \lVert \partial u \right \rVert_{H^2},\\
\leq & \left \lVert \nabla \partial_t u \right \rVert_{L^2}   (\left \lVert \nabla \partial_t u \right \rVert_{L^2} +\left \lVert  \partial_t u \right \rVert_{L^2})( \left \lVert  \partial u \right \rVert_{H^2}^2 +\left \lVert  \partial u \right \rVert_{H^2}^3 )  + \left \lVert \nabla \partial_t u \right \rVert_{L^2}^2 \left \lVert  \partial_t u \right \rVert_{L^2}\left \lVert \partial u \right \rVert_{H^2},\\
\leq & \left \lVert \nabla \partial_t u \right \rVert_{L^2}^2  ( \left \lVert \partial_t u \right \rVert_{L^2} \left \lVert  \partial u \right \rVert_{H^2}+ \left \lVert  \partial u \right \rVert_{H^2}^2 +\left \lVert  \partial u \right \rVert_{H^2}^3 ) + \left \lVert \partial_t u \right \rVert_{L^2}^2 (\left \lVert  \partial u \right \rVert_{H^2}^2 +\left \lVert  \partial u \right \rVert_{H^2}^3).
\end{align*}
By (\ref{estimate of partial t u}),
\begin{eqnarray*}
\big|-{\rm I}-{\rm II}+{\rm III}\big| &\leq & \left \lVert \nabla \partial_t u \right \rVert_{L^2}^2 (\left \lVert  \partial_t u(0) \right \rVert_{L^2}+\left \lVert \partial u \right \rVert_{C(H^2)}  )(\left \lVert \partial u \right \rVert_{C(H^2)}+\left \lVert \partial u \right \rVert_{C(H^2)}^2 )\\
& & +  (\left \lVert  \partial_t u(0) \right \rVert_{L^2}^2+\left \lVert \partial u \right \rVert_{C(H^2)}^2  ) P_5(\left \lVert \partial u \right \rVert_{C(H^2)}).
\end{eqnarray*}
By Gronwall inequality, we get
\begin{eqnarray*}
\left \lVert \nabla\partial_t u \right \rVert_{L^2} \lesssim P(\left \lVert g_0 \right \rVert_{H^1})Q(\left \lVert u \right \rVert_{C(\mathfrak{H}^3)}).
\end{eqnarray*}
\end{proof}

\begin{proof}[Proof of Theorem \ref{main result 1}]
We choose data $g_0^{\delta}$ such that $\left\lVert g_0^{\delta}\right\rVert_{H^1}<C$ and $\delta^2\left\lVert g_0^{\delta}\right\rVert_{H^2}^2<C$. Without any restriction we make the bootstrap assumption
\begin{equation}       \label{bootstrap assumption}
\left\lVert u\right\rVert_{C([0,T];\mathfrak{H}^2)}+\left\lVert \nabla\tau(u)\right\rVert_{C([0,T];L^2)}\leq 2C(\left\lVert u(0)\right\rVert_{\mathfrak{H}^2},\left\lVert \nabla\tau(u(0))\right\rVert_{L^2}).
\end{equation}
Define the energy functional by
\begin{equation*}
E_1(u,\partial_t u)=\frac{1}{2} \left \lVert d u \right \rVert_{L^2}^2+\frac{\delta^2}{2} \left \lVert \partial_t u \right \rVert_{L^2}^2 ,
\end{equation*}
then by (\ref{approximate equation}), we have $\frac{d}{dt}E_1 =0$. Define the second order energy functional by
\begin{equation*}
E_2(u,\partial_t u)=\frac{1}{2} \left \lVert \nabla du \right \rVert_{L^2}^2+\frac{\delta^2}{2} \left \lVert \nabla\partial_t u \right \rVert_{L^2}^2 ,
\end{equation*}
by (\ref{approximate equation}) we have
\begin{align}             \nonumber     
\frac{d}{dt}E_2=&\frac{1}{2}\frac{d}{dt}\left \lVert \nabla du \right \rVert_{L^2}^2+\delta^2 \int_{\H^2}\< \nabla_i \nabla_t\partial_t u, \nabla_i\partial_t u \>  +\< [\nabla_t,\nabla_i]\partial_t u,\nabla_i\partial_t u \>{\rm dvol}_g\\      \label{dt E2}
=&\frac{1}{2}\frac{d}{dt}\left \lVert \nabla du \right \rVert_{L^2}^2+\int_{\H^2}\< J\nabla_i \partial_t u+\nabla_i \tau(u),\nabla_i \partial_t u \>+\delta^2\< [\nabla_t,\nabla_i]\partial_t u,\nabla_i\partial_t u \>{\rm dvol}_g,
\end{align}
by integration by parts and $\<JX,X\>=0$, the second term of (\ref{dt E2}) becomes
\begin{align}         \nonumber
&\int_{\H^2} \< \nabla_i \tau(u),\nabla_i \partial_t u \>{\rm dvol}_g\\ \nonumber
=& \int_{\H^2} \<\nabla_i(\nabla_j e_j u-(\nabla_j e_j)u),\nabla_i \partial_t u\>{\rm dvol}_g,\\          \label{dt E2  second term}
=& \int_{\H^2} \< [\nabla_i,\nabla_j]e_j u,\nabla_i \partial_t u \>-\< \nabla_i e_j u,[\nabla_j,\nabla_t]e_i u \>-\< \nabla_i e_j u, \nabla_t \nabla_j e_i u \>{\rm dvol}_g, 
\end{align}
furthermore, the last term of (\ref{dt E2  second term}) becomes
\begin{align*}
&\int_{\H^2}-\< \nabla_i e_j u, \nabla_t \nabla_j e_i u \>{\rm dvol}_g\\
=& \int_{\H^2} -\< \nabla_i e_j u ,\nabla_t [\nabla_j,\nabla_i]u \> -\< \nabla_i \nabla_j u-(\nabla_i e_j)u,\nabla_t(\nabla_i e_j u-(\nabla_i e_j)u) \>\\
& -\< \nabla_ie_j u,\nabla_t (\nabla_i e_j)u \>-\< (\nabla_i e_j)u,\nabla_t \nabla_ie_j u \>+\< \nabla_i e_ju,\nabla_t(\nabla_ie_j)u\>{\rm dvol}_g\\
=& -\frac{1}{2}\frac{d}{dt} \left\lVert \nabla du\right\rVert_{L^2}^2+\int_{\H^2}-\< \nabla_i e_j u,\nabla_t[\nabla_j,\nabla_i]u+\nabla_t(\nabla_i e_j)u \> -\< (\nabla_i e_j)u,[\nabla_t,\nabla_i]e_j u \>\\
&-\< (\nabla_2e_j)u,\nabla_t e_j u \>+\< \nabla_i(\nabla_ie_j)u,\nabla_t e_j u \>+\< \nabla_ie_j u,\nabla_t(\nabla_ie_j)u \>{\rm dvol}_g.
\end{align*}
Hence, by (\ref{Sobolev inequality 4}) and H\"{o}lder inequality we have
\begin{align*}
\frac{d}{dt}E_2\lesssim (\left \lVert  \partial_t u \right \rVert_{L^2}+\left \lVert \nabla \partial_t u \right \rVert_{L^2})P(\left \lVert  u \right \rVert_{\mathfrak{H}^3}) +\delta^2 \left \lVert \nabla \partial_t u \right \rVert_{L^2}^2\left \lVert  \partial_t u \right \rVert_{L^2} \left \lVert u \right \rVert_{\mathfrak{H}^3}.
\end{align*}
Define the third order energy functional by
\begin{equation*}
E_3 =\frac{1}{2} \left \lVert \nabla\tau(u) \right \rVert_{L^2}^2+ \frac{\delta^2}{2} \left \lVert \nabla^2 \partial_t u \right \rVert_{L^2}^2.
\end{equation*}
By integration by parts gives
\begin{align}     \nonumber
&\frac{d}{dt}E_3\\                       \label{E3 with delta}
\lesssim & \int_{\H^2} \delta^2 \big( |\partial_t u||du| |\nabla \partial_t u|+|du|^2 |\partial_t u|^2 +|\partial_t u|^2 |\nabla^2 u|+|du||\partial_t u|^2 \big) (|\nabla^2 \partial_t u|+|\nabla \partial_t u|){\rm dvol}_g\\          \label{E3 without delta}
 & +\int_{\H^2}|du|^2 |\tau(u)||\nabla\partial_t u| +|\nabla \tau(u)| (|\partial_t u| |du|^2 +|\nabla\partial_t u|) \\\nonumber
&  +|\nabla \tau(u)|( |du|^3 |\partial_t u| + |du|^2 |\nabla\partial_t u| +|\nabla^2 u| |du| |\partial_t u|) {\rm dvol}_g
\end{align}
By (\ref{Sobolev inequality 2}) and H\"{o}lder inequality, we have
\begin{align*}
(\ref{E3 with delta}) \lesssim & \delta^2 (\left \lVert \nabla^2 \partial_t u \right \rVert_{L^2}+\left \lVert \nabla \partial_t u \right \rVert_{L^2})(\left \lVert  \partial_t u \right \rVert_{L^4} \left \lVert \nabla \partial_t u \right \rVert_{L^4} \left \lVert \partial u \right \rVert_{H^2} + \left \lVert  \partial_t u \right \rVert_{L^4}^2 \left \lVert  \partial u \right \rVert_{H^2}^2 \\
&+\left \lVert  \partial_t u \right \rVert_{L^8}^2\left \lVert \nabla^2 u \right \rVert_{L^4}+\left \lVert \partial_t u \right \rVert_{L^4}^2 \left \lVert  \partial u \right \rVert_{H^2})\\
\lesssim & \delta^2 (\left \lVert \nabla^2 \partial_t u \right \rVert_{L^2}+\left \lVert \nabla \partial_t u \right \rVert_{L^2}) ( \left \lVert \nabla^2 \partial_t u \right \rVert_{L^2}^{1/2}\left \lVert \nabla \partial_t u \right \rVert_{L^2}\left \lVert  \partial_t u \right \rVert_{L^2}^{1/2}\left\lVert u \right\rVert_{\mathfrak{H}^3} \\
&+\left \lVert \nabla \partial_t u \right \rVert_{L^2}\left \lVert  \partial_t u \right \rVert_{L^2}\left \lVert  \partial u \right \rVert_{H^2}^2 +\left \lVert \nabla \partial_t u \right \rVert_{L^2}^{3/2}\left \lVert  \partial_t u \right \rVert_{L^2}^{1/2}\left \lVert \partial u \right \rVert_{H^2}\\
&+\left \lVert \nabla \partial_t u \right \rVert_{L^2}\left \lVert  \partial_t u \right \rVert_{L^2}\left \lVert  \partial u \right \rVert_{H^2} ),\\
\lesssim & \delta^2 \left \lVert  \partial_t u \right \rVert_{H^2}^2 \left \lVert \nabla \partial_t u \right \rVert_{L^2} P(\left \lVert  u \right \rVert_{\mathfrak{H}^3}).
\end{align*}
Similarly, we have
\begin{align*}
(\ref{E3 without delta}) \lesssim & (\left \lVert \nabla \partial_t u \right \rVert_{L^2}+ \left \lVert  \partial_t u \right \rVert_{L^2})P(\left \lVert  u \right \rVert_{\mathfrak{H}^3}).
\end{align*}
Hence,
\begin{equation}    \label{dt E123}
\begin{aligned}
&\frac{d}{dt}(E_1+E_2+E_3)\\
\leq & C\delta^2 \left\lVert \partial_t u\right\rVert_{H^2}^2\left\lVert \nabla \partial_t u\right\rVert_{L^2}P(\left \lVert  u \right \rVert_{\mathfrak{H}^3})+ C (\left\lVert \nabla \partial_t u\right\rVert_{L^2}+\left\lVert  \partial_t u\right\rVert_{L^2})P(\left \lVert  u \right \rVert_{\mathfrak{H}^3}).
\end{aligned}
\end{equation}
Since we have by integration by parts
\begin{equation*}
\left\lVert \nabla^2 du\right\rVert_{L^2}^2\lesssim \left\lVert \nabla \tau(u)\right\rVert_{L^2}^2+\left\lVert du\right\rVert_{L^6}^6+\left\lVert \nabla du\right\rVert_{L^4}^2\left\lVert du\right\rVert_{L^4}^2+\left\lVert \nabla^2 du\right\rVert_{L^2}^2,
\end{equation*}
then by (\ref{Sobolev inequality 2}) we obtain
\begin{equation}       \label{relation of nabla2 du and nabla tau u}
\left\lVert \nabla^2 du\right\rVert_{L^2}^2\lesssim P(\left\lVert u\right\rVert_{\mathfrak{H}^2}^2)+\left\lVert \nabla\tau(u)\right\rVert_{L^2}^2.
\end{equation}
Thus, integrating (\ref{dt E123}) in time and taking the supremum over $t\in [0,T]$ , we have
\begin{equation}          \label{dt E123 integrating in time}
\begin{aligned}
& \delta^2 \left \lVert \partial_t u \right \rVert_{C(H^2)}^2+\left \lVert   u \right \rVert_{C(H^2)}^2+\left \lVert \nabla\tau(u) \right \rVert_{C(L^2)}^2\\
\leq & \delta^2 \left \lVert g(0) \right \rVert_{H^2}^2+\left \lVert   u(0) \right \rVert_{H^2}^2+\left \lVert \nabla\tau(u(0)) \right \rVert_{L^2}^2\\
& +CT\delta^2 \left \lVert \partial_t u \right \rVert_{C(H^2)}^2 P(\left \lVert u \right \rVert_{C(\mathfrak{H}^2)}+\left \lVert\nabla\tau( u) \right \rVert_{C(L^2)})+CTQ(\left \lVert u \right \rVert_{C(\mathfrak{H}^2)}+\left \lVert\nabla\tau( u) \right \rVert_{C(L^2)}).
\end{aligned}
\end{equation}
Choosing $T$ small such that $CTP(\left \lVert u \right \rVert_{C([0,\tilde{T}];\mathfrak{H}^2)}+\left \lVert\nabla\tau( u) \right \rVert_{C([0,\tilde{T}];L^2)})<\frac{1}{2},$ from (\ref{dt E123 integrating in time}) we have
\begin{equation}\label{dt E123  case 1}
\begin{aligned}        
&\left \lVert   u \right \rVert_{C(H^2)}^2+\left \lVert  \nabla \tau( u) \right \rVert_{C(L^2)}^2+\frac{\delta^2}{2}\left \lVert  \partial_t u \right \rVert_{C(H^2)}^2-\delta^2\left \lVert g_0 \right \rVert_{H^2}^2\\
\leq & \left \lVert   u(0)\right \rVert_{\mathfrak{H}^2}^2+\left \lVert  \nabla \tau( u(0)) \right \rVert_{L^2}^2+CTQ(\left \lVert u \right \rVert_{C(\mathfrak{H}^2)}+\left \lVert\nabla\tau( u) \right \rVert_{C(L^2)}).
\end{aligned}
\end{equation}
If $\left \lVert  \partial_t u \right \rVert_{C(H^2)}^2\geq 2\left \lVert g_0 \right \rVert_{H^2}^2$, (\ref{dt E123  case 1}) implies 
\begin{equation*}
\left \lVert   u \right \rVert_{C(H^2)}^2+\left \lVert  \nabla \tau( u) \right \rVert_{C(L^2)}^2\leq \left \lVert   u(0)\right \rVert_{\mathfrak{H}^2}^2+\left \lVert  \nabla \tau( u(0)) \right \rVert_{L^2}^2+CTQ(\left \lVert u \right \rVert_{C(\mathfrak{H}^2)}+\left \lVert\nabla\tau( u) \right \rVert_{C(L^2)}).
\end{equation*}
If $\left \lVert  \partial_t u \right \rVert_{C(H^2)}^2< 2\left \lVert g_0 \right \rVert_{H^2}^2$, from (\ref{dt E123 integrating in time}) we obtain
\begin{equation*}
\left \lVert   u \right \rVert_{C(H^2)}^2+\left \lVert  \nabla \tau( u) \right \rVert_{C(L^2)}^2\leq C+\left \lVert   u(0)\right \rVert_{\mathfrak{H}^2}^2+\left \lVert  \nabla \tau( u(0)) \right \rVert_{L^2}^2+3CTQ(\left \lVert u \right \rVert_{C(\mathfrak{H}^2)}+\left \lVert\nabla\tau( u) \right \rVert_{C(L^2)}).
\end{equation*}
Hence, by the bootstrap assumption (\ref{bootstrap assumption}), there exists $T$ small such that
\begin{equation*}
\left \lVert   u \right \rVert_{C(H^2)}+\left \lVert  \nabla \tau( u) \right \rVert_{C(L^2)}< \frac{3}{2}C.
\end{equation*}
Therefore, by (\ref{relation of nabla2 du and nabla tau u}) we have
\begin{equation*}
\left \lVert   u \right \rVert_{C(\mathfrak{H}^3)}\leq C(\left \lVert   u(0) \right \rVert_{\mathfrak{H}^3})
\end{equation*}
for some fixed $T>0$ depending only on the size of data $u(0)$.

\end{proof}

\section{The Coulomb gauge representation of the equation}
In this section, we rewrite the equivariant Schr\"{o}dinger map in the Coulomb gauge, then obtain the $(\psi^+,\ \psi^-)$-system of coupled Schr\"{o}dinger equations. Conversely, we can recover the map $u$ from $\psi^+$ or $\psi^-$ at fixed time.

We choose $v\in T_u \H^2$ such that $v\cdot v=1$ and define $w=u\times v$. Thus
$$w\cdot w=1,\ \ \  w\cdot u=w\cdot v=0, \ \ \    w\times u=v,  \ \ \   v\times w=u .$$

Since $u$ is 1-equivariant it is natural to work with 1-equivariant frame, that is
$$v=e^{\theta R} \bar{v}(r),\mbox{ }\mbox{ }w=e^{\theta R} \bar{w}(r),$$
where $\bar{v}$, $\bar{w}$ are unit symmetric vectors in $\H^2$. On one hand in such a frame we obtain the differentiated fields $\psi_k$ and the connection coefficients $A_k$, by
$$\psi_k= \partial_k u \cdot v+i\partial_k u\cdot w,\ \ A_k= \partial_k v \cdot w.$$
On the other hand, given $\psi_k$ and $A_k$ we can return to the frame $(u,v,w)$ via the ODE system:
\begin{eqnarray}
\left\{
\begin{array}{lll}           \label{system of u,v,w}
 \partial_k u &=& (\Re \psi_k)v +(\Im \psi_k)w,\\
 \partial_k v &=& -(\Re \psi_k)u +A_k w,\\
 \partial_k w &=& -(\Im \psi_k)u -A_k v.
\end{array}
\right.
\end{eqnarray}
If we introduce the covariant differentiation
\begin{equation*}
D_k=\partial_k+iA_k, \mbox{ }\mbox{ }k\in \{0,\ 1,\ 2\},
\end{equation*}
then the compatibility conditions are imposed
\begin{equation} \label{compatibility conditions}
D_k \psi_l=D_l \psi_k,\mbox{ }\mbox{ } l,k\in\{0,\ 1,\ 2\}.
\end{equation}
Moreover, the curvature of this connection is given by
\begin{equation} \label{curvature}
D_k D_l-D_l D_k=i(\partial_k A_l- \partial_l A_k)=i\Im(\psi_k \bar{\psi_l}).
\end{equation}
An important geometric feature is that $\psi_2$, $A_2$ are closely related to the original map. Precisely, for $A_2$ we have
\begin{equation*}
A_2=(-v_2,v_1,0)\cdot  (w_1,w_2,w_3)=u_3,
\end{equation*}
and
\begin{equation*}
\psi_2= w_3-iv_3.
\end{equation*}
Hence we obtain $|\psi_2|^2=u_1^2+u_2^2$, and the following important conservation law
\begin{equation*}
A_2^2+|\psi_2|^2=1.
\end{equation*}

We now turn to choose the orthonormal frame $(\bar{v},\bar{w})$ on $\S^2$. For the equivariant Schr\"{o}dinger map, we use the Coulomb gauge ${\rm div}A=0$, namely, in the polar coordinate, $\partial_r A_1+\frac{\partial_{\theta}^2}{\sinh^2 r} A_2=0$. Since $A_2=u_3$ is radial, we can choose $A_1=0$, i.e
\begin{equation*}
\partial_r \bar{v} \cdot  \bar{w}=0,
\end{equation*}
which can be represented as ODE
\begin{eqnarray}                    \label{ODE of bar(v)}
\partial_r \bar{v}=-(\bar{v}\cdot  \partial_r \bar{u}) \bar{u}.
\end{eqnarray}
Then for matrix $U=(\bar{u},\bar{v},\bar{w})$, we have
\begin{equation*}
\partial_r U=M\cdot  U,
\end{equation*}
where $M=-u\cdot \partial_r u^{T}+ \partial_r u\cdot u^T$ is an antisymmetric matrix.

The ODE (\ref{ODE of bar(v)}) need to be initialized at some point. To avoid introducing a constant time-dependent potential into the equation via $A_0$, we need to choose this initialization uniformly with respect to $t$. Since we restrict the data $\lim\limits_{r \rightarrow \infty}\bar{u}(r,t)=\vec{k}$ for any t, we can fix the choice of $\bar{v}$ and $\bar{w}$ at infinity,
\begin{equation}                    \label{initial data of bar v}
\lim\limits_{r \rightarrow \infty}\bar{v}(r,t)=\vec{i},\ \ \ \ \ \lim\limits_{r \rightarrow \infty}\bar{w}(r,t)=\vec{j}.
\end{equation}
The existence and uniqueness of (\ref{ODE of bar(v)}) satisfying (\ref{initial data of bar v}) is standard. Indeed, for $u\in \mathfrak{H}^1$, using the Picard iteration scheme
\begin{eqnarray*}
\bar{v}=\sum\limits^{\infty}_{i=0}\bar{v}_i,\ \ \bar{v}_0=\bar{v}(\infty)=\vec{i},\ \ \bar{v}_i(r)=\int_{\infty}^r M(s)\bar{v}_{i-1}ds.
\end{eqnarray*}
By H\"{o}lder inequality, we have
\begin{align*}
\left \lVert \bar{v}_i(r) \right \rVert_{C([R,\infty))}  \leq &\left \lVert \partial_r \bar{u} \right \rVert_{L^2(R,\infty)} \left \lVert \frac{1}{\sinh r} \right \rVert_{L^2(R,\infty)}\left \lVert \bar{v}_{i-1}(r) \right \rVert_{C(R,\infty)},\\
\lesssim & \left \lVert u \right \rVert_{\mathfrak{H}^1} \left \lVert \bar{v}_{i-1}(r) \right \rVert_{C(R,\infty)},
\end{align*}
and
\begin{align*}
\left\lVert \partial_r \bar{v}_i \right\rVert_{L^2(R,\infty)}\leq  \left\lVert  \bar{v}_{i-1} \right\rVert_{C(R,\infty)}\left\lVert u \right\rVert_{\mathfrak{H}^1}.
\end{align*}
we choose $R$ large enough such that $\left \lVert \bar{u} \right \rVert_{\dot{H}^1_e([R,\infty))}<\epsilon $, we have $\left \lVert \bar{v}_i(r) \right \rVert_{C([R,\infty))} \leq \epsilon^{i} \left \lVert \bar{v}_0 \right \rVert_{C([R,\infty))}$. Hence, there exists unique solution $\left\lVert \bar{v}\right\rVert_{C(R,\infty)}+\left\lVert \partial_r\bar{v}\right\rVert_{L^2(R,\infty)}\lesssim \left\lVert u\right\rVert_{\mathfrak{H}^1}$. Then by $u\in\mathfrak{H}^1$, in a similar argument, for any $\epsilon>0$,there exists $\delta>0$ sufficiently small, such that $\left\lVert u\right\rVert_{\mathfrak{H}^1(R-\delta,R)}\left\lVert \sinh^{-1}r\right\rVert_{L^2(R-\delta,R)}\ll 1$ for $R-\delta>\epsilon$, the solution can be extended to $r=\epsilon$. Finally, we extend the solution to $r=0$. The first two components of $\bar{v}_i$ can be estimated immediately
\begin{align*}
\left\lVert \bar{v}_i^{1,2}\right\rVert_{C(0,\epsilon)}\lesssim \left\lVert \bar{v}_{i-1}\right\rVert_{C(0,\epsilon)}\left\lVert u\right\rVert_{\mathfrak{H}^1},
\end{align*}
for the third component of $\bar{v}_i$, by integration by parts and $\bar{u}(r)\rightarrow \vec{k}$ as $r\rightarrow 0$, we have
\begin{align*}
\left\lVert \bar{v}_i^3\right\rVert_{C(0,\epsilon)}= & \left\lVert \int_{\epsilon}^r \partial_r(\bar{v}_{i-1}\cdot(\bar{u}-\vec{k})\bar{u}^3)-\partial_r\bar{v}_{i-1}\cdot(\bar{u}-\vec{k})\bar{u}^3-\bar{v}_{i-1}\cdot(\bar{u}-\vec{k})\partial_r \bar{u}^3 ds \right\rVert_{C(0,\epsilon)},\\
\lesssim & \left(\lVert \partial_r \bar{v}_{i-1}\right\rVert_{L^2(0,\epsilon)}+\left\lVert \bar{v}_{i-1} \right\rVert_{C(0,\epsilon)})(\left\lVert u\right\rVert_{\mathfrak{H}^1(0,\epsilon)}+\left\lVert \bar{u}-\vec{k}\right\rVert_{C(0,\epsilon)}),
\end{align*}
and
\begin{align*}
\left\lVert \partial_r\bar{v}_i\right\rVert_{L^2(0,\epsilon)}\lesssim \left\lVert \bar{v}_{i-1}\right\rVert_{C(0,\epsilon)}\left\lVert u\right\rVert_{\mathfrak{H}^1(0,\epsilon)}.
\end{align*}
we choose $\epsilon$ small such that $\left\lVert u\right\rVert_{\mathfrak{H}^1(0,\epsilon)}+\left\lVert \bar{u}-\vec{k}\right\rVert_{C(0,\epsilon)}\ll 1$, then the iteration scheme gives the unique solution in $(0,\epsilon)$ with $\left\lVert \bar{v}\right\rVert_{C(0,\epsilon)}+\left\lVert \partial_r\bar{v}\right\rVert_{L^2(0,\epsilon)}\lesssim (\left\lVert u\right\rVert_{\mathfrak{H}^1(0,\epsilon)}+\left\lVert \bar{u}-\vec{k}\right\rVert_{C(0,\epsilon)}) \left\lVert \bar{v}_0\right\rVert_{C(0,\epsilon)}$. Therefore, by the above procedure, there exists a unique solution of (\ref{ODE of bar(v)}) satisfying (\ref{initial data of bar v}), moreover, we have
\begin{equation}                \label{control v by u}
\left\lVert \bar{v}\right\rVert_{C(0,\infty)}+\left\lVert \partial_r\bar{v}\right\rVert_{L^2(0,\infty)}\lesssim \left\lVert u\right\rVert_{\mathfrak{H}^1(0,\infty)}.
\end{equation}

\subsection{The Schr\"{o}dinger maps system in the Coulomb gauge: dynamic equations for $\psi_k$}
We derive the Schr\"{o}dinger equations for the differentiated fields $\psi_1$ and $\psi_2$.

In the geodesic polar coordinate, the Schr\"{o}dinger map flow can be written as
\begin{equation}    \label{psi0}
\psi_0=i(D_1 \psi_1+\coth r \psi_1 +\frac{1}{\sinh^2 r}D_2 \psi_2).
\end{equation}
Applying the operators $D_1$ and $D_2$ to both sides of this equation, we obtain
\begin{equation}   \label{differentialpsi0}
\left\{\begin{aligned}
D_1 \psi_0   = & i(D_1 D_1 \psi_1 + \coth r D_1 \psi_1 +\frac{1}{\sinh ^2r}D_1 D_2 \psi_2 -\frac{\psi_1}{\sinh ^2 r} -\frac{2\cosh r}{\sinh ^3 r}D_2 \psi_2),\\
D_2 \psi_0   = & i(D_2 D_1 \psi_1 +\coth r D_2 \psi_1 +\frac{D_2 D_2 \psi_2}{\sinh^2 r}).
\end{aligned}
\right.
\end{equation}
By the compatibility condition (\ref{compatibility conditions}), curvature of the connection (\ref{curvature}) and the Coulomb gauge $A_1=0$, we can derive the equations for $\psi_1$ and $\psi_2$,
\begin{equation}    \label{system of psi1 and psi2 version1}
\left\{\begin{aligned}
(i\partial_t +\Delta) \psi_1 = & (A_0 +\frac{A_2^2}{\sinh^2 r}+\frac{1}{\sinh^2 r})\psi_1+ \frac{2i \cosh r A_2}{\sinh^3 r}\psi_2 -\frac{i\Im(\psi_1 \bar{\psi}_2)\psi_2}{\sinh^2 r},\\
(i\partial_t +\Delta) \psi_2 = & (A_0 +\frac{A_2^2}{\sinh^2 r})\psi_1 + i\Im(\psi_1 \bar{\psi}_2)\psi_1.
\end{aligned}
\right.
\end{equation}
where $\Delta:= \partial_{rr}+\coth r\partial_r$. Then (\ref{system of psi1 and psi2 version1}) can be written as
\begin{equation}        \label{system of psi1 and psi2 version2}
\left\{\begin{aligned}
(i\partial_t +\Delta -\frac{2}{\sinh^2 r}) \psi_1 -\frac{2i\cosh r}{\sinh^2 r} \frac{\psi_2}{\sinh r}  =&  (A_0 +\frac{A_2^2-1}{\sinh^2 r})\psi_1\\
& + (\frac{2i \cosh r ( A_2-1)}{\sinh^2 r} -i\Im(\psi_1 \frac{\bar{\psi}_2}{\sinh r}))\frac{\psi_2}{\sinh r},\\
(i\partial_t +\Delta -\frac{2}{\sinh^2 r})\frac{\psi_2}{\sinh r} +\frac{2i\cosh r}{\sinh^2 r} \psi_1  = & (A_0 +\frac{A_2^2-1}{\sinh^2 r})\frac{\psi_2}{\sinh r}\\
&- (\frac{2i \cosh r ( A_2-1)}{\sinh^2 r} -i\Im(\psi_1 \frac{\bar{\psi}_2}{\sinh r}))\psi_1.
\end{aligned}
\right.
\end{equation}
where $A_0$ and $A_2-1$ can be expressed in terms of $\psi_1$ and $\psi_2$. In fact, from the curvature (\ref{curvature}) for $k=1,\mbox{ }l=2$ and compatibility condition (\ref{compatibility conditions}), we have
\begin{equation}                \label{compatibility A2 psi2}
\partial_r A_2 = \Im(\psi_1 \bar{\psi}_2), \mbox{ } \partial_r \psi_2 =iA_2 \psi_1.
\end{equation}
Since $A_2(0)=1$, (\ref{compatibility A2 psi2}) gives
\begin{equation}           \label{A2 with psi1 and psi2}
A_2 -1 = \int_0^r \Im(\psi_1 \bar{\psi}_2)(s)ds.
\end{equation}
From (\ref{curvature}) when $k=0,\ l=1$ and (\ref{psi0}), we have
\begin{equation*}
\partial_r A_0 = \Im(\psi_1 \bar{\psi}_0)=-\frac{1}{2\sinh^2 r} \partial_r (\sinh^2 r |\psi_1|^2-|\psi_2|^2),
\end{equation*}
which together with initial data of $(\bar{v},\bar{w})$ frame, yields
\begin{equation}         \label{A0 with psi1 and psi2}
A_0 (t,r) = -\frac{1}{2}(|\psi_1|^2-|\frac{\psi_2}{\sinh r}|^2)+\int_r^{+\infty} \frac{\cosh s}{\sinh s}(|\psi_1|^2-|\frac{\psi_2}{\sinh s}|^2) ds.
\end{equation}
Therefore the two variables $\psi_1$ and $\psi_2$ are not independent.

Since the linear part of this system is not decoupled, we introduce the two new variables $\psi^+$ and $\psi^-$, defined as
\begin{equation}                   \label{psi+psi-}
\psi^+= \psi_1 +i\frac{\psi_2}{\sinh r},\mbox{ }\psi^-= \psi_1 -i\frac{\psi_2}{\sinh r}.
\end{equation}
From (\ref{system of psi1 and psi2 version2}) and $A_2^2+|\psi_2|^2 =1$, we obtain
\begin{equation} \label{system of psi+ and psi- version}
\left\{\begin{aligned}
(i\partial_t +\Delta -2\frac{\cosh r+1}{\sinh^2 r}) \psi^+  = (A_0 +2\frac{\cosh r(A_2-1)}{\sinh^2 r}-\Im(\psi^+ \frac{\bar{\psi}_2}{\sinh r}))\psi^+, \\
(i\partial_t +\Delta +2\frac{\cosh r-1}{\sinh^2 r}) \psi^-  =  (A_0 -2\frac{\cosh r(A_2-1)}{\sinh^2 r}+\Im(\psi^- \frac{\bar{\psi}_2}{\sinh r}))\psi^-.
\end{aligned}
\right.
\end{equation}
It turns out that the linear part of $\psi^{\pm}$-system is decoupled. The compatibility condition (\ref{compatibility conditions}) is reduced to
\begin{equation}              \label{compatibility psi+psi-}
\partial_r \sinh r(\psi^+-\psi^-) =- A_2 (\psi^++\psi^-),
\end{equation}
and the coefficients $A_0$ and $A_2-1$ can be expressed in terms of $\psi^{\pm}$,
\begin{align}        \label{A2 with psi+ psi-}
& A_2-1=  \int_0^r \frac{|\psi^+|^2 -|\psi^-|^2}{4} \sinh sds,  \\     \label{A0 with psi+ psi-}
& A_0=  -\frac{1}{2} \Re(\bar{\psi}^+ \psi^-)+\int_r^{\infty} \frac{\cosh s}{\sinh s}\Re(\bar{\psi}^+ \psi^-) ds.
\end{align}
Define $\mathcal{V}^{\pm}$ as the vector
\begin{equation}                    \label{definition of V-}
\mathcal{V}^{\pm}=\partial_r u \pm \frac{u \times  (k\times u)}{\sinh r} \in T_u (\H^2),
\end{equation}
then $\psi^{\pm}$ is the representation of $\mathcal{V}^{\pm}$ in the coordinate frame $(v,w)$ and the energy of $u$ has a new representation, i.e
\begin{align*}
E(u)
=& \pi \int_0^{\infty} (|\partial_r \bar{u}|^2 +\frac{|\bar{u} \times (k\times \bar{u})|^2}{\sinh^2 r}) \sinh r dr,\\
=& \pi \int_0^{\infty} |\mathcal{V}^{\pm}|^2 \sinh r \mp 2 \partial_r \bar{u} \cdot  (\bar{u} \times   (k \times \bar{u}))dr, \\
=& \pi \left \lVert \mathcal{V}^{\pm} \right \rVert _{L^2}^2 \mp 2\pi \int_0^{\infty}  \partial_r \bar{u}_3dr,\\
=& \pi \left \lVert \psi^{\pm} \right \rVert _{L^2}^2 \mp 2\pi (\bar{u}_3(\infty)-\bar{u}_3(0)),\\
=& \pi \left\lVert \psi^{\pm}\right\rVert_{L^2}^2.
\end{align*}
Hence, $\left\lVert \psi^{\pm}\right\rVert_{L^2}^2$ is conserved for all time. Moreover, if we assume that $\left\lVert \psi^{\pm}\right\rVert_{L^2},\ \left\lVert \tilde{\psi}^{\pm}\right\rVert_{L^2}<2$ and $\left\lVert u-\tilde{u}\right\rVert_{\mathfrak{H}^1}\ll 1$, we obtain the Lipschitz continuity of $\psi^{\pm}$ with $u$, namely
\begin{equation}                      \label{Lipschitz continuity of psi+-with u}
\left\lVert \psi^{\pm}-\tilde{\psi}^{\pm}\right\rVert_{L^2} \lesssim \left\lVert u-\tilde{u}\right\rVert_{\mathfrak{H}^1}.
\end{equation}
In fact, by the above assumptions, (\ref{A2 with psi+ psi-}) implies $u_3,\ \tilde{u}_3\gtrsim 1-\frac{1}{4}\left\lVert \psi^-\right\rVert_{L^2}>0$. On interval $[R,\infty)$, by (\ref{ODE of bar(v)}), we have
\begin{align*}
(v-\tilde{v})(r)=\int_r^{\infty}(v-\tilde{v})\cdot \partial_r u\ u+\tilde{v}\cdot\partial_r(u-\tilde{u})u-\partial_r \tilde{v}\cdot u(u-\tilde{u})ds,
\end{align*}
then (\ref{control v by u}) and $u_3-\tilde{u}_3=\frac{u_1^2-\tilde{u}_1^2+u_2^2-\tilde{u}_2^2}{u_3+\tilde{u}_3}$ imply
\begin{align*}
\left\lVert v-\tilde{v}\right\rVert_{C(R,\infty)}\lesssim \left\lVert v-\tilde{v}\right\rVert_{C(R,\infty)}\left\lVert \partial_r u\right\rVert_{L^2(R,\infty)}+\left\lVert u-\tilde{u}\right\rVert_{\mathfrak{H}^1},
\end{align*}
choose $R$ large enough, we have $\left\lVert v-\tilde{v}\right\rVert_{C(R,\infty)}\lesssim \left\lVert u-\tilde{u}\right\rVert_{\mathfrak{H}^1(R,\infty)}$. Then for any $\epsilon>0$ small, on interval $[\epsilon,\ R]$, there exists $\delta>0$ such that any interval $I\subset[\epsilon,R]$ with $|I|<\delta$, we have $\left\lVert \partial_r u\right\rVert_{L^2(I)}\left\lVert \sinh^{-1}r\right\rVert_{L^2(I)}\ll 1$. By a similar argument to that on $[R,\infty)$, we obtain $\left\lVert v-\tilde{v}\right\rVert_{C(\epsilon,R)}\lesssim \left\lVert u-\tilde{u}\right\rVert_{\mathfrak{H}^1(\epsilon,R)}$. Finally, on interval $(0,\epsilon)$, by Sobolev embedding (\ref{Sobolev embedding H1e}), we have
\begin{align*}
\left\lVert \frac{u_3-\tilde{u}_3}{\sinh r}\right\rVert_{L^2(\epsilon)}+\left\lVert u_3-\tilde{u}_3\right\rVert_{L^{\infty}(0,\epsilon)}\lesssim \left\lVert u-\tilde{u}\right\rVert_{\mathfrak{H}^1(0,\epsilon)},
\end{align*}
then we get
\begin{align*}
\left\lVert \partial_r (v-\tilde{v})\right\rVert_{L^2(0,\epsilon)}\lesssim \left\lVert v-\tilde{v}\right\rVert_{C(0,\epsilon)}\left\lVert \partial_r u\right\rVert_{L^2}+\left\lVert u-\tilde{u}\right\rVert_{\mathfrak{H}^1},
\end{align*}
which implies by integration by parts
\begin{align*}
\left\lVert v-\tilde{v}\right\rVert_{C(0,\epsilon)}\lesssim & \left\lVert \int_r^{\epsilon}\partial_r((v-\tilde{v})\cdot(u-\vec{k})u)-\partial_r(v-\tilde{v})\cdot(u-\vec{k})u-(v-\tilde{v})\cdot(u-\vec{k})\partial_r uds\right\rVert_{C(0,\epsilon)}\\
& +\left\lVert \int_r^{\epsilon}\partial_r(\tilde{v}\cdot(u-\tilde{u})u)-\partial_r\tilde{v}\cdot(u-\tilde{u})u-\tilde{v}\cdot(u-\tilde{u})\partial_r uds\right\rVert_{C(0,\epsilon)}\\
& + \left\lVert \partial_r\tilde{v}\right\rVert_{L^2(0,\epsilon)}\left\lVert \frac{u-\tilde{u}}{\sinh r}\right\rVert_{L^2(0,\epsilon)},\\
\lesssim & \left\lVert v-\tilde{v}\right\rVert_{C(0,\epsilon)} (\left\lVert u-\vec{k}\right\rVert_{C(0,\epsilon)}+\left\lVert \frac{u-\vec{k}}{\sinh r}\right\rVert_{L^2(0,\epsilon)})\\
& +\left\lVert \partial_r(v-\tilde{v})\right\rVert_{L^2(0,\epsilon)}\left\lVert \frac{u-\vec{k}}{\sinh r}\right\rVert_{L^2(0,\epsilon)}+\left\lVert u-\tilde{u}\right\rVert_{\mathfrak{H}^1(0,\epsilon)},\\
\lesssim & \left\lVert v-\tilde{v}\right\rVert_{C(0,\epsilon)}\left\lVert u\right\rVert_{\mathfrak{H}^1(0,\epsilon)} +\left\lVert u-\tilde{u}\right\rVert_{\mathfrak{H}^1(0,\epsilon)},
\end{align*}
which together with $\left\lVert u\right\rVert_{\mathfrak{H}^1(0,\epsilon)}\ll 1$, yields $\left\lVert v-\tilde{v}\right\rVert_{C(0,\epsilon)}\lesssim \left\lVert u-\tilde{u}\right\rVert_{\mathfrak{H}^1(0,\epsilon)}$. Therefore, we obtain
\begin{align}                       \label{Lipschitz v with u}
\left\lVert v-\tilde{v}\right\rVert_{C(0,\epsilon)}\lesssim \left\lVert u-\tilde{u}\right\rVert_{\mathfrak{H}^1}.
\end{align}
Then by (\ref{Lipschitz v with u}), (\ref{definition of V-}) and Sobolev embedding (\ref{Sobolev embedding H1e}), the Lipschitz continuity (\ref{Lipschitz continuity of psi+-with u}) follows.

In this paper we will work with the key system (\ref{system of psi+ and psi- version}) to obtain the space-time estimates for $\psi^{\pm}$.

Suppose $\psi^{\pm}$ satisfies the compatibility condition (\ref{compatibility psi+psi-}) and $\left \lVert \psi^{\pm} \right \rVert _{L^2}<\infty$, define $A_2$, $\psi_1$, $\psi_2$ by (\ref{A2 with psi+ psi-}) and (\ref{psi+psi-}), then they satisfy the relation (\ref{compatibility A2 psi2}). Furthermore, we claim that $\psi_1 \in L^2$ and $\psi_2,\ A_2 -1 \in \dot{H}^1_e$. In fact, by (\ref{A2 with psi+ psi-}) and (\ref{psi+psi-}), we have $\psi_1 \in L^2$, $A_2\in L^{\infty}$, from (\ref{psi+psi-}), (\ref{compatibility A2 psi2}) and (\ref{basic inequality2}), we get $\frac{\psi_2}{\sinh r},\ \partial_r \psi_2,\ \partial_r A_2$, and $\frac{A_2-1}{\sinh r}\in L^2 $.

Denote $R_+\psi^+=e^{i2\theta}\psi^+$ and $R_-\psi^-=\psi^-$. Then we have
\begin{proposition}      \label{equivalent of u and psi+- proposition}
\begin{equation}         \label{equivalent of u and psi+- inequality}
\left\lVert u\right\rVert_{\mathfrak{H}^3}\sim \left\lVert R_+\psi^+\right\rVert_{H^2}+\left\lVert R_-\psi^-\right\rVert_{H^2}.
\end{equation}
\end{proposition}

\begin{proof}
If $u\in\mathfrak{H}^1$, we easily obtain $\left\lVert du\right\rVert_{L^2}^2=\pi\left\lVert \psi^{\pm}\right\rVert_{L^2}^2$. If $u\in\mathfrak{H}^2$, by the equivariance condition, we have 
\begin{align}           \label{representation of u in H2}
\partial_{rr}u,\ \ (\frac{\cosh r}{\sinh r}\partial_r-\frac{1}{\sinh^2 r})(u_1,\ u_2),\ \ \frac{\cosh r}{\sinh r}\partial_r u_3\in L^2,
\end{align}
Since $A_1=0$, then $\nabla_r(v+iw)=0$, which gives
\begin{equation}
\begin{aligned}
\partial_r\psi^{\pm}=&\partial_r(\mathcal{V}^{\pm}\cdot(v+iw)),\\
=&\partial_r\mathcal{V}^{\pm}\cdot(v+iw),
\end{aligned}
\end{equation}
by the representation of $\mathcal{V}^{\pm}$ (\ref{definition of V-}), we have
\begin{align}                         \nonumber
\partial_r\mathcal{V}^{\pm}= & \partial_r(\partial_r u\pm \frac{-u_3 u+\vec{k}}{\sinh r}),\\     \nonumber
= & \partial_{rr} u\mp\frac{\partial_r u_3\cdot u+u_3\partial_r u}{\sinh r}\mp \frac{\cosh r(\vec{k}-u_3u)}{\sinh^2 r},\\ \nonumber
= & (\partial_{rr}\mp\frac{\partial_r}{\sinh r}\pm \frac{\cosh r}{\sinh^2 r})u\mp \frac{u_3 -1}{\sinh r}\partial_r u\mp \frac{\cosh r}{\sinh^2 r}\vec{k}\\      \nonumber
& \pm \frac{\cosh r(u_3-1)}{\sinh^2 r}u\mp \frac{\partial_r u_3\cdot u}{\sinh r},\\             \label{F+-}
= & ((\partial_{rr}\mp(\frac{\partial_r}{\sinh r}-\frac{\cosh r}{\sinh^2 r}))u_1,(\partial_{rr}\mp(\frac{\partial_r}{\sinh r}-\frac{\cosh r}{\sinh^2 r}))u_2,(\partial_{rr}\mp\frac{\partial_r}{\sinh r})u_3)\\\nonumber
& \pm(\frac{\cosh r(u_3-1)}{\sinh^2 r}\vec{k}-\frac{u_3-1}{\sinh r}\partial_r u)\pm(\frac{\cosh r(u_3-1)}{\sinh^2 r}-\frac{\partial_r u_3}{\sinh r})u,
\end{align}
denote $F^{\pm}=(\ref{F+-})\cdot (v+iw)$, then
\begin{align}                    \label{partial psi+-}
\partial_r\psi^{\pm}=F^{\pm}\pm(\frac{1-u_3}{\sinh r}\psi^-+i\frac{\cosh r-1}{\sinh r}(u_3-1)\frac{\psi_2}{\sinh r}).
\end{align}
For $\psi^-$, since $\big|\frac{u_3-1}{\sinh r}\big|\lesssim \frac{u_1^2+u_2^2}{\sinh r}\in L^1(dr)$, applying $e^{\int_r^{\infty}\frac{u_3-1}{\sinh s}ds}$ to both sides of (\ref{partial psi+-}), by $F^{\pm}\in L^2$, we have $\partial_r(e^{\int_r^{\infty}\frac{u_3-1}{\sinh s}ds}\psi^-)\in L^2$. Since $u_3$ and $\psi^-$ are radial, we obtain $e^{\int_r^{\infty}\frac{u_3-1}{\sinh s}ds}\psi^-\in \dot{H}^1$, which gives $\psi^-\in L^4$ by (\ref{Sobolev inequality 2}). Hence, by (\ref{partial psi+-}) and $\big|\frac{u_3-1}{\sinh r}\big|\lesssim \frac{|u_1|+|u_2|}{\sinh r}\in L^4$, we have $\partial_r \psi^-\in L^2$. It also follows that $\partial_r \psi^+\in L^2$. 

In order to prove $\frac{\psi^+}{\sinh r}\in L^2$, we rewrite
\begin{align*}
\frac{\psi^+}{\sinh r}=& (\frac{\partial_r u}{\sinh r}+\frac{-u_3 u+\vec{k}}{\sinh^2 r})\cdot(v+iw),\\
=& ((\frac{\partial_r }{\sinh r}-\frac{1}{\sinh^2 r})u_1,(\frac{\partial_r }{\sinh r}-\frac{1}{\sinh^2 r})u_2,\frac{\partial_r }{\sinh r}u_3)\cdot(v+iw)\\
&+(\frac{1-u_3}{\sinh^2 r}u_1,\frac{1-u_3}{\sinh^2 r}u_2,\frac{u_1^2+u_2^2}{\sinh^2 r})\cdot(v+iw),
\end{align*}
by $\big|\frac{1-u_3}{\sinh^2 r}\big|\lesssim \frac{u_1^2+u_2^2}{\sinh^2 r}\in L^2$ and (\ref{representation of u in H2}), we have $\frac{\psi^+}{\sinh r}\in L^2$.

Conversely, if $R_{\pm}\psi^{\pm}\in H^1$, (\ref{Sobolev inequality 2}) implies $\psi^{\pm}\in L^4$. Then by $\big|\frac{u_3-1}{\sinh r}\big|\lesssim \frac{|\psi_2|}{\sinh r}\in L^4$ and (\ref{partial psi+-}), we have $F^{\pm}\in L^2$, namely, $[\partial_{rr}\mp(\frac{\cosh r}{\sinh r}\partial_r+\frac{\partial_{\theta}^2}{\sinh^2 r})]u\cdot(v+iw)\in L^2$. The part of $[\partial_{rr}\mp(\frac{\cosh r}{\sinh r}\partial_r+\frac{\partial_{\theta}^2}{\sinh^2 r})]u$ in the normal space is $-|\psi_1|^2\pm\big|\frac{\psi_2}{\sinh r} \big|\in L^2$ by $\psi^{\pm}\in L^4$. Therefore, (\ref{representation of u in H2}) is obtained.

If $u\in\mathfrak{H}^3$, by (\ref{equivalent of Hl of u}) and Lemma \ref{equivalent of Sobolev lemma}, we obtain $\nabla(-\Delta)u_i\in L^2$ for $i=1,\ 2,\ 3$, then by equivariance condition, we get
\begin{align}            \label{representation of u in H3 1}
\partial_r (\partial_{rr}+\frac{\cosh r}{\sinh r}\partial_r-\frac{1}{\sinh^2 r})(u_1,\ u_2),\ \partial_r(\partial_{rr}+\frac{\cosh r}{\sinh r}\partial_r)u_3\in L^2,
\end{align}
and 
\begin{align}           \label{representation of u in H3 2}
\frac{1}{\sinh r}(\partial_{rr}+\frac{\cosh r}{\sinh r}\partial_r-\frac{1}{\sinh^2 r})(u_1,\ u_2)\in L^2.
\end{align}
In order to prove $R_{\pm}\psi^{\pm}\in H^2$, it suffices to prove 
\begin{align}
(\partial_{rr}+\frac{\cosh r}{\sinh r}\partial_r-\frac{4}{\sinh^2 r})\psi^+,\ (\partial_{rr}+\frac{\cosh r}{\sinh r}\partial_r)\psi^-\in L^2.
\end{align}
By (\ref{partial psi+-}), we have 
\begin{align}            \nonumber
\partial_{rr}\psi^{\pm}=&\partial_r\big((\partial_{rr}\mp(\frac{\partial_r}{\sinh r}-\frac{\cosh r}{\sinh^2 r}))u_1,(\partial_{rr}\mp(\frac{\partial_r}{\sinh r}-\frac{\cosh r}{\sinh^2 r}))u_2,(\partial_{rr}\mp\frac{\partial_r}{\sinh r})u_3\big)\cdot(v+iw)\\\nonumber
& \pm\partial_r(i\frac{\cosh r-1}{\sinh r}(u_3-1)\frac{\psi_2}{\sinh r}\mp\frac{u_3-1}{\sinh r}\psi^-),\\        \label{estimate partial rr psi- 1}
=& \partial_r F^{\pm}-\frac{u_3-1}{\sinh r}\partial_r\psi^-\pm(u_3-1)(i\frac{\cosh r-1}{\sinh^2 r}(\frac{\psi_2}{\sinh r}+\frac{\partial_r(\psi^+-\psi^-)}{2i})\pm \frac{\cosh r}{\sinh^2 r}\psi^-)\\    \label{estimate partial rr psi- 2}
& \pm \frac{\partial_r A_2}{\sinh r}(i\frac{\cosh r-1}{\sinh r}\psi_2\mp\psi^-).
\end{align}
Since $R_{\pm}\psi^{\pm}\in H^1$, (\ref{Sobolev inequality 2}) implies $\psi^{\pm}\in L^4\cap L^6$, then by $|u_3-1|\lesssim |\psi_2|^2$ and (\ref{A2 with psi1 and psi2}), the third term of (\ref{estimate partial rr psi- 1}) and (\ref{estimate partial rr psi- 2}) are in $L^2$. From (\ref{representation of u in H3 1}), we also have $\partial_r F^{\pm}\in L^2$. Hence, $\partial_r(e^{\int_{\infty}^r\frac{u_3-1}{\sinh s}ds}\partial_r \psi^-)\in L^2$, which further gives $e^{\int_{\infty}^r\frac{u_3-1}{\sinh s}ds}\partial_r \psi^-\in L^4$, this implies $\partial_r \psi^-\in L^4$. Since $\frac{u_3-1}{\sinh r}\in L^4$, we obtain $\partial_{rr}\psi^-\in L^2$, therefore, we also get $\partial_{rr}\psi^+\in L^2$. 

Next, we estimate this term 
\begin{align}           \label{estimate third term of delta psi-}
\frac{\cosh r}{\sinh r}\partial_r \psi^-=\frac{\cosh r}{\sinh r}(F^--i\frac{\cosh r-1}{\sinh r}(u_3-1)\frac{\psi_2}{\sinh r}+\frac{u_3-1}{\sinh r}\psi^-).
\end{align}
By (\ref{representation of u in H3 2}), the first two components of $\frac{\cosh r}{\sinh r}F^-$ are in $L^2$. For the third component, which can be written as 
\begin{align}            \label{third component}
\frac{\cosh r}{\sinh r}(\partial_{rr}+\frac{\partial_r}{\sinh r})u_3\vec{k}(v+iw)=\frac{i(\cosh r-1)}{\sinh r}\psi_2(\partial_{rr}+\frac{\partial_r}{\sinh r})u_3+\frac{i\psi_2}{\sinh r}(\partial_{rr}+\frac{\partial_r}{\sinh r})u_3.
\end{align}
By (\ref{representation of u in H2}) and (\ref{representation of u in H3 1}), we have $\Delta u_3\in L^4$. Hence, the right hand side of (\ref{third component}) is in $L^2$. By a similar argument, the third term in (\ref{estimate third term of delta psi-}) 
is also in $L^2$. Thus, $R_-\psi^-\in H^2$.

For $R_+\psi^+$, we need to estimate
\begin{align} \nonumber
&(\partial_{rr}+\frac{\cosh r}{\sinh r}\partial_r-\frac{4}{\sinh^2 r})\psi^+\\     \nonumber 
=& \big((\partial_r^3+\frac{\cosh r-1}{\sinh r}\partial_r^2+\frac{\cosh r-4}{\sinh^2 r}\partial_r+\frac{3}{\sinh^3 r})u_1,(\partial_r^3+\frac{\cosh r-1}{\sinh r}\partial_r^2+\frac{\cosh r-4}{\sinh^2 r}\partial_r+\frac{3}{\sinh^3 r})u_2,\\ \label{estimate delta psi+}
&(\partial_r^3+\frac{\cosh r-1}{\sinh r}\partial_r^2-\frac{1}{\sinh^2 r}\partial_r)u_3\big)\cdot(v+iw)\\\nonumber
&-\frac{4(1-u_3)}{\sinh^3 r}i\psi_2+\partial_r(\frac{\cosh r-1}{2\sinh r}(u_3-1)(\psi^+-\psi^-)-\frac{u_3-1}{\sinh r}\psi^-)-\frac{3i\partial_r u_3}{\sinh^2 r}\psi_2\\\nonumber
&+\frac{\cosh r}{2\sinh r}\frac{\cosh r-1}{\sinh r}(u_3-1)(\psi^+-\psi^-)-\frac{\cosh r}{\sinh r}\frac{u_3-1}{\sinh r}\psi^-.
\end{align}
By (\ref{representation of u in H2}) and (\ref{representation of u in H3 2}), the first two components of (\ref{estimate delta psi+}) are in $L^2$, namely for $i=1,\ 2$
\begin{align*}
&(\partial_r^3+\frac{\cosh r-1}{\sinh r}\partial_r^2+\frac{\cosh r-4}{\sinh^2 r}\partial_r+\frac{3}{\sinh^3 r})u_i\\
=& \partial_r\Delta u_i-\frac{1}{\sinh r}\Delta u_i+(\frac{\cosh r-1}{\sinh r}\partial_r^2+4\frac{\cosh r-1}{\sinh^2 r}\partial_r)u_i.
\end{align*}
For the third component, since $\partial_r \Delta u_3\in L^2$, it suffices to estimate 
\begin{align*}
\frac{\cosh r}{\sinh r}\partial_r^2 u_3 \vec{k}\cdot(v+iw)=i\frac{\cosh r-1}{\sinh r}\psi_2 \partial_r^2 u_3+\frac{1}{2}(\psi^+-\psi^-)\partial_r^2 u_3.
\end{align*}
By (\ref{representation of u in H2}) and $|\psi_2|\leq 1$, we get $\frac{\cosh r-1}{\sinh r}\psi_2 \partial_r^2 u_3\in L^2$. By (\ref{Sobolev embedding H1e}) and (\ref{Sobolev inequality 4}), we have $\left\lVert \psi^+\right\rVert_{L^{\infty}}\lesssim \left\lVert \psi^+\right\rVert_{\dot{H}^1_e}$ and $\left\lVert \psi^-\right\rVert_{L^{\infty}}\lesssim \left\lVert \psi^-\right\rVert_{H^2}$, then $\frac{1}{2}(\psi^+-\psi^-)\partial_r^2 u_3\in L^2$. Therefore, (\ref{estimate delta psi+}) are in $L^2$. The other terms are also easily obtained by Sobolev embedding and $A_2^2+|\psi_2|^2=1$. Thus, $R_{+}\psi^+\in H^2$ is obtained.

\end{proof}

\subsection{Recovering the map from $\psi^+$}                %ÀûÓÃpsi+ µÃµ½psi-  £¬²¢µÃµ½Ó³Éäu
Here we will keep track of $\psi^+ \in L^2$, since it contains all the information about the map. Indeed, by (\ref{compatibility A2 psi2}), we have the system of $(A_2, \psi_2)$
\begin{equation}                           \label{system of A2 psi2}
\left\{
\begin{aligned}
& \partial_rA_2=  \Im(\psi^+ \bar{\psi}_2)-\frac{|\psi_2|^2}{\sinh r},\\
& \partial_r \psi_2 =iA_2 \psi^+ +A_2 \frac{\psi_2}{\sinh r}.\\
\end{aligned}
\right.
\end{equation}
Then from the choice of $(\bar{v}(\infty),\ \bar{w}(\infty))$ (\ref{initial data of bar v}), it gives the data $(A_2(\infty),\ \psi_2(\infty))=(1,\ 0)$. Given $\psi^+ \in L^2$ with $\left\lVert \psi^+ \right\rVert_{L^2}<2$, we reconstruct $A_2 -1,\ \psi_2 \in \dot{H}^1_e$ by above system (\ref{system of A2 psi2}), then by the system in (\ref{system of u,v,w}) with condition (\ref{initial data of bar v}), we can return to the map $u$.

                                                         %Lemma 4.1,ÀûÓÃpsi+ µÃµ½psi-
\begin{lemma}                      \label{construct psi2 A2 from psi+}
Let $\psi^+ \in L^2$, such that $\left \lVert  \psi^+  \right \rVert _{L^2} <2 $, the system (\ref{system of A2 psi2}) has a unique solution $(A_2, \psi_2)$ satisfying $\psi_2, \mbox{ }A_2 -1 \in \dot{H}^1_e$, and
\begin{equation}             \label{construct psi2 A2 from psi+ control}
\left \lVert  \psi_2  \right \rVert_{\dot{H}^1_e}+\left \lVert  A_2 -1  \right \rVert_{\dot{H}^1_e}+\left \lVert \frac{A_2 -1}{\sinh r}  \right \rVert_{L^1(dr)} \lesssim \left \lVert  \psi^+  \right \rVert_{L^2}.
\end{equation}
Moreover, we have the following properties:\\
(i)\ \ If $\psi^+\in L^p$, with $1 \leq  p< \infty$, then $\psi^-, \mbox{ }\frac{\psi_2}{\sinh r},\mbox{ }\frac{A_2 -1}{\sinh r}\in L^p$ and
\begin{equation}             \label{construct psi2 A2 from psi+ control 1}
\left \lVert  \psi^- \right \rVert_{L^p}+\left \lVert  \frac{\psi_2}{\sinh r} \right \rVert_{L^p}+\left \lVert  \frac{A_2 -1}{\sinh r} \right \rVert_{L^p} \lesssim \left \lVert  \psi^+  \right \rVert_{L^p}.
\end{equation}
(ii)\ \ Given $\epsilon >0$, and $R$ such that $\left \lVert \psi^+ \right \rVert_{L^2(\R \backslash[R^{-1},R])} \leq \epsilon$, then we have
\begin{equation}             \label{construct psi2 A2 from psi+ control 2}
\left \lVert \psi_2 \right \rVert_{\dot{H}^1_e(\R \backslash[\epsilon R^{-1},R])}+\left \lVert A_2 -1 \right \rVert_{\dot{H}^1_e(\R \backslash[\epsilon R^{-1},R])}  \leq \epsilon.
\end{equation}
(iii)\ \ If $(\tilde{A}_2,\ \tilde{\psi}_2)$ is another solution to (\ref{system of A2 psi2}) corresponding to $\tilde{\psi}^+$, then
\begin{equation}            \label{construct psi2 A2 from psi+ control 3}
\left\lVert\psi^--\tilde{\psi}^- \right\rVert_{L^2}+\left\lVert \psi_2-\tilde{\psi}_2\right\rVert_{\dot{H}^1_e}+\left\lVert A_2-\tilde{A}_2\right\rVert_{\dot{H}^1_e}\lesssim \left\lVert \psi^+-\tilde{\psi}^+\right\rVert_{L^2}.
\end{equation}
(iv)\ \ If $R_+\psi^+\in H^s$, then $R_-\psi^-\in H^s$ for $s\in\{1,\ 2\}$.
\end{lemma}

                                                            %Lemma 4.1 µÄÖ¤Ã÷¡£
\begin{proof}
We consider the ODE system (\ref{system of A2 psi2}) with boundary condition
$$\lim\limits _{r \rightarrow \infty} A_2(r)=1,\ \ \lim\limits _{r\rightarrow \infty} \psi_2 (r)=0.$$
The system and boundary condition imply $A_2^2+|\psi_2|^2=1$. We define $\psi^-= \psi^+-2i \frac{\psi_2}{\sinh r}$, $\psi_1=\psi^+-i\frac{\psi_2}{\sinh r}$, then we get $\partial_r A_2= \frac{1}{4}\sinh r(|\psi^+|^2-|\psi^-|^2)$ from (\ref{system of A2 psi2}). Since $A_2(\infty)=1$ which yields by integration from infinity
\begin{equation*}
A_2-1= \frac{1}{4}\int_r^{\infty} (|\psi^-|^2-|\psi^+|^2)\sinh s ds.
\end{equation*}
Thus we have $A_2>1-\frac{1}{4}\left\lVert\psi^+\right\rVert_{L^2}^2>0$.

                                                     %ÏÈÖ¤ÔÚ[R,ÎÞÇî) ÉÏ½âµÄ´æÔÚ£¬Î¨Ò»¡£
To prove existence, by choosing $R$ large enough such that $\left \lVert  \psi^+  \right \rVert_{L^2(R,\infty)} \leq \epsilon$. We want to seek $\psi_2$ with the property that $\left \lVert  \psi_2  \right \rVert_{\dot{H}^1_e(R,\infty)} \lesssim \epsilon$. This implies that $|\psi_2|\lesssim \epsilon$, then we have $A_2-1 \lesssim \epsilon^2$. By the relation $A_2^2+|\psi_2|^2=1$, we get $A_2=\sqrt{1-|\psi_2|^2}$. Now we only need to consider the $\psi_2$-equation in
\begin{equation*}
X=\{\psi_2\in \dot{H}^1_e: \mbox{ } \left \lVert  \psi_2  \right \rVert_{\dot{H}^1_e }\leq 2C\epsilon \}
\end{equation*}

Rewrite the $\psi_2$ equation as
\begin{equation*}
\partial_r \psi_2 =i\psi^+ +i(A_2 -1)\psi^+ +(A_2 -1)\frac{\psi_2}{\sinh r} + \frac{\psi_2}{\sinh r},
\end{equation*}
then
\begin{equation*}
\partial_r \psi_2 - \frac{\psi_2}{\sinh r} =i\psi^+ +i(A_2 -1)\psi^+ +(A_2 -1)\frac{\psi_2}{\sinh r},
\end{equation*}
Multiply by $e^{-\int ^r_{\infty} \sinh ^{-1}sds}$ on both sides, we have
\begin{equation*}
\partial_r (e^{-\int ^r_{\infty} \sinh ^{-1}sds} \psi_2) =e^{-\int ^r_{\infty} \sinh ^{-1}sds} \big(i\psi^+ +i(A_2 -1)\psi^+ +(A_2 -1)\frac{\psi_2}{\sinh r} \big),
\end{equation*}
Integrating from infinity we obtain
\begin{equation*}
\psi_2 (r)= \int_{\infty}^r e^{\int ^r_{\rho} \sinh ^{-1}sds  } \big( i\psi^+ +i(A_2 -1)\psi^+ +(A_2 -1)\frac{\psi_2}{\sinh \rho} \big) d \rho.
\end{equation*}
Define the map $\mathcal{T}: \mbox{ }\dot{H}^1_e(R,\infty)\rightarrow \dot{H}^1_e(R,\infty)$ by
\begin{equation*}
\mathcal{T}(\psi_2) (r)= \int_{\infty}^r e^{\int ^r_{\rho} \sinh ^{-1}sds  } \big( i\psi^+ +i(A_2 -1)\psi^+ +(A_2 -1)\frac{\psi_2}{\sinh \rho} \big) d \rho.
\end{equation*}
Now it suffices to show that $\mathcal{T}$ is a contraction map in $X$. Indeed, the estimate (\ref{basic inequality5}) and Sobolev embedding lead to
\begin{align*}
\left \lVert \mathcal{T} \psi_2 \right \rVert_{\dot{H}^1_e(R,\infty)}
\leq & \left \lVert i\psi^+ +i(A_2 -1)\psi^+ +(A_2 -1)\frac{\psi_2}{\sinh r} \right \rVert_{L^2(R,\infty)} \\
 & + \left \lVert \sinh^{-1}r \int_{\infty}^r e^{\int ^r_{\rho} \sinh ^{-1}sds  } ( i\psi^+ +i(A_2 -1)\psi^+ +(A_2 -1)\frac{\psi_2}{\sinh \rho}) d \rho.  \right \rVert_{L^2(R,\infty)}, \\
\leq &    \left \lVert \psi^+  \right \rVert_{L^2(R,\infty)} + C \left \lVert \psi^+  \right \rVert_{L^2(R,\infty)} + \left \lVert \frac{|\psi_2|^2}{A_2 +1} \frac{\psi_2}{\sinh r}  \right \rVert_{L^2(R,\infty)} \\
&        +C  \left \lVert i\psi^+ +i(A_2 -1)\psi^+ +(A_2 -1)\frac{\psi_2}{\sinh r} \right \rVert_{L^2(R,\infty)},\\
\leq &  C(\left \lVert \psi^+  \right \rVert_{L^2(R,\infty)} +\left \lVert \psi_2  \right \rVert_{L^{\infty}(R,\infty)}^2 \left \lVert \psi_2 \right \rVert_{\dot{H}^1_e(R,\infty)}),\\
\leq &  C(\left \lVert \psi^+  \right \rVert_{L^2(R,\infty)} +\left \lVert \psi_2  \right \rVert_{\dot{H}^1_e(R,\infty)}^3),\\
\leq & C(\epsilon+(2C \epsilon)^3),\\
<& 2C\epsilon.
\end{align*}
And the map $\mathcal{T}$ is Lipschitz with a small Lipschitz constant,
\begin{align*}
 &\left \lVert \mathcal{T}(\psi_2)-\mathcal{T}(\tilde{\psi}_2) \right \rVert_{\dot{H}^1_e(R,\infty)}\\
=& \left \lVert  \int_{\infty}^r e^{\int ^r_{\rho} \sinh ^{-1}sds  } (i(A_2 -1)\psi^+ +(A_2 -1)\frac{\psi_2}{\sinh \rho}  -i(\tilde{A}_2-1)\psi^+  -(\tilde{A}_2-1)\frac{\tilde{\psi}_2}{\sinh \rho}) d \rho.  \right \rVert_{\dot{H}^1_e(R,\infty)},\\
\leq & C \left \lVert i(A_2 -\tilde{A}_2)\psi^+ +(A_2 -\tilde{A}_2)\frac{\psi_2}{\sinh r} +(\tilde{A}_2-1)\frac{\psi_2-\tilde{\psi}_2}{\sinh r} \right \rVert_{L^2(R,\infty)},\\
\leq & C(\left \lVert \psi^+ \right \rVert_{L^2(R,\infty)} \left\lVert \frac{|\psi_2|+|\tilde{\psi}_2|}{A_2+\tilde{A}_2}\right\rVert_{L^{\infty}(R,\infty)} \left\lVert \psi_2-\tilde{\psi}_2\right\rVert_{L^{\infty}(R,\infty)}\\
& + \left \lVert \frac{\psi_2-\tilde{\psi}_2}{\sinh r}\right \rVert_{L^2(R,\infty)}  \left \lVert (|\psi_2|+|\tilde{\psi}_2|) \psi_2 \right \rVert_{L^{\infty}(R,\infty)}+ \left \lVert \tilde{\psi}_2 \right \rVert_{L^{\infty}(R,\infty)}^2  \left \lVert \frac{\psi_2-\tilde{\psi}_2}{\sinh r}\right \rVert_{L^2(R,\infty)}),\\
\leq & C (2C\epsilon^2 +2(2C \epsilon)^2 ) \left \lVert \psi_2-\tilde{\psi}_2\right \rVert_{\dot{H}^1_e(R,\infty)},\\
\leq & 2C\epsilon  \left \lVert \psi_2-\tilde{\psi}_2\right \rVert_{\dot{H}^1_e(R,\infty)}.
\end{align*}
Therefore there exists a unique solution $\psi_2 \in \dot{H}^1_e(R,\infty) $.

                                                  %°Ñ½âÑÓÍØµ½r=0
Next we extend the solution to $r=0$. Consider the equation $\partial_r \psi_2 =iA_2 \psi^+ +A_2 \frac{\psi_2}{\sinh r}$ with data $\psi_2(R)$. By Duhamel formula, it suffices to consider the map
\begin{equation*}
\mathcal{J}(\psi_2)(r)=e^{\int_R^r \sinh^{-1}sds} \psi_2(R)+\int_R^r e^{\int_{\rho}^r  \sinh^{-1}s ds } (i\psi^+ +i(A_2 -1)\psi^+ +(A_2 -1)\frac{\psi_2}{\sinh \rho}) d\rho,
\end{equation*}
and the space
\begin{equation*}
Y=\{ \psi_2\in \dot{H}^1_e(R-a,R): \left \lVert \psi_2 \right \rVert _{\dot{H}^1_e(R-a,R)} <2C \epsilon \}.
\end{equation*}
Since $\frac{1}{\sinh r} e^{2\int_R^r \sinh^{-1}sds}$ is bounded, there exists $a=a(\epsilon,\mbox{ } \left \lVert \psi^+ \right \rVert_{L^2} )>0$ such that
$$ \left \lVert \frac{1}{\sinh r} e^{\int_R^r \sinh^{-1}sds} \psi_2(R) \right \rVert_{L^2(R-a,R)}<\frac{C\epsilon}{2}.$$
By (\ref{basic inequality5}), we obtain that
\begin{eqnarray*}
\left \lVert \mathcal{J}(\psi_2) \right \rVert_{\dot{H}^1_e(R-a,R)}
&\leq & \left \lVert \frac{1}{\sinh r} e^{\int_R^r \sinh^{-1}sds} \psi_2(R) \right \rVert_{L^2(R-a,R)}\\
&&  +  C\left \lVert i\psi^+ +i(A_2 -1)\psi^+ +(A_2 -1)\frac{\psi_2}{\sinh r} \right \rVert_{L^2(R-a,R)},\\
&\leq & \frac{C\epsilon}{2} +C\left \lVert \psi^+ \right \rVert_{L^2(R-a,R)}+ C\left \lVert \psi_2 \right \rVert_{L^{\infty}}^2\left \lVert \psi_2 \right \rVert_{\dot{H^1_e}(R-a,R)},\\
&\leq & \frac{C\epsilon}{2} +C\epsilon +(2C\epsilon)^3,\\
&\leq & 2C\epsilon.
\end{eqnarray*}
Meanwhile we have
\begin{eqnarray*}
&&\left \lVert \mathcal{J}(\psi_2)-\mathcal{J}(\tilde{\psi}_2) \right \rVert_{\dot{H}^1_e(R-a,R)} \\
&\leq  & C(\left \lVert i(A_2 -\tilde{A}_2)\psi^+ +(A_2 -\tilde{A}_2)\frac{\psi_2}{\sinh r} +(\tilde{A}_2-1)\frac{\psi_2- \tilde{\psi}_2}{\sinh r} \right \rVert_{L^2(R,\infty)},\\
&\leq & C\left \lVert \psi_2-\tilde{\psi}_2 \right \rVert_{\dot{H}^1_e(R-a,R)}(\left \lVert \psi^+ \right \rVert_{L^2(R-a,R)}+\left \lVert \psi_2 \right \rVert_{\dot{H}^1_e(R-a,R)}+\left \lVert \tilde{\psi}_2 \right \rVert_{\dot{H}^1_e(R,\infty)}^2,\\
&\leq & C(\epsilon+3C\epsilon)\left \lVert \psi_2-\tilde{\psi}_2 \right \rVert_{\dot{H}^1_e(R-a,R)},\\
&\leq & C' \epsilon \left \lVert \psi_2-\tilde{\psi}_2 \right \rVert_{\dot{H}^1_e(R-a,R)}.
\end{eqnarray*}
Therefore $\mathcal{J}$ is a contraction map in $Y$. Since the lifespan interval $a$ only depends on $\epsilon$ and $\left \lVert \psi^+ \right \rVert_{L^2}$, we can extend the solution to $r=0$. Thus the existence of $\psi_2$ in $[0,\infty)$ follows, and the $A_2$ is obtained by $A_2(r)=\sqrt{1-|\psi_2|^2}$.

Next we obtain the bound for (\ref{construct psi2 A2 from psi+ control}). Let $G=\frac{\psi_2}{1+A_2}$, then the system gives
\begin{equation*}
\frac{d}{dr}|G|^2=2\Re(\frac{i\psi^+}{1+A_2}\bar{G})+\frac{2|G|^2}{\sinh r},
\end{equation*}
or equivalently
\begin{equation*}
\frac{d}{dr}|G|^2-\frac{2}{\sinh r} |G|^2 =-2 \Im (\frac{\psi^+}{1+A_2}\bar{G}),
\end{equation*}
which implies
\begin{equation*}
\Big|\frac{d}{dr}|G|-\frac{|G|}{\sinh r} \Big| \leq  \frac{|\psi^+|}{1+A_2},
\end{equation*}
namely,
\begin{equation*}
\Big| \frac{d}{dr} ( e^{-\int ^r_{\infty} \sinh^{-1}s ds} |G|) \Big| \leq e^{-\int ^r_{\infty} \sinh^{-1}s ds} \frac{|\psi^+|}{1+A_2},
\end{equation*}
therefore
\begin{equation*}
|G|(r)\leq \int_r^{\infty} e^{\int ^r_{\rho} \sinh^{-1}s ds} \frac{|\psi^+|}{1+A_2}(\rho) d\rho.
\end{equation*}
Since $|A_2|\leq1$, we get
\begin{equation}                  \label{psi2 pointwise control by psi+}
|\psi_2|(r)\lesssim \int_r^{\infty} e^{\int ^r_{\rho} \sinh^{-1}s ds} \frac{|\psi^+|}{1+A_2}(\rho) d\rho.
\end{equation}
By (\ref{basic inequality5}) and $A_2>0$, (\ref{psi2 pointwise control by psi+}) gives $\left \lVert \frac{\psi_2}{\sinh r} \right \rVert_{L^2} \lesssim \left \lVert \psi^+ \right \rVert_{L^2}$. The bounds for $\left \lVert \partial_r \psi_2  \right \rVert_{L^2}$ and $\left \lVert \partial_r A_2 \right \rVert_{L^2}$ follow directly from (\ref{system of A2 psi2}). The bounds for $\left \lVert \frac{A_2-1}{\sinh r} \right \rVert_{L^2}$ and $\left \lVert \frac{A_2-1}{\sinh r} \right \rVert_{L^1(dr)}$ are obtained by the compatibility relation $A_2^2+|\psi_2|^2=1$.

Now we prove the additional properties (i)-(iv). First, we have the bound for (\ref{construct psi2 A2 from psi+ control 1}). If $\psi^+\in L^p$, by (\ref{basic inequality5}) and (\ref{psi2 pointwise control by psi+}), we obtain $\left\lVert \frac{\psi_2}{sinh r}\right\rVert_{L^p}\lesssim \left\lVert \psi^+\right\rVert_{L^p}$, then the $L^p$-bound for $\psi^-$ and $\frac{A_2 -1}{\sinh r}$ are obtained immediately by the definition of $\psi^-$ and $A_2^2+|\psi_2|^2=1$.

Second, we obtain (\ref{construct psi2 A2 from psi+ control 2}). By (\ref{system of A2 psi2}) and $A_2^2+|\psi_2|^2=1$, we have
$$\left \lVert A_2-1 \right \rVert_{\dot{H}^1_e(\R \backslash [\epsilon R^{-1},R])}+\left \lVert \partial_r \psi_2 \right \rVert_{L^2(\R \backslash [\epsilon R^{-1},R])} \lesssim \epsilon + \left \lVert \frac{\psi_2}{\sinh r} \right \rVert_{L^2(\R \backslash [\epsilon R^{-1},R])}.$$
It suffices to get the $L^2$-bound for $\frac{\psi_2}{\sinh r}$. From (\ref{psi2 pointwise control by psi+}), we have
\begin{equation*}
|\psi_2| \lesssim \int_r^{\infty} e^{\int_{\rho}^r \sinh^{-1}sds}  \textbf{1}_{(0,R^{-1}]} |\psi^+|d\rho + \int_r^{\infty} e^{\int_{\rho}^r \sinh^{-1}sds}  \textbf{1}_{(R^{-1},\infty)} |\psi^+|d\rho.
\end{equation*}
For the first term we use (\ref{basic inequality5}) and the smallness of $\left\lVert\psi^+\right\rVert_{L^2(0,R^{-1})}$. For the second term, by H\"{o}lder inequality, we have
\begin{align*}
& \left \lVert \frac{1}{\sinh r} \int_r^{\infty} e^{\int_{\rho}^r \sinh^{-1}sds}  \textbf{1}_{(R^{-1},\infty)} |\psi^+|d\rho \right \rVert_{L^2(0,\epsilon R^{-1})}\\
\leq & \left \lVert \frac{1}{\sinh r} \big( \int_{R^{-1}}^{\infty} | e^{\int_{\rho}^r \sinh^{-1}sds} \sinh^{-1} \rho |^2  \sinh \rho d\rho  \big)^{\frac{1}{2}} \left \lVert \psi^+ \right \rVert_{L^2}  \right \rVert_{L^2(0,\epsilon R^{-1})},\\
= & \big(\int_0^{\epsilon R^{-1}} \frac{1}{\sinh r} \int_{R^{-1}}^{\infty} e^{2\int_{\rho}^r \sinh^{-1}s ds } \sinh^{-1} \rho d\rho \ \  dr \big)^{\frac{1}{2}}  \left \lVert \psi^+ \right \rVert_{L^2},\\
\lesssim & \epsilon.
\end{align*}
Then by (\ref{basic inequality5}), we easily obtain
\begin{equation*}
\left \lVert \frac{\psi_2}{\sinh r} \right \rVert_{L^2(R,\infty)} \leq \left \lVert \frac{1}{\sinh r} \int_r^{\infty} e^{\int_{\rho}^r \sinh^{-1}s ds} \textbf{1}_{[R,\infty)} |\psi^+| d\rho \right \rVert_{L^2} \lesssim \epsilon.
\end{equation*}
Thus the $L^2$-bound follows.

                                        %Ö¤Ã÷ÒýÀíµÄ£¨iii£©
Third, we get the Lipschitz continuity (\ref{construct psi2 A2 from psi+ control 3}). For notational convenience we denote
\begin{equation*}
\delta \psi^+=\psi^+-\tilde{\psi}^+,\ \ \delta\psi_2=\psi_2-\tilde{\psi}_2,\ \ \delta A_2=A_2-\tilde{A}_2.
\end{equation*}
Without any restriction in generality, we can make the assumption $\left\lVert \delta\psi^+\right\rVert_{L^2}\ll 1$ and the bootstrap assumption
$$\left\lVert \delta\psi_2\right\rVert_{L^{\infty}}+\left\lVert \delta A_2\right\rVert_{L^{\infty}}\lesssim \left\lVert \delta \psi^+\right\rVert_{L^2}^{\frac{1}{2}}.$$
By (\ref{system of A2 psi2}) and $A_2^2+|\psi_2|^2=1$, we derive the equations
\begin{equation*}
\left\{\begin{aligned}
&\partial_r \delta\psi_2=\frac{\delta\psi_2}{\sinh r}+i\tilde{\psi}^+\delta A_2+\frac{A_2-1}{\sinh r}\delta\psi_2+\frac{\tilde{\psi}_2}{\sinh r}\delta A_2+iA_2\delta\psi^+,\\
&\partial_r \delta A_2=\frac{2\delta A_2}{\sinh r}+\Im(\psi^+ \overline{\delta\psi_2})+\frac{2(\tilde{A}_2-1)}{\sinh r}\delta A_2+\Im(\delta\psi^+\bar{\tilde{\psi}}_2)+\frac{(\delta A_2)^2}{\sinh r}.
\end{aligned}
\right.
\end{equation*}
Since $\left\lVert \delta\psi^+\right\rVert_{L^2}\ll 1$ and $\frac{(\delta A_2)^2}{\sinh r}$ is a high order term, $ \Im(\delta\psi^+\bar{\tilde{\psi}}_2)+\frac{(\delta A_2)^2}{\sinh r}$ and $iA_2\delta\psi^+$ can be regarded as error terms. Let $X=(\Re \delta\psi^+,\ \Im\delta\psi^+,\ \delta A_2)^T$, we have
\begin{equation}          \label{X-equation}
\partial_r X=\frac{1}{\sinh r}LX+BX+E,
\end{equation}
where $L={\rm diag}\{1,\ 1,\ 2\}$,
\begin{equation*}
B=\left(\begin{matrix}
\frac{A_2-1}{\sinh r} & 0 & -\Im\tilde{\psi}^++\frac{\Re \tilde{\psi}_2}{\sinh r} \\
0 & \frac{A_2-1}{\sinh r} &  \Re\tilde{\psi}^++\frac{\Im \tilde{\psi}_2}{\sinh r} \\
 \Im\psi^+  & \Re\psi^+ & \frac{2(\tilde{A}_2-1)}{\sinh r}
\end{matrix}
\right),\ \
E=\left(\begin{matrix}
\Re(iA_2\delta\psi^+)\\
\Im(iA_2\delta\psi^+)\\
 \Im(\delta\psi^+\bar{\tilde{\psi}}_2)+\frac{(\delta A_2)^2}{\sinh r}
\end{matrix}
\right).
\end{equation*}
From (\ref{construct psi2 A2 from psi+ control 1}) we obtain the $L^2$-norm of $B$ is bounded. Then we decompose $B=B_1+B_2:=B\textbf{1}_{\geq \epsilon}(r)+B\textbf{1}_{<\epsilon}(r)$ for small $\epsilon$. By the $L^2$-bound for $B$, we have $|rB|\rightarrow 0,\ as\ r\rightarrow 0$, which gives $|B_2|\ll \frac{1}{r}$ in $(0,\ \epsilon)$. We also easily obtain $\left\lVert B_1\right\rVert_{L^1(dr)}\lesssim |\log\epsilon|^{1/2}\left\lVert B\right\rVert_{L^2}$ by H\"{o}lder inequality. Then we can construct the bounded matrix $e^{-\int_{\infty}^r B_1ds}$ such that $\partial_r e^{-\int_{\infty}^r B_1ds}=-e^{-\int_{\infty}^r B_1ds}B_1$. Hence (\ref{X-equation}) can be written as
\begin{equation*}
\partial_r(e^{-\int_{\infty}^r B_1ds}X)=\frac{1}{\sinh r}L(e^{-\int_{\infty}^r B_1ds}X)+e^{-\int_{\infty}^r B_1ds}(B_2X+E),
\end{equation*}
then
\begin{equation*}
X=e^{\int_{\infty}^r B_1ds}\int_{\infty}^r {\rm diag}(e^{\int_{\rho}^r \sinh^{-1}sds},\ e^{\int_{\rho}^r \sinh^{-1}sds},\ e^{2\int_{\rho}^r \sinh^{-1}sds}) e^{-\int_{\infty}^r B_1ds}(B_2X+E)d\rho.
\end{equation*}
By the above expression of $X$ and (\ref{basic inequality5}), we have
\begin{equation*}
\begin{aligned}
\left\lVert X\right\rVert_{L^{\infty}} \lesssim & \left\lVert B_2X\right\rVert_{L^1(dr)}+ \left\lVert \delta\psi^+ \right\rVert_{L^2}+\left\lVert \frac{X}{\sinh r}\right\rVert_{L^2}^2,\\
\lesssim & \left\lVert \delta\psi^+ \right\rVert_{L^2}+\left\lVert \frac{X}{\sinh r}\right\rVert_{L^2}+\left\lVert \frac{X}{\sinh r}\right\rVert_{L^2}^2,
\end{aligned}
\end{equation*}
and
\begin{equation*}
\begin{aligned}
\left\lVert \frac{X}{\sinh r}\right\rVert_{L^2} \lesssim  &\left\lVert B_2\sinh r\right\rVert_{L^{\infty}}\left\lVert \frac{X}{\sinh r}\right\rVert_{L^2}+\left\lVert \delta\psi^+\right\rVert_{L^2}+\left\lVert \delta A_2\right\rVert_{L^{\infty}}\left\lVert \frac{\delta A_2}{\sinh r}\right\rVert_{L^2},\\
\leq & c\left\lVert \frac{X}{\sinh r}\right\rVert_{L^2}+C(\left\lVert \delta\psi^+\right\rVert_{L^2}+\left\lVert \frac{X}{\sinh r}\right\rVert_{L^2}^2+\left\lVert \frac{X}{\sinh r}\right\rVert_{L^2}^3).
\end{aligned}
\end{equation*}
where $c\ll 1$. Hence, $\left\lVert \frac{X}{\sinh r}\right\rVert_{L^2}\lesssim \left\lVert \delta\psi^+\right\rVert_{L^2}$ and $\left\lVert X\right\rVert_{L^{\infty}}\lesssim \left\lVert \delta\psi^+\right\rVert_{L^2}$. Furthermore, by (\ref{X-equation}) we have $\left\lVert \partial_r X\right\rVert_{L^2}\lesssim \left\lVert \delta\psi^+\right\rVert_{L^2}$.

                                          %Ö¤Ã÷ÒýÀíµÄ£¨iv£©¡£
Finally we prove (iv). If $s=1$, by (\ref{system of A2 psi2}), we have
\begin{equation*}
\begin{aligned}
\partial_r \psi^-= &\partial_r \psi^+-2i\partial_r \frac{\psi_2}{\sinh r},\\
= &\partial_r \psi^++ 2A_2\frac{\psi^+}{\sinh r}-2i\frac{A_2-1}{\sinh r} \frac{\psi_2}{\sinh r}+2i\frac{\cosh r-1}{\sinh r} \frac{\psi_2}{\sinh r}
\end{aligned}
\end{equation*}
then by (\ref{construct psi2 A2 from psi+ control}), (\ref{construct psi2 A2 from psi+ control 1}) and Gagliardo-Nirenberg inequality, we obtain
\begin{equation*}
\begin{aligned}
\left\lVert \partial_r \psi^-\right\rVert_{L^2} \lesssim & \left\lVert \partial_r\psi^+\right\rVert_{L^2} + \left\lVert \frac{\psi^+}{\sinh r}\right\rVert_{L^2} +\left\lVert  \frac{\psi_2}{\sinh r}\right\rVert_{L^4}^2+\left\lVert  \frac{\psi_2}{\sinh r}\right\rVert_{L^2},\\
\lesssim & \left\lVert R_+\psi^+\right\rVert_{H^1}+ \left\lVert \psi^+\right\rVert_{L^4}^2,\\
\lesssim & \left\lVert R_+\psi^+\right\rVert_{H^1}(1+\left\lVert \psi^+\right\rVert_{L^2}).
\end{aligned}
\end{equation*}

If $s=2$, by (\ref{system of A2 psi2}) and $A_2^2+|\psi_2|^2=1$, we have
\begin{equation*}
\begin{aligned}
\partial_{rr}\psi^-= & \partial_{rr}\psi^+ + (\Im(\psi^+ \frac{\bar{\psi}_2}{\sinh r})-|\frac{\psi_2}{\sinh r}|^2)(2\psi^+-2i\frac{\psi_2}{\sinh r})\\
& +2A_2\frac{\cosh r}{\sinh r}\partial_r \psi^+ -2A_2\frac{\psi^+}{\sinh^2 r}+2A_2\frac{1-\cosh r}{\sinh r}\partial_r \psi^+ -2A_2 \frac{\cosh r -1}{\sinh^2 r}\psi^+ \\
& + (iA_2 \psi^++A_2\frac{\psi_2}{\sinh r})(-2i\frac{1}{A_2+1}|\frac{\psi_2}{\sinh r}|^2+2i\frac{\cosh r-1}{\sinh^2 r})\\
& + 4i\frac{1}{A_2+1}|\frac{\psi_2}{\sinh r}|^2 \frac{\psi_2}{\sinh r}  +4i \frac{\cosh r-1}{\sinh^2 r}(A_2 -1)\frac{\psi_2}{\sinh r}-2i\frac{(\cosh r-1)^2}{\sinh^2 r}\frac{\psi_2}{\sinh r}.
\end{aligned}
\end{equation*}
by Sobolev embedding $H^2\hookrightarrow L^6$, we get $\partial_{rr}\psi^-\in L^2$. Similarly, $\frac{\cosh r}{\sinh r}\partial_r \psi^-\in L^2$. Hence, $R_-\psi^-\in H^2$.
\end{proof}

\begin{proposition}        \label{construct u proposition}                     %ÃüÌâ   ´Ópsi+µÃµ½ Ó³Éä u
Given $\psi^+\in L^2$ with $\left\lVert \psi^+\right\rVert_{L^2}<2$, then there is a unique map $u: \H^2\longrightarrow \S^2$ with the property that $\psi^+$ is the representation of $\mathcal{V}^+$ relative to a Coulomb gauge satisfying (\ref{definition of V-}) with $E(u)=\pi \left \lVert \psi^+ \right \rVert _{L^2}^2$. Moreover, the map $\psi^+ \rightarrow u$ is Lipschitz continuous in the following sense:
\begin{equation*}
\left \lVert u-\tilde{u} \right \rVert _{\dot{H}^1}\lesssim \left \lVert \psi^+-\tilde{\psi}^+ \right \rVert _{L^2}.
\end{equation*}
\end{proposition}
                                                         %Ö¤Ã÷ÃüÌâ
\begin{proof}
Given $\left \lVert \psi^+ \right \rVert _{L^2}< 2$, by Lemma \ref{construct psi2 A2 from psi+}, there is a unique solution $(A_2,\ \psi_2)$. Let $\psi_1=\psi^+-i\frac{\psi_2}{\sinh r}$. Now we solve the system of $U=(\bar{u},\bar{v},\bar{w})^T$, that is
\begin{equation}          \label{system of bar u,v,w}
\left\{
\begin{aligned}
\partial_r \left( \begin{matrix}
\bar{u}\\
\bar{v}\\
\bar{w}
\end{matrix}
\right)
=M \left( \begin{matrix}
\bar{u}\\
\bar{v}\\
\bar{w}
\end{matrix}
\right)
:= \left( \begin{matrix}
0& \Re \psi_1 & \Im \psi_1\\
-\Re \psi_1 & 0 & 0\\
-\Im\psi_1 &\ 0 &\  0
\end{matrix}
\right)
\left( \begin{matrix}
\bar{u}\\
\bar{v}\\
\bar{w}
\end{matrix}
\right)\\
(\bar{u}(\infty),\ \bar{v}(\infty),\ \bar{w}(\infty))=(\vec{k},\ \vec{i},\ \vec{j}).
\end{aligned}
\right.
\end{equation}
Since $\psi_1=-(A_2 -1)\psi_1 -i\partial_r \psi_2$, $M$ can be rewritten as $M=M_1+\partial_r M_2$, where
\begin{equation}              \label{M1 and M2}
M_1 =\left( \begin{matrix}
0 & -\Re (A_2 -1)\psi_1 & -\Im (A_2 -1)\psi_1 \\
 \Re (A_2 -1)\psi_1 & 0 & 0 \\
 \Im (A_2 -1)\psi_1  & 0 & 0
\end{matrix}
\right), \mbox{ }
M_2 =\left( \begin{matrix}
0 &  \Im \psi_2 & -\Re \psi_2 \\
-\Im \psi_2 & 0 & 0 \\
 \Re \psi_2  & 0 & 0
\end{matrix}
\right),
\end{equation}
and by (\ref{construct psi2 A2 from psi+ control}), $\left \lVert M_1 \right \rVert _{L^1 (dr)} +\left \lVert M_2 \right \rVert _{\dot{H}^1_e} \lesssim \left \lVert \psi^+ \right \rVert _{L^2}$. If we restrict $M_1$ and $M_2$ on $[R,\infty)$ for sufficiently large $R$, we can assume $\left \lVert M_1 \right \rVert _{L^1 (dr)(R,\infty)} +\left \lVert M_2 \right \rVert _{\dot{H}^1_e(R,\infty)} \lesssim \epsilon$. This allow us to construct solutions with data at $r=\infty$ by using the iteration scheme
\begin{equation*}
U=\sum \limits_i U_i, \mbox{ } U_0=U(\infty) ,\mbox{ } U_i(r)=\int_{\infty}^r M(s) U_{i-1}ds.
\end{equation*}
Let $X=\{ U\in C([R,\infty)): \partial_r U\in L^2([R,\infty)), \lim \limits_{r\rightarrow \infty} U(r)\ \ {\rm exists} \}$. We run the iteration scheme in $X$. For $U_{i-1}\in X$, we have
\begin{equation*}
\begin{aligned}
U_i(r) =& \int_{\infty}^r M(s) U_{i-1}(s)ds,\\
= & \int_{\infty}^r M_1 U_{i-1} +\partial_r M_2 U_{i-1} ds,\\
= & \int_{\infty}^r M_1 U_{i-1}ds +M_2(r)U_{i-1}(r) -\int_{\infty}^r M_2 \partial_r U_{i-1}ds,
\end{aligned}
\end{equation*}
from which we obtain
\begin{equation*}
\begin{aligned}
\left \lVert U_i \right \rVert_{C([R,\infty))}  \leq &  \left \lVert M_1 \right \rVert_{L^1(dr)(R,\infty)} \left \lVert U_{i-1} \right \rVert_{C([R,\infty))}\\
 & +\left \lVert M_2 \right \rVert_{C([R,\infty))} \left \lVert U_{i-1} \right \rVert_{C([R,\infty))}+\left \lVert \frac{M_2}{\sinh r} \right \rVert_{L^2(R,\infty)} \left \lVert \partial_r U_{i-1} \right \rVert_{L^2(R,\infty)},\\
\lesssim & \left \lVert \psi^+ \right \rVert_{L^2(R,\infty)} (\left \lVert U_{i-1} \right \rVert_{C([R,\infty))}+ \left \lVert \partial_r U_{i-1} \right \rVert_{L^2(R,\infty)}),\\
\left \lVert \partial_r U_i \right \rVert_{L^2(R,\infty)} =& \left \lVert M_1 U_{i-1} +\partial_r M_2 U_{i-1} \right \rVert_{L^2(R,\infty)},\\
\leq & \left \lVert M_1 \right \rVert_{L^2(R,\infty)} \left \lVert U_{i-1} \right \rVert_{C([R,\infty))} + \left \lVert \partial_r M_2 \right \rVert_{L^2(R,\infty)} \left \lVert U_{i-1} \right \rVert_{C([R,\infty))},\\
\lesssim & \left \lVert \psi^+ \right \rVert_{L^2(R,\infty)} \left \lVert U_{i-1} \right \rVert_{C([R,\infty))}.
\end{aligned}
\end{equation*}
Therefore,
\begin{equation*}
\left \lVert U_i \right \rVert_{C([R,\infty))}+\left \lVert \partial_r U_i \right \rVert_{L^2([R,\infty))} \lesssim \left \lVert \psi^+ \right \rVert_{L^2([R,\infty))}^i.
\end{equation*}
Then by choosing $R$ large enough, we can use the iteration scheme to construct a solution $U$ on $[R,\infty)$.

The uniqueness of (\ref{system of bar u,v,w}) is obtained by conservation law, that is, apply $( \bar{u},\bar{v},\bar{w})  $ to both side of (\ref{system of bar u,v,w}), we have $\frac{1}{2} \partial_r ( |\bar{u}| ^2+ |\bar{v}| ^2+ |\bar{w}| ^2) =0$.

Since $\partial_r \bar{u}\cdot  \bar{v} =- \partial_r \bar{v}\cdot  \bar{u}$ by (\ref{system of bar u,v,w}), we have $\bar{u}\cdot  \bar{v}(r)= {\rm const}$, which together with $\lim\limits_{r\rightarrow\infty}U(r)=U(\infty)=(\vec{k},\vec{i},\vec{j})^T$ yields $\bar{u}\cdot\bar{v}(r)=\bar{u}\cdot  \bar{v}(\infty)=0$. Similarly, we also have $\bar{u}\cdot  \bar{w} =\bar{v}\cdot  \bar{w}=0$ and $|\bar{v}| =|\bar{w}| =1$. Thus $U$ satisfies the orthonormality condition.

Next, the solution constructed above can be extended to $(0,\infty)$. Since $\left \lVert \psi^+ \right \rVert_{L^2}<2$, then for any $\epsilon>0$, there exists $\delta>0$ such that $\left \lVert \psi^+ \right \rVert_{L^2(R-\delta,R)}<\epsilon$. Define $U_i(r)= \int_R^r MU_{i-1}ds$, denote $I=[R-\delta,\ R]$, we have
\begin{equation*}
\left \lVert U_i \right \rVert_{C(I)} \leq \left \lVert M_1 \right \rVert_{L^1(dr)(I)}\left \lVert U_{i-1} \right \rVert_{C(I)}+\left \lVert \partial_r M_2 \right \rVert_{L^2(I)}(\int_R^{R-\delta} \sh^{-1}s ds)^{1/2}\left \lVert U_{i-1} \right \rVert_{C(I)},
\end{equation*}
and
\begin{equation*}
\left \lVert \partial_r  U_i \right \rVert_{L^2(I)} \leq (\left \lVert M_1 \right \rVert_{L^2(I)}+\left \lVert \partial_r M_2 \right \rVert_{L^2(I)})\left \lVert U_{i-1} \right \rVert_{C(I)}.
\end{equation*}
By (\ref{M1 and M2}), we have $\left \lVert M_1 \right \rVert_{L^1(dr)(I)}+\left \lVert M_1 \right \rVert_{L^2(I)} + \left \lVert \partial_r M_2 \right \rVert_{L^2(I)} \lesssim \left \lVert \psi^+ \right \rVert_{L^2(I)}$, therefore,
\begin{eqnarray*}
\left \lVert U_i \right \rVert_{C(I)}+\left \lVert \partial_r U_i \right \rVert_{L^2(I)}
 & \lesssim & \left \lVert \psi^+ \right \rVert_{L^2(I)} \Big(1+(\int_R^{R-\delta} \sinh^{-1}s ds)^{1/2} \Big)  \Big(\left \lVert U_{i-1} \right \rVert_{C(I)}+\left \lVert \partial_r U_{i-1} \right \rVert_{L^2(I)} \Big),\\
 &\lesssim & \left \lVert \psi^+ \right \rVert_{L^2(I)}^i.
\end{eqnarray*}
By choosing $\delta$ small such that $(\int_R^{R-\delta} \sinh^{-1}s ds)^{1/2}$ is small, then we can still rely on iteration scheme to extend the solution to $[\epsilon R^{-1},\infty)$.

On the interval $(0,\epsilon R^{-1}]$, by (\ref{construct psi2 A2 from psi+ control 2}), we have $\left \lVert \psi_2 \right \rVert_{C((0,\epsilon R^{-1}])} \lesssim \left \lVert \psi_2 \right \rVert_{\dot{H}^1_e((0,\epsilon R^{-1}])}\lesssim \epsilon$. By a similar argument to that on $[R,\infty)$, we extend the solution to $r=0$. As a byproduct,
\begin{equation*}
\left \lVert U-U_0 \right \rVert_{C(0,\infty)}+\left \lVert \partial_r U \right \rVert_{L^2(0,\infty)} \lesssim \left \lVert \psi^+ \right \rVert_{L^2}.
\end{equation*}

From the system (\ref{system of bar u,v,w}), we know that $\bar{u}_3$ and $q= \bar{w}_3-i\bar{v}_3$ solve the system
\begin{equation*}
\left\{
\begin{aligned}
&\partial_r q =i\bar{u}_3 \psi_1, \\
&\partial_r \bar{u}_3 = \Im (\psi_1 \bar{q}).
\end{aligned}
\right.
\end{equation*}
with boundary condition $(\bar{u}_3,q)(\infty)=(1,0)$. By uniqueness, $A_2=\bar{u}_3$, $\psi_2=q$.

Next, we construct the system of $(u,v,w)$ by equivariant setup, that is, apply $(\bar{u},\bar{v},\bar{w})$ by $e^{\theta R}$. From $\psi_2=  \bar{w}_3-i\bar{v}_3$ and the orthonormality condition, (\ref{system of u,v,w}) is satisfied for $k=2$.

Given $\psi^+, \tilde{\psi}^+ \in L^2$, we construct $U$ and $\tilde{U}$ as above. From the construction it follows that
\begin{equation}               \label{U-tildeU by psi+}
\left \lVert U-\tilde{U} \right \rVert_{C(0,\infty)}+\left \lVert \partial_r (U-\tilde{U}) \right \rVert_{L^2(0,\infty)} \lesssim \left \lVert \psi^+ -\tilde{\psi}^+ \right \rVert_{L^2}.
\end{equation}
which implies $\left \lVert \partial_r (u-\tilde{u}) \right \rVert_{L^2(0,\infty)} \lesssim \left \lVert \psi^+ -\tilde{\psi}^+ \right \rVert_{L^2}$. Since $u_1= v_2 w_3 -v_3 w_2$, $\tilde{u}_1= \tilde{v}_2 \tilde{w}_3 -\tilde{v}_3 \tilde{w}_2$, $\psi_2= w_3-iv_3$, $\tilde{\psi}_2=  \tilde{w}_3 -i \tilde{v}_3$, by (\ref{U-tildeU by psi+}), we have
\begin{equation*}
\begin{aligned}
\left \lVert \frac{u_1-\tilde{u}_1}{\sinh r} \right \rVert_{L^2}
\leq & \left \lVert \frac{1}{\sinh r}[(v_2-\tilde{v}_2)\tilde{w}_3+v_2 (w_3 -\tilde{w}_3)] \right \rVert_{L^2}\\
&   +\left \lVert \frac{1}{\sinh r}[(v_3-\tilde{v}_3)w_2+\tilde{v}_3 (w_2 -\tilde{w}_2)] \right \rVert_{L^2}, \\
 \lesssim & \left \lVert U-\tilde{U} \right \rVert_{L^{\infty}} \left \lVert \frac{\tilde{\psi}_2}{\sinh r} \right \rVert_{L^2}+ \left \lVert U \right \rVert_{L^{\infty}}  \left \lVert \frac{\psi_2-\tilde{\psi}_2}{\sinh r} \right \rVert_{L^2} , \\
 \lesssim & \left \lVert \psi^+ -\tilde{\psi}^+ \right \rVert_{L^2}(\left \lVert \psi^+  \right \rVert_{L^2}+ \left \lVert \tilde{\psi}^+ \right \rVert_{L^2}).
\end{aligned}
\end{equation*}
A similar argument shows that $\left \lVert \frac{u_2-\tilde{u}_2}{\sinh r} \right \rVert_{L^2}+\left \lVert \frac{u_3-\tilde{u}_3}{\sinh r} \right \rVert_{L^2} \lesssim  \left \lVert \psi^+ -\tilde{\psi}^+ \right \rVert_{L^2}(\left \lVert \psi^+  \right \rVert_{L^2}+ \left \lVert \tilde{\psi}^+ \right \rVert_{L^2})$. Therefore, $\left \lVert u-\tilde{u} \right \rVert_{\dot{H}^1} \lesssim \left \lVert \psi^+-\tilde{\psi}^+ \right \rVert_{L^2}$.
\end{proof}

\section{The Cauchy problem}
In this section we concerned with the $(\psi^+,\ \psi^-)$-system which we recall here
\begin{equation}    \label{system}
\left\{
\begin{aligned}
&(i \partial_t +\Delta_{\H^2}-2 \frac{\cosh r -1}{\sinh ^2 r}) e^{i2 \theta} \psi^+ =   [A_0 + 2 \frac{\cosh r (A_2-1)}{\sinh^2 r}-\Im(\psi^+ \frac{\bar{\psi_2}}{\sinh r})] e^{i2\theta} \psi^+,  \\                %\psi^+ µÄ·½³Ì¡£
&(i \partial_t +\Delta_{\H^2} +2 \frac{\cosh r -1}{\sinh ^2 r}) \psi^- =   [A_0 -2\frac{\cosh r(A_2 -1)}{\sinh^2 r}+ \Im(\psi^- \frac{\bar{\psi_2}}{\sinh r})]\psi^-.   %\psi^- µÄ·½³Ì
\end{aligned}
\right.
\end{equation}
with initial data $\psi^{\pm}(t_0)=\psi^{\pm}_0$. Where $A_0$, $A_2$, $\psi_2$ are given by (\ref{A0 with psi+ psi-}), (\ref{A2 with psi+ psi-}), (\ref{psi+psi-}). Since the system (\ref{system}) arised from the Schr\"{o}dinger map (\ref{Schrodinger map}), we will show that $(\psi^+,\ \psi^-)$ satisfy the compatibility condition.

For simplicity of notations, we denote $\left\lVert f^{\pm}\right\rVert=\left\lVert f^+\right\rVert+\left\lVert f^-\right\rVert$. Since our analysis relies on $L_t^pL_x^q$-norm, we define the norm of $f$ by $\left\lVert f\right\rVert_{L_I^pL^q}:=\left\lVert f\right\rVert_{L^p(I;L^q)}$. Finally, we denote the nonlinearities by
\begin{equation*}
\begin{aligned}
&F^+(\psi^+)=[A_0 + 2 \frac{\cosh r (A_2-1)}{\sinh^2 r}-\Im(\psi^+ \frac{\bar{\psi_2}}{\sinh r})]e^{i2\theta}\psi^+,\\
&F^-(\psi^-)=[A_0 -2\frac{\cosh r(A_2 -1)}{\sinh^2 r}+ \Im(\psi^- \frac{\bar{\psi_2}}{\sinh r})]\psi^-.
\end{aligned}
\end{equation*}

\subsection{Srichartz estimates}

To understand the well-posedness of (\ref{system}), we need to obtain the Strichartz estimates. The $\psi^+$-equation in (\ref{system}) is a nonlinear Schr\"{o}dinger equation with positive and exponential decay potential. More generally, we consider the Schr\"{o}dinger equation
\begin{equation} \label{linear Schrodinger1}
\left\{
\begin{aligned}
&(i \partial_t +\Delta_{\H^2}-V)u =F, \\
&u(0)=f,
\end{aligned}
\right.
\end{equation}
where $V\in e^{-\alpha r}L^{\infty}(\H^2;\R)$ for $\alpha \geq 1$ is a positive potential. In this section we always denote potential $V$ as (\ref{linear Schrodinger1}). For simplicity, we denote $p'=\frac{p}{p-1}$ for $p\in[1,\ \infty]$. $(p,\ q)$ is called admissible pair, if
\begin{equation*}
(p,\ q) \in \Big\{(p,\ q)\in (2,\ \infty)\times(2,\ \infty):\frac{1}{p}+\frac{1}{q}=\frac{1}{2}\Big\} \cup \Big\{(\infty,\ 2)\Big\}.
\end{equation*}
Then we obtain the following Strichartz estimates.
\begin{theorem}        \label{Strichartz estimates}
Let $(p,\ q),\ (\tilde{p}',\ \tilde{q}')$ be admissible pairs, $I\subset \R$ be open interval.\\
(i)\ \ If $f\in L^2(\H^2)$, then
\begin{equation*}
\left\lVert e^{it(\Delta_{\H^2}-V)}f\right\rVert_{L^p_IL^q} \lesssim \left\lVert f\right\rVert_{L^2},
\end{equation*}
(ii)\ \ If $F\in L^{\tilde{p}}_IL^{\tilde{q}}$, then
\begin{equation*}
\left\lVert \int_0^t e^{i(t-s)(\Delta_{\H^2}-V)}F(s)ds\right\rVert_{L^p_IL^q} \lesssim \left\lVert F\right\rVert_{L^{\tilde{p}}_IL^{\tilde{q}}}.
\end{equation*}
\end{theorem}

Based on a standard theory, the above results are obtained by the following dispersive estimates immediately.

\begin{proposition}                    \label{dispersive estimate proposition}
Assume $V\in e^{-\alpha r}L^{\infty}(\H^2;\R),\ \alpha\geq 1$, is a positive potential, then we have
\begin{equation}                \label{dispersive estimate inequality}
||e^{it (\Delta_{\H^2} -V)} ||_{L^{1} \rightarrow L^{\infty}} \leq
\left\{
\begin{aligned}
C|t|^{-1}, &\ \ {\rm if}\ \mbox{ } 0<|t|<1, &\\
C|t|^{-\frac{3}{2}},&\ \ {\rm if}\  \mbox{ } |t|\geq 1. &
\end{aligned}
\right.
\end{equation}
\end{proposition}

By standard convention the resolvent of Laplacian $-\Delta_{\H^2}$ on $\H^2$ is written as $R_0(s)=(-\Delta_{\H^2}-s(1-s))^{-1}$ with $\Re s>\frac{1}{2}$ corresponding to the resolvent set $s(1-s)\in \C-[\frac{1}{4},\ \infty)$. The kernel of $R_0(s)$ is
\begin{equation} \label{R0}
R_0(s;z,w)=Q^0_{s-1}(\cosh r),
\end{equation}
where $Q^0_{s-1}$ is Legendre function, $r:=d(z,w)$. With the hyperbolic convention for spectral parameter, Stone's formula gives the continuous part of the spectral resolution as
\begin{equation*}
\begin{aligned}
d \Pi(\lambda) =& 2i\lambda [R_0(\frac{1}{2}+i\lambda )-R_0(\frac{1}{2}-i\lambda)]d\lambda,\\
 =& -4\lambda \Im R_0(\frac{1}{2}+i\lambda)d\lambda.
\end{aligned}
\end{equation*}
Then we use the spectral resolution to write
\begin{eqnarray*}
e^{it\Delta_{\H^2}}f(z) &=& \frac{e^{\frac{i}{4}}}{2\pi i} \int ^{+\infty}_0 \int _{\H^2} e^{it\lambda ^2} f(w) d\Pi (\lambda;z,w)dw d\lambda,\\
&=& \frac{e^{\frac{i}{4}}}{2\pi i} \int ^{+\infty}_0 \int _{\H^2} e^{it\lambda ^2}  \lambda \int ^{+\infty}_r \frac{\sin \lambda s }{\sqrt{\cosh s -\cosh r}} f(w) ds dw d \lambda.
\end{eqnarray*}
Similarly, from \cite{BoMa}, the resolvent of $-\Delta_{\H^2}+V$ for potential $V$ defined as above is given by $R_V(s)=(-\Delta_{\H^2}+V-s(1-s))^{-1}$ and the continuous component of spectral resolution is given by
\begin{equation*}
d\Pi _V (\lambda)=-4\lambda \Im R_V(\frac{1}{2}+i\lambda)d\lambda,
\end{equation*}
then the kernel of Schr\"{o}dinger propagator can be written as
\begin{equation*}
e^{it(\Delta_{\H^2}-V)}f(z) = \frac{e^{\frac{i}{4}}}{2\pi i} \int ^{+\infty}_0 \int _{\H^2} e^{it\lambda ^2} \lambda \Im R_V(\frac{1}{2}+i \lambda;z,w)f(w)dw d\lambda.
\end{equation*}
By Birman-Schwinger type resolvent expansion for all frequencies:
\begin{equation*}
R_V(s)=R_0(s)+R_0(s)[-VR_0(s)]+[R_0(s)V]R_V(s)[VR_0(s)],
\end{equation*}
we get
\begin{align}                                       \nonumber
 & e^{it(\Delta_{\H^2}-V)}f(z)\\                    \label{resolvent representation.1}
=& \frac{e^{\frac{i}{4}}}{2\pi i} \int ^{+\infty}_0 \int _{\H^2} e^{it\lambda ^2} \lambda \Im R_0(\frac{1}{2}+i \lambda;z,w)f(w)dw d\lambda\\                       \label{resolvent representation.2}
 & +\frac{e^{\frac{i}{4}}}{2\pi i} \int ^{+\infty}_0 \int _{\H^2} e^{it\lambda ^2} \lambda \Im R_0(\frac{1}{2}+i \lambda)[-V\Im R_0(\frac{1}{2}+i \lambda)]f(w)dw d\lambda\\                           \label{resolvent representation.3}
 & +\frac{e^{\frac{i}{4}}}{2\pi i}\int ^{+\infty}_0 \int _{\H^2} e^{it\lambda ^2} \lambda \Im [R_0(\frac{1}{2}+i \lambda)V]R_V(\frac{1}{2}+i \lambda)[V\Im R_0(\frac{1}{2}+i \lambda)]f(w)dw d\lambda.
\end{align}

Before proving Proposition \ref{dispersive estimate proposition}, we recall the pointwise bounds on the resolvent kernel from \cite{BoMa}. This bounds will be crucial for the dispersive estimates.
\begin{lemma}     \label{R_0 lambda <1}
For the free resolvent kernel the pointwise bounds are valid for $0\leq \lambda \leq 1$ and $r\in (0,\ \infty)$
\begin{eqnarray*}
|R_0(\frac{1}{2}+i\lambda;z,w)|\leq
\left\{
\begin{array}{lll}
C|\log r|, &r\leq 1,& \\
C\lambda^{-\frac{1}{2}}e^{-\frac{1}{2}r}, & r>1,&
\end{array}
\right.
\end{eqnarray*}
\begin{eqnarray*}
|\partial_{\lambda} R_0(\frac{1}{2}+i\lambda;z,w)|\leq
\left\{
\begin{array}{lll}
C|\log r|, &r\leq 1,& \\
C\lambda^{-\frac{1}{2}}e^{-(\frac{1}{2}-\epsilon) r}, & r>1,&
\end{array}
\right.
\end{eqnarray*}
where $r:=d(z,w)$.
\end{lemma}

\begin{lemma}  \label{R_0 lambda >1}
For the free resolvent kernel the pointweise bounds are valid for $  \lambda \geq 1$, and $r \in (0,\infty)$
\begin{eqnarray*}
|R_0(\frac{1}{2}+i\lambda;z,w)|\leq
\left\{
\begin{array}{lll}
C|\log r|, &\lambda r\leq 1,& \\
C\lambda^{-\frac{1}{2}}e^{-\frac{1}{2}r}, &\lambda r>1,&
\end{array}
\right.
\end{eqnarray*}
\begin{eqnarray*}
|\partial_{\lambda} R_0(\frac{1}{2}+i\lambda;z,w)|\leq
\left\{
\begin{array}{lll}
C|\log r|, &\lambda r\leq 1,& \\
C\lambda^{-\frac{1}{2}}e^{-(\frac{1}{2}-\epsilon)r}, &\lambda r>1,&
\end{array}
\right.
\end{eqnarray*}
where $r:=d(z,w)$.
\end{lemma}

We also recall the meromorphic continuation from \cite{BoMa}.
\begin{lemma}  \label{R_V}
For $V\in e^{-\alpha r}L^{\infty}(\H^2)$ with $\alpha>0$, the resolvent $R_V(s)$ admits a meromorphic continuation to the half-plane $\Re s>\frac{1}{2}-\delta$ as a bounded operator
\begin{equation*}
R_V(s): e^{-\delta r}L^2(\H^2) \rightarrow e^{\delta r}L^2(\H^2)
\end{equation*}
for $\delta <\frac{\alpha}{2}$. And there exists a constant $M_V$ such that for all $\lambda \in \R$ with $|\lambda|\geq M_V$,
\begin{equation*}
||e^{-\frac{\alpha}{2}r}\partial_{\lambda}^q R_V(\frac{1}{2}+i\lambda)e^{-\frac{\alpha}{2}r}||_{L^2\rightarrow L^2} \leq C_{q,\alpha} |\lambda|^{-1}.
\end{equation*}
If the $R_V(\frac{1}{2}+i\lambda)$ has no pole at $\lambda=0$, we can extend the estimate through $\lambda=0$ to give
\begin{equation*}
||e^{-\frac{\alpha}{2}r}\partial_{\lambda}^q R_V(\frac{1}{2}+i\lambda)e^{-\frac{\alpha}{2}r}||_{L^2\rightarrow L^2} \leq C_{q,\alpha} \langle \lambda \rangle ^{-1}.
\end{equation*}
\end{lemma}

In order to prove Proposition \ref{dispersive estimate proposition}, we also need the following lemma.
\begin{lemma}   \label{important lemma}                             %¹À¼Æ
\begin{equation} \label{lemma 4.5}
\int_r^{\infty} \frac{e^{i \frac{(s+a)^2}{4t}}(s+a)}{\sqrt{\cosh s- \cosh r}}ds  \lesssim
\left\{
\begin{aligned}
\sqrt{t} \sqrt{\frac{r+a}{\sinh r}},\ \  r\geq  \frac{\sqrt{t}}{t},\\
\sqrt{t}(1+\frac{a}{r}),\ \  r < \frac{\sqrt{t}}{2}.
\end{aligned}
\right.
\end{equation}
\end{lemma}

                                                %ÒýÀí  5.5
\begin{proof}
The proof roughly follows the approach in \cite{Ba}. Before proving the lemma, we recall two useful estimates, that is,
\begin{equation}            \label{lemma 4.5 useful estimate 1}
\frac{1}{\cosh s-\cosh \rho} \leq \frac{c}{(s-\rho)\sqrt{\cosh \rho}} \leq \frac{c}{s-\rho},\ \ \ {\rm for}\ s>\rho \geq 0,
\end{equation}
and
\begin{equation}            \label{lemma 4.5 useful estimate 2}
\frac{1}{\cosh s-\cosh \rho}\leq \frac{c}{\sqrt{(s-\rho)\sinh \rho}},\ \ \ {\rm for}\ s>\rho >0.
\end{equation}

Case 1: $r \geq \frac{\sqrt{t}}{2}$.

Let $s=\frac{\tau t}{r+a}+r$, then
\begin{align}                               %±äÁ¿Ìæ»»
{\rm LHS}(\ref{lemma 4.5}) =& \displaystyle \int_0^{\infty} \frac{e^{\frac{i}{4t} (\frac{\tau ^2 t^2}{(r+a)^2}  +(r+a)^2 +\tau t)}    (\frac{\tau t}{r+a}  +r+a)}{\sqrt{ \cosh (\frac{\tau t}{r+a} +r)  -\cosh r}}   \frac{t}{r+a}  d \tau, \nonumber \\           \label{lemma 4.5  3}
=& 2t e^{i\frac{(r+a)^2}{4t}} \displaystyle \int_0^{\infty}  \frac{e^{i(\frac{\tau ^2 t}{4(r+a)^2} +\frac{\tau}{2})} (\frac{\tau t}{2(r+a)^2}  +\frac{1}{2})}{\sqrt{ \cosh (\frac{\tau t}{r+a} +r)  -\cosh r} }  d\tau.
\end{align}
Denote $\Phi(\tau):=\frac{\tau^2 t}{4(r+a)^2}+\frac{\tau}{2}$, then (\ref{lemma 4.5  3}) can be written as
\begin{equation}                                   \label{lemma 4.5  4}
2\sqrt{t} e^{i\frac{(r+a)^2}{4t}} \sqrt{\frac{r+a}{\sinh r}} \displaystyle \int_0^{\infty}  \frac{e^{i\Phi(\tau)} \Phi '(\tau)}{\sqrt{\frac{r+a}{t\sinh r} (\cosh (\frac{\tau t}{r+a} +r)  -\cosh r}) }  d\tau
\end{equation}
Since $\int_0^{\infty}= \int_0^1+\int_1^{\infty}$, (\ref{lemma 4.5  4}) can be split into
\begin{equation}     \label{lemma 4.5 split}
2\sqrt{t} e^{i\frac{(r+a)^2}{4t}} \sqrt{\frac{r+a}{\sinh r}}(I_1+I_2).
\end{equation}
For $I_1$, by (\ref{lemma 4.5 useful estimate 2}) and $r\geq \frac{\sqrt{t}}{2}$ we have  %¼ÆËãI_1
\begin{align*}                                  %±äÁ¿Ìæ»»
I_1 \lesssim &  \int_0^1 \frac{\frac{\tau t}{2(r+a)^2}+\frac{1}{2}}{\sqrt{\frac{r+a}{t\sinh r} \frac{\tau t}{r+a} \sinh r}} d\tau, \\
 \lesssim &  \int _0^1 \frac{1}{\sqrt{\tau}}d\tau, \\
\lesssim & 1 .
\end{align*}
For $I_2$, Let                                                     %¼ÆËãI_2
\begin{equation}                      \label{lemma 4.5 alpha(tau)}
\alpha (\tau)=[\frac{r+a}{t\sinh r}(\cosh(\frac{\tau t}{r+a} +r)-\cosh r)]^{-\frac{1}{2}}.
\end{equation}
By integrating by parts in $I_2$, we get
\begin{align*}                                 %±äÁ¿Ìæ»»
I_2  = &  \int_1^{\infty} \partial_{\tau} e^{i\Phi(\tau)} \alpha(\tau)d\tau, \\
 \leq  & |e^{i\Phi(\tau)}\alpha(\tau)|_1^{\infty}|+|\int _1^{\infty}e^{i\Phi(\tau)} \alpha '(\tau)| d\tau ,\\
\lesssim & \sup\limits_{\tau \geq 1} |\alpha(\tau) |+ \int_1^{\infty} |\alpha '(\tau)|d\tau.
\end{align*}
Notice that $\alpha (1)\lesssim 1$ and $\alpha '(\tau)<0$ by (\ref{lemma 4.5 alpha(tau)}), hence,
\begin{equation*}
|I_2|\lesssim |\alpha(1)| \lesssim 1.
\end{equation*}
That is $I_1$ and $I_2$ are bounded. Therefore (\ref{lemma 4.5}) follows (\ref{lemma 4.5 split}) in the region $r \geq \frac{\sqrt{t}}{2}$.

Case 2: $r <\frac{\sqrt{t}}{2}$.

Let us split the left hand side of (\ref{lemma 4.5}) into three parts:
\begin{equation*}
{\rm LHS}(\ref{lemma 4.5})=\int_{r+a}^{2r+a} +\int_{2r+a}^{\sqrt{t}+a} +\int_{\sqrt{t}+a}^{\infty}=:J_1+J_2+J_3.
\end{equation*}
For $J_1$, we assume $r>0$, otherwise $J_1=0$ immediately, then                                  %¹À¼ÆI_1
\begin{align*}
J_1  \lesssim & \int_{r+a}^{2r+a} \frac{u}{\sqrt{\cosh (u-a)-\cosh r}} du,\\
     \lesssim & \int_{r+a}^{2r+a} \frac{u}{\sqrt{(u-a-r)\sinh r}} du,\\
    \lesssim &  (2r+a)\sqrt{\frac{r}{\sinh r}}.
\end{align*}
Since we are in the case $r <\frac{\sqrt{t}}{2}$, we get that
\begin{equation*}
|J_1| \lesssim \sqrt{t}+a.
\end{equation*}
For $J_2$, by (\ref{lemma 4.5 useful estimate 1}) we have                                             %¹À¼ÆI_2
\begin{align*}
J_2  \lesssim & \int_{2r+a}^{\sqrt{t}+a} \frac{u}{\sqrt{\cosh (u-a)-\cosh r}} du,\\
     \lesssim & \int_{2r+a}^{\sqrt{t}+a} \frac{u}{u-a-r} du,\\
    = &  \int_{2r+a}^{\sqrt{t}+a} 1+\frac{a+r}{u-a-r} du,\\
    \lesssim & \sqrt{t}+\frac{a}{r}\sqrt{t}.\\
\end{align*}
For $J_3$, let $u=\sqrt{t}\tau +a$, we get that                             %¹À¼ÆI_3
\begin{equation*}
J_3   =   \int_1^{\infty} \frac{e^{\frac{i}{4t}(t \tau^2 +a^2+2\sqrt{t}\tau)}(\sqrt{t}\tau+a)}{\sqrt{\cosh (\sqrt{t}\tau) -\cosh r}} \sqrt{t} d\tau
      =     2te^{\frac{ia^2}{4t}}\int_1^{\infty} \frac{e^{i(\frac{\tau^2}{4}+\frac{a \tau}{2\sqrt{t}})} (\frac{\tau}{2}+\frac{a}{2\sqrt{t}})  }{\sqrt{\cosh(\sqrt{t} \tau)-\cosh r}}   d\tau.
\end{equation*}
Then $J_3$ can be written as
\begin{equation*}
J_3  =   2te^{\frac{ia^2}{4t}}\int_1^{\infty}  e^{i\psi(\tau)}  \psi '(\tau)  \beta (\tau) d\tau,
\end{equation*}
where $\psi(\tau)=\frac{\tau^2}{4}+\frac{a \tau}{2\sqrt{t}}$ and $\beta (\tau)= [\cosh(\sqrt{t} \tau)-\cosh r]^{-\frac{1}{2}}$.
By integration by parts, we get
\begin{equation*}
|J_3| \leq 2t(|\beta (1)|+\int_1^{\infty} |\beta '(\tau)|d\tau).
\end{equation*}
Since the derivative of $\beta$ is negative, we obtain
\begin{equation*}
|J_3| \leq 4t|\beta(1)| \lesssim t(\cosh \sqrt{t}  -\cosh r)^{-\frac{1}{2}} \lesssim t (\sqrt{t}-r)^{-1} \lesssim \sqrt{t}.
\end{equation*}
Therefore, we have
\begin{equation*}
{\rm LHS}(\ref{lemma 4.5}) \lesssim \sqrt{t} +a+\frac{a}{r}\sqrt{t} \lesssim \sqrt{t} +\frac{a}{r}\sqrt{t},
\end{equation*}
in the region $r< \frac{\sqrt{t}}{2}$.
\end{proof}

                                                          % proof of Thm5.1.
\begin{proof}[Proof of Proposition \ref{dispersive estimate proposition}]
The estimate for $|t|\geq 1$ in (\ref{dispersive estimate inequality}) has been proved in \cite{BoMa}, we only prove the case $0<|t|<1$ here. In order to estimate $e^{it(\Delta_{\H^2}-V)}$, it suffices to bound (\ref{resolvent representation.1})-(\ref{resolvent representation.3}) respectively.
(\ref{resolvent representation.1}) is indeed $e^{it\Delta_{\H^2}}f$, which can be estimated in \cite{AP}. To estimate (\ref{resolvent representation.2}), we rewrite it by (\ref{R0}) as
\begin{align}                               %°ÑÒª¹À¼ÆµÄ±í´ïÊ½ÖØÐ´
 & \int_0^{\infty} e^{-it \lambda^2} \lambda \int_{z_0,z_1}V(z_1) \int_{r_0}^{\infty} \int_{r_1}^{\infty} \frac{\sin \lambda (s+s')}{\sqrt{\cosh s-\cosh r_1} \sqrt{\cosh s'-\cosh r_0} }   ds' ds f(z_0)dz_0 dz_1d \lambda  ,    \nonumber   \\                                  \label{thm 4.1 second term 1}
=&  t^{-\frac{3}{2}} \int_{z_0,z_1}V(z_1)\int_{r_0}^{\infty} \int_{r_1}^{\infty} \frac{e^{i\frac{(s+s')^2}{4t}} (s+s')}{\sqrt{\cosh s-\cosh r_1} \sqrt{\cosh s'-\cosh r_0} }  ds'ds f(z_0)dz_0 dz_1 .
\end{align}                                                         %ÀûÓÃLemma 5.5
By Lemma \ref{important lemma}, since $\sqrt{r_1+s} \leq \sqrt{r_1}+\sqrt{s}$ for $r_1\geq 0$ and $s\geq 0$ we get
\begin{align}                                         \label{thm 4.1 second term 2}
 &  \int_{r_0}^{\infty} \int_{r_1}^{\infty} \frac{e^{i\frac{(s+s')^2}{4t}} (s+s')}{\sqrt{\cosh s-\cosh r_1} \sqrt{\cosh s'-\cosh r_0} }  ds'ds\\
\lesssim & t^{1/2} \int_{r_0}^{\infty} \frac{\sqrt{\frac{r_1 +s}{\sinh r_1}}  1_{\geq \frac{\sqrt{t}}{2}}(r_1)  +(1+\frac{s}{r_1}) 1_{<\frac{\sqrt{t}}{2}}(r_1) }{\sqrt{\cosh s- \cosh r_0}} ds,\nonumber \\
\lesssim & t^{1/2} \int_{r_0}^{\infty} \frac{ \frac{\sqrt{r_1}+\sqrt{s}}{\sqrt{\sinh r_1}} 1_{\geq \frac{\sqrt{t}}{2}}(r_1)  +(1+\frac{s}{r_1}) 1_{<\frac{\sqrt{t}}{2}}(r_1) }{\sqrt{\cosh s- \cosh r_0}} ds, \nonumber \\               \label{thm 4.1 second term I}
\lesssim & t^{1/2} \int_{r_0}^{\infty} \frac{1}{\sqrt{\cosh s -\cosh r_0}} ds (\sqrt{\frac{r_1}{\sinh r_1}}1_{\geq \frac{\sqrt{t}}{2}}(r_1)+1_{<\frac{\sqrt{t}}{2}(r_1)})\\                         \label{thm 4.1 second term II}
    &        +t^{1/2}  \int_{r_0}^{\infty} \frac{\sqrt{s}}{\sqrt{\cosh s -\cosh r_0}} ds \frac{1}{\sqrt{\sinh r_1}} 1_{\geq \frac{\sqrt{t}}{2}}(r_1)\\                                                       \label{thm 4.1 second term III}
     &        +t^{1/2}    \int_{r_0}^{\infty} \frac{s}{\sqrt{\cosh s -\cosh r_0}} ds \frac{1}{ r_1} 1_{< \frac{\sqrt{t}}{2}}(r_1).
\end{align}
It suffices to estimate the three integrals the right hand side. By (16) of \cite{AP}, for $r_0>0$ we have
\begin{equation}             \label{thm 4.1 second term 3}
\begin{aligned}
\int_{r_0}^{\infty} \frac{1}{\sqrt{\cosh s -\cosh r_0}} ds
\lesssim &  \int_{r_0}^{r_0+1}  \frac{1}{\sqrt{(s-r_0)\sinh r_0}} ds  + \int_{r_0+1}^{\infty} e^{-\frac{s}{2}}ds,\\
\lesssim &  \frac{1}{\sqrt{\sinh r_0}}.
\end{aligned}
\end{equation}                                                     %¹À¼ÆµÚ¶þ¸ö»ý·ÖÏî
For the second integral, we make the change of variables $u=s-r_0$,
\begin{equation}   \label{thm 4.1 second term 4}
\int_{r_0}^{\infty} \frac{\sqrt{s}}{\sqrt{\cosh s -\cosh r_0}} ds
= \int_{0}^{\infty}  \frac{\sqrt{u+r_0}}{\sqrt{\cosh (u+r_0)-\cosh r_0}} du.
\end{equation}
If $r_0=0$,
\begin{equation*}
(\ref{thm 4.1 second term 4})=  \int_0^{\infty} \frac{\sqrt{u}}{\sqrt{\cosh u-1}} du \lesssim \int_0^1 \frac{\sqrt{u}}{u}du + \int_1^{\infty} \frac{\sqrt{u}}{\sqrt{\cosh u}} du < \infty.
\end{equation*}
If $r_0>0$, by (\ref{thm 4.1 second term 3}) we have
\begin{align*}
(\ref{thm 4.1 second term 4})  \lesssim  & \int_{r_0}^{\infty} \frac{\sqrt{s-r_0}+\sqrt{r_0}}{\sqrt{\cosh s- \cosh r_0}}ds  ,\\
 \lesssim &  \int_0^{\infty} \frac{\sqrt{u}}{\sqrt{\cosh (u+r_0)-\cosh r_0}}du +\int_{r_0}^{\infty} \frac{\sqrt{r_0}}{\sqrt{\cosh s- \cosh r_0}}ds,\\
\lesssim &   \int_0^{\infty} \frac{\sqrt{u}}{\sqrt{(\cosh u-1)\cosh r_0}}du + \frac{1}{\sqrt{\sinh r_0}},\\
\lesssim &  \frac{1}{\sqrt{\cosh r_0}}+ \frac{\sqrt{r_0}}{\sqrt{\sinh r_0}}.
\end{align*}
                                                             %¹À¼ÆµÚÈý¸ö»ý·ÖÏî
The third integral can be estimated similar to the second one. If $r_0=0$,
\begin{equation*}
\int_{r_0}^{\infty} \frac{s}{\sqrt{\cosh s-\cosh r_0}}ds = \int_0^{\infty} \frac{s}{\sqrt{\cosh s-1}}ds<\infty.
\end{equation*}
If $r_0>0$,
\begin{align*}
\int_{r_0}^{\infty} \frac{s}{\sqrt{\cosh s-\cosh r_0}}ds
= &     \int_{r_0}^{\infty}\frac{s-r_0}{\sqrt{\cosh s-\cosh r_0}}ds +\int_{r_0}^{\infty}\frac{r_0}{\sqrt{\cosh s-\cosh r_0}}ds,\\
\lesssim &    \frac{1}{\sqrt{\cosh r_0}}+\frac{r_0}{\sqrt{\sinh r_0}}.
\end{align*}
In conclusion, we obtained
\begin{align*}
(\ref{thm 4.1 second term 2}) \lesssim & t^{1/2} \big[ \frac{1}{\sqrt{\sinh r_0}} (\sqrt{\frac{r_1}{\sinh r_1}}1_{\geq \frac{\sqrt{t}}{2}}(r_1)+1_{<\frac{\sqrt{t}}{2}(r_1)})\\
 &        +\frac{1}{\sqrt{\cosh r_0}} \frac{1}{\sqrt{\sinh r_1}} 1_{\geq \frac{\sqrt{t}}{2}}(r_1)+\frac{1}{\sqrt{\cosh r_0}} \frac{1}{ r_1} 1_{< \frac{\sqrt{t}}{2}}(r_1) \big],\\
\lesssim & t^{1/2} \big[ ( \frac{1}{\sqrt{\sinh r_0}}  \sqrt{\frac{r_1}{\sinh r_1}}  +  \frac{1}{\sqrt{\cosh r_0}} \frac{1}{\sqrt{\sinh r_1}}  )1_{\geq \frac{\sqrt{t}}{2}}(r_1)    \\
&+( \frac{1}{\sqrt{\sinh r_0}} +   \frac{1}{\sqrt{\cosh r_0}} \frac{1}{ r_1}       )      1_{<\frac{\sqrt{t}}{2}}(r_1) \big].
\end{align*}                                             %µÃµ½resolvent expansion µÚ¶þÏîµÄ¹À¼Æ
Therefore,
\begin{align*}
(\ref{thm 4.1 second term 1}) \lesssim &  t^{-1}\int_{z_0,z_1}V(z_1)[ ( \frac{1}{\sqrt{\sinh r_0}}  \sqrt{\frac{r_1}{\sinh r_1}}  +  \frac{1}{\sqrt{\cosh r_0}} \frac{1}{\sqrt{\sinh r_1}}  )1_{\geq \frac{\sqrt{t}}{2}}(r_1)   \\
&   +( \frac{1}{\sqrt{\sinh r_0}} +   \frac{1}{\sqrt{\cosh r_0}} \frac{1}{ r_1}       )      1_{<\frac{\sqrt{t}}{2}}(r_1)]g(z_0) d z_0 d z_1 , \\
\lesssim &  |t|^{-1} \left \lVert g \right \rVert _{L^1}.
\end{align*}

                                            %¿ªÊ¼¹À¼Æresolvent expansion µÄµÚÈýÏî¡£
Finally, we estimate the (\ref{resolvent representation.3}). By duality, it suffices to prove for $\left \lVert h \right \rVert_{L^1}=1$,
\begin{equation}      \label{thm 4.1 third term 1}
\langle \int_0^{\infty} e^{it \lambda^2} \lambda \int_{\H^2}  \Im [R_0 V]R_V [VR_0]g(w)dwd\lambda,h \rangle \lesssim |t|^{-1}\left \lVert g \right \rVert _{L^1}.
\end{equation}
write
\begin{align*}
A(g)(z)=\int   V(z)R_0(-\frac{1}{2}+i\lambda;z,z_0)g(z_0)dz_0,\\
B(g)(z)=\int   V(z)\partial_{\lambda} R_0(-\frac{1}{2}+i\lambda;z,z_0)g(z_0)dz_0.
\end{align*}
Then by integration by parts and Lemma \ref{important lemma}, we have
\begin{align*}
(\ref{thm 4.1 third term 1})
= &   \int_0^{\infty} e^{it \lambda^2} \lambda \int  \Im R_0(-\frac{1}{2}+i\lambda; y_0,y_1) V(y_1) R_V(-\frac{1}{2}+i\lambda; y_1,z_1) \\
  &   V(z_1)R_0(-\frac{1}{2}+i\lambda; z_1,z_0)g(z_0)dz_0dz_1dy_1h(y_0)dy_0d\lambda,\\
= &   \int_0^{\infty} e^{it \lambda^2} \lambda \int  R_V(-\frac{1}{2}+i\lambda; y,z) V(z)R_0(-\frac{1}{2}+i\lambda);z,z_0)g(z_0)dz_0dz \\
  &  \int V(y) R_0(y,y_0)h(y_0) dy_0dy d\lambda ,\\
=& |t|^{-1}\int_0^{\infty} e^{it \lambda^2} [\int \partial_{\lambda} R_V(y,z)A(g)(z)dz A(h)(y)dy   \\
&   +\int R_V(y,z)B(g)(z)dz A(h)(y)dy \\
&   +\int R_V(y,z) A(g)(z)dzB(h)(y)dy ]   d\lambda,\\
\lesssim & |t|^{-1} \int_0^{\infty}  \langle  \lambda  \rangle ^{-1}  \Big(  \left \lVert   e^{\frac{1}{4}|z|} A(g)(z)  \right \rVert _{L^2}  \left \lVert   e^{\frac{1}{4}|y|} A(h)(y)  \right \rVert _{L^2} \\
   &   +      \left \lVert   e^{\frac{1}{4}|z|}B(g)(z)  \right \rVert _{L^2}  \left \lVert  e^{\frac{1}{4}|y|} A(h)(y)  \right \rVert _{L^2} \\
   &  + \left \lVert e^{\frac{1}{4}|z|} A(g)(z) \right \rVert _{L^2}  \left \lVert   e^{\frac{1}{4}|y|} B(h)(y)  \right \rVert _{L^2} \Big) d\lambda .
\end{align*}
By Lemma \ref{R_0 lambda <1} and Lemma \ref{R_0 lambda >1}, Young's inequality and Holder's inequality, we have
\begin{align*}
 \left \lVert   e^{\frac{1}{4}|z|} A(g)(z)  \right \rVert _{L^2}
 \lesssim &   \left \lVert \int  e^{\frac{1}{4}|z|} V(z)(|\log r| 1_{\leq 1}(r)+e^{-\frac{1}{2}r} 1_{>1}(r) )1_{\leq 1}(\lambda) g(z_0) dz_0 \right \rVert _{L^2}  \\
    &        + \left \lVert \int  e^{\frac{1}{4}|z|} V(z)(|\log r| 1_{\leq 1}(\lambda r)+\lambda^{-\frac{1}{2}} e^{-\frac{1}{2}r} 1_{>1}(\lambda r) ) 1_{>1}(\lambda) g(z_0) dz_0 \right \rVert _{L^2},\\
 \lesssim &  1_{\leq 1}(\lambda) \left \lVert  g  \right \rVert _{L^1} (\left \lVert  |\log r| 1_{\leq 1 }(r)  \right \rVert _{L^2}  +\left \lVert  e^{-\frac{1}{2}r} 1_{>1}(r)  \right \rVert _{L^{\infty}}        )    \\
    &     +   1_{> 1}(\lambda) \left \lVert  g  \right \rVert _{L^1} (\left \lVert  |\log r| 1_{\leq 1 }(\lambda r)  \right \rVert _{L^2}  +\left \lVert \lambda^{-\frac{1}{2}}  e^{-\frac{1}{2}r} 1_{>1}(\lambda r)  \right \rVert _{L^{\infty}}        )  ,\\
 \lesssim &  (1_{\leq 1}(\lambda)+1_{>1}(\lambda) \lambda^{-\frac{1}{2}}) \left \lVert  g  \right \rVert _{L^1} , \\
 \lesssim &  \langle \lambda \rangle ^{-\frac{1}{2}} \left \lVert  g  \right \rVert _{L^1}.
\end{align*}
Similarly, we have
\begin{equation*}
\left \lVert   e^{\frac{1}{4}|z|}B(g)(z)  \right \rVert _{L^2} \lesssim \langle \lambda \rangle ^{-\frac{1}{2}} \left \lVert  g  \right \rVert _{L^1}.
\end{equation*}
Therefore,
\begin{equation*}
(\ref{resolvent representation.3}) \lesssim \sup_{\left \lVert h \right \rVert_{L^1}=1} |t|^{-1} \int_0^{\infty} \langle \lambda \rangle ^{-\frac{3}{2}}  \left \lVert   g  \right \rVert _{L^1}  \left \lVert   h \right \rVert _{L^1}  d\lambda \lesssim |t|^{-1} \left \lVert   g  \right \rVert _{L^1}.
\end{equation*}
Thus Proposition \ref{dispersive estimate proposition} follows.
\end{proof}
                                             %ÒýÀí5.5
\subsection{The Cauchy theory}
Here we consider the Cauchy problem for (\ref{system}). The local well-posedness of (\ref{system}) is directly by Strichartz estimates in Theorem \ref{Strichartz estimates}. Then for small initial data, since the operator $-\Delta_{\H^2}-2\frac{\cosh r-1}{\sinh^2 r}$ has discrete spectrum, we use perturbation method (see \cite{TaoViZh}) to prove global well-posedness.
\begin{theorem}                          \label{Cauchy result}                             %ÊÊ¶¨ÐÔ½á¹û£¨main result£©
Consider the problem (\ref{system}) with data $\left\lVert\psi_0^{\pm}\right\rVert_{L^2}<2$, where $A_0$, $A_2$, $\psi_2$ are given by (\ref{A0 with psi+ psi-}), (\ref{A2 with psi+ psi-}), (\ref{psi+psi-}). Then there exists a unique maximal-lifespan solution pair $(\psi^+, \psi^-): I \times \R^2 \longrightarrow \C \times \C$ with $t_0 \in I$ and $\psi^{\pm}(t_0)=\psi^{\pm}_0$ with the following additional properties:\\
(i)  If $\left\lVert\psi_0^{\pm}\right\rVert_{L^2}<2$, then there exists $T=T(\psi^{\pm}_0)>0$, and a unique solution $(\psi^+,\psi^-)$
of the system in the time interval $[-T,T]$ with $(\psi^+,\psi^-)\in L^4([-T,T];L^4)\cap C([-T,T];L^2)$.\\
(ii)  If $(\psi^{+},\psi^-):(T_0,T_1)\times R^+ \longrightarrow \C \times \C $, $|T_1-T_0|<\infty$, is a solution to (\ref{system}) with $\left \lVert\psi^{\pm} \right \rVert_{L^4L^4(T_0,T_1)} < \infty $, then $\psi^{\pm}$ can be extended to a solution on a larger time interval.\\
(iii)  There exists $\epsilon >0$ such that $\left \lVert   \psi^{\pm}_0 \right \rVert _{L^2}\leq \epsilon $, then for any compact interval $J\subset \R$, (\ref{system}) has a unique global solution $\psi^{\pm}(t)\in L^4(J;L^4)\cap C(J;L^2)$, moreover, $\left \lVert   \psi^{\pm} \right \rVert _{L^4_J L^4} \lesssim C(J,\ \left \lVert   \psi^{\pm}_0 \right \rVert _{L^2})$.\\
(iv)  For every $A>0$, and $\epsilon >0$, there is $\delta>0$ such that if $\psi^{\pm}$ is a solution satisfying $\left \lVert \psi^{\pm} \right \rVert_{L^4_I L^4}\leq A$ and $M(\psi^{\pm}_0-\tilde{\psi}^{\pm}_0) \leq \delta$, then there exists a solution such that $\left \lVert \psi^{\pm} -\tilde{\psi}^{\pm} \right \rVert_{L^4_I L^4}\leq \epsilon$, and $M(\psi^{\pm}-\tilde{\psi}^{\pm}) \leq \epsilon$, $\forall\ t\in I$. \\
(v)  Assume that $R_{\pm}\psi^{\pm}_0 \in H^s$, for $s\in {1,\ 2}$. If $\left \lVert \psi^{\pm} \right \rVert_{L^4_I L^4}\leq M$, then the solution $\psi^{\pm}$ satisfies
\begin{equation}        \label{regularity thm}
\left \lVert R_{\pm} \psi^{\pm} \right \rVert_{H^s}\lesssim_M  \left \lVert R_{\pm} \psi^{\pm}_0 \right \rVert_{H^s}+1,\ \ \ \forall\ \ t\in I,
\end{equation}
and it has Lipschitz dependence with respect to the initial data.
\end{theorem}

\begin{proof}
(i) Consider the system (\ref{system of psi+ and psi- version}) in the space
$$X=\big\{(\psi^+,\psi^-)\in C([0,T];L^2) \cap L^4_TL^4: \left \lVert \psi^{\pm} \right \rVert_{C([0,T];L^2)}\leq 2 \left \lVert \psi^{\pm}_0 \right \rVert_{L^2} ,\left \lVert \psi^{\pm} \right \rVert_{L^4_TL^4}\leq \epsilon \big\}.$$
Given the formulas for $A_0$, $A_2$ and $\psi_2$ by (\ref{A0 with psi+ psi-}), (\ref{A2 with psi+ psi-}), (\ref{psi+psi-}), using Lemma \ref{basic inequality}, we obtain
\begin{equation}  \label{thm4.6.(1)}
\left \lVert   A_0 \right \rVert _{L^2}+\left \lVert  \frac{\cosh r (A_2 -1)}{\sinh^2 r} \right \rVert _{L^2} \lesssim \left \lVert   \psi^{\pm} \right \rVert _{L^4}, \mbox{ } \left \lVert   \frac{\psi_2}{\sinh r} \right \rVert _{L^4}\lesssim \left \lVert   \psi^{\pm} \right \rVert _{L^4}.
\end{equation}
In a similar argument, we also obtain that
\begin{equation}   \label{different nonlinearity}
\left \lVert   F^{\pm}(\psi^{\pm})-F^{\pm}(\widetilde{\psi}^{\pm}) \right \rVert _{L^{\frac{4}{3}}} \lesssim \left \lVert   \psi^{\pm}-\widetilde{\psi}^{\pm} \right \rVert _{L^4} (\left \lVert   \psi^{\pm} \right \rVert _{L^4}^2 +\left \lVert  \widetilde{\psi}^{\pm} \right \rVert _{L^4}^2).
\end{equation}
Denote $V=2\frac{\cosh r-1}{\sinh^2 r}$, then by Duhamel formula, define the maps
\begin{equation}    \label{contraction map T}
\begin{aligned}
\mathcal{T}^+(\psi^{+}) = e^{it(\Delta-V)}e^{i2 \theta} \psi^+_0  -i \int_0^t e^{i(t-s)(\Delta -V)}F^+ (\psi^{+}) ds,\\
\mathcal{T}^-(\psi^{-}) = e^{it\Delta}\psi^-_0  -i \int_0^t e^{i(t-s)\Delta } (F^- (\psi^{-})-V \psi^-) ds.
\end{aligned}
\end{equation}
For any $\epsilon >0$, there exists $\phi^{\pm}_0 \in C_0^{\infty}$, such that $\left \lVert  \psi_0^{\pm}-\phi_0^{\pm} \right \rVert _{L^2}< \frac{\epsilon}{4C}$, and there exists $T_1>0$, s.t $T_1^{\frac{1}{4}}\left \lVert   \phi_0^{\pm} \right \rVert _{\dot{H}^{\frac{1}{2}}} <\frac{\epsilon}{4C}$, then dispersive estimates and (\ref{thm4.6.(1)}) imply
\begin{equation*}
\begin{aligned}
&    \left \lVert \mathcal{T}^{+}(\psi^{+}) \right \rVert _{L^4_{T_1} L^4}  \\
 \leq &    \left \lVert  e^{it(\Delta -V)}e^{i2 \theta} (\psi^+_0 -\phi^+_0)  \right \rVert _{L^4_{T_1} L^4} + \left \lVert  e^{it(\Delta -V)}e^{i2 \theta} \phi^+_0  \right \rVert _{L^4_{T_1} L^4} \\
&  + \left \lVert  \int_0^t \left \lVert  e^{i(t-s)(\Delta -V)}F^+(s) \textbf{1}_{[0,T_1]}(s) \right \rVert_{L^4_x} ds \right \rVert _{L^4_t}\\
\leq & C\left \lVert \psi^+_0 -\phi^+_0  \right \rVert _{L^2}+CT_1^{\frac{1}{4}} \left \lVert \phi^+_0  \right \rVert _{\dot{H}^{\frac{1}{2}}}+C \left \lVert  \int_0^t |t-s|^{-1/2} \left \lVert F^+(s) \textbf{1}_{[0,T_1]}(s) \right \rVert_{L^{4/3}_x} ds \right \rVert _{L^4_t}\\
\leq & \frac{\epsilon}{2} +C \left \lVert  F^+(t) \textbf{1}_{[0,T_1]}(t) \right \rVert _{L^{4/3}L^{4/3}}\\
\leq &  \frac{\epsilon}{2}+C \left \lVert \psi^{\pm} \right \rVert _{L^4_{T_1} L^4}^3\\
\leq & \epsilon.
\end{aligned}
\end{equation*}
Similarly, we have
\begin{equation*}
\begin{aligned}
\left \lVert \mathcal{T}^-(\psi^{-}) \right \rVert _{L^4_T L^4} \leq &  \frac{\epsilon}{2}+C \left \lVert \psi^{\pm} \right \rVert _{L^4_T L^4}^3+C\left \lVert V \right \rVert _{L^2_T L^2} \left \lVert \psi^{-} \right \rVert _{L^4_T L^4}\\
\leq & (\frac{3}{4}+ C\left \lVert V \right \rVert _{L^2_T L^2})\epsilon.
\end{aligned}
\end{equation*}
Since $V\in L^2 $ independent on t, there exists $0<T<T_1$, such that $C\left \lVert V \right \rVert _{L^2_T L^2}< \delta$, hence $\left \lVert\mathcal{T}^{-}(\psi^{-}) \right \rVert _{L^4_T L^4}<\epsilon$.
We can also show that $\mathcal{T}^{\pm}(\psi^{\pm})\in C([0,T];L^2)$. Indeed, by Strichartz estimates, we have
\begin{equation*}
\left \lVert \mathcal{T}^+(\psi^+) \right \rVert _{C([0,T];L^2)} \leq  \left \lVert \psi^+_0 \right \rVert _{L^2}+C\left \lVert \psi^{\pm} \right \rVert _{L^4_T L^4}^3\leq  2\left \lVert \psi^+_0 \right \rVert _{L^2},
\end{equation*}
and
\begin{equation*}
\left \lVert \mathcal{T}^{-}(\psi^{-}) \right \rVert _{C([0,T];L^2)} \leq \left \lVert \psi^-_0 \right \rVert _{L^2}+C\left \lVert \psi^{\pm} \right \rVert _{L^4_T L^4}^3+C\left \lVert V \right \rVert _{L^2_T L^2} \left \lVert \psi^{-} \right \rVert _{L^4_T L^4}\leq 2\left \lVert \psi^+_0 \right \rVert _{L^2}.
\end{equation*}
Therefore, $(\mathcal{T}^+(\psi^{+}),\mathcal{T}^-(\psi^{-}))\in X$ for any $\psi^{\pm}\in X$.

Then we need to show $(\mathcal{T}^+(\psi^{+}),\ \mathcal{T}^-(\psi^{-}))$ is a contraction map. By (\ref{different nonlinearity}), we get
\begin{equation*}
\begin{aligned}
 & \left \lVert \mathcal{T}^{+}(\psi^{+})-\mathcal{T}^{+}(\tilde{\psi}^{+}) \right \rVert _{C([0,T];L^2) \cap L^4_T L^4} \\
\leq & C\left \lVert \psi^{\pm}-\tilde{\psi}^{\pm} \right \rVert _{L^4_TL^4}( \left \lVert \psi^{\pm} \right \rVert _{L^4_TL^4}^2 +\left \lVert \tilde{\psi}^{\pm} \right \rVert _{L^4_TL^4}^2),\\
\leq  & C \epsilon^2 \left \lVert \psi^{\pm}-\tilde{\psi}^{\pm} \right \rVert _{L^4_TL^4}.
\end{aligned}
\end{equation*}
and
\begin{equation*}
\begin{aligned}
 & \left \lVert \mathcal{T}^{-}(\psi^{\pm})-\mathcal{T}^{-}(\tilde{\psi}^{\pm}) \right \rVert _{C([0,T];L^2) \cap L^4_T L^4} \\
\leq  & C \epsilon^2 \left \lVert \psi^{\pm}-\tilde{\psi}^{\pm} \right \rVert _{L^4_TL^4}+ C\left \lVert V \right \rVert _{L^2_T L^2} \left \lVert \psi^{-}- \tilde{\psi}^- \right \rVert _{L^4_T L^4},\\
\leq & (C\epsilon^2 +\delta) \left \lVert \psi^{\pm}-\tilde{\psi}^{\pm} \right \rVert _{L^4_TL^4}.
\end{aligned}
\end{equation*}
In conclusion, $(\mathcal{T}^+(\psi^{+}),\mathcal{T}^-(\psi^{-}))$ is a contraction map in $X$, by the fixed point theorem, there exists a unique solution in $X$ for small $T$ depending only on $\psi^{\pm}_0$ and $\left\lVert V\right\rVert_{L^2}$.

                                                             %Ö¤Ã÷Ð¡³õÖµÕûÌåÊÊ¶¨ÐÔ¡£
(iii) Let $u^{\pm}: I\times R^+ \longrightarrow \C$ be an approximate solution to system (\ref{system}) in the sense that
\begin{equation}             \label{system of u}
\left \{
\begin{aligned}
(i\partial_t +\Delta-V)u^+ &=F^+(u^{+}),\\
(i\partial_t +\Delta)u^- &=F^-(u^{-}),\\
u^+(0)=e^{i2\theta}\psi^{+}_0,\ & u^-(0)=\psi^{-}_0.
\end{aligned}
\right.
\end{equation}
Based on standard fixed point argument, by the Strichartz estimates for Schr\"{o}dinger operators $-\Delta_{\H^2}$ and $-\Delta_{\H^2}+V$, there exists $\epsilon>0$ such that if $E(\psi^{\pm}_0)=\left \lVert \psi^{\pm}_0 \right \rVert _{L^2} \leq \epsilon$, then (\ref{system of u}) has a unique global solution $u^{\pm}\in C(\R;L^2) \cap L^4L^4$, moreover, $\left \lVert u^{\pm} \right \rVert _{L^{\infty}L^2 \cap L^4L^4(\R\times\H^2)} \leq C \epsilon$.

Now we show using a perturbative argument that (\ref{system}) is global well-posed for $E(\psi^{\pm}_0)< \epsilon$. First we show that for $T$ sufficiently small depending only on $E(\psi^{\pm}_0)$, and $V$, the solution $(\psi^+,\psi^-)$ to (\ref{system}) on $[0,T]$ satisfies an a priori estimate
\begin{equation}         \label{estimate for psi}
\left \lVert \psi^{\pm} \right \rVert _{L^{\infty}L^2 \cap L^4L^4(0,T)} \leq 8C \epsilon.
\end{equation}
Fix a small parameter $\eta>0$, since $V=\frac{\cosh r-1}{\sinh^2 r}$, there exists $T>0$ for $I=[0,T]$, such that
\begin{equation*}
\left \lVert V \right \rVert _{L^2_I L^2} < \eta.
\end{equation*}
Further, by Duhamel formula, Strichartz estimates and $\left \lVert u^{\pm} \right \rVert _{L^{\infty}L^2 \cap L^4L^4} \leq C \epsilon$, we have
\begin{equation}            \label{(iii)1}
\begin{aligned}
\left \lVert e^{it(\Delta-V)}  u^{+}(0) \right \rVert _{L^4_I L^4}
\leq & \left \lVert u^+ \right \rVert _{L^4_I L^4}+\left \lVert \int_0^t e^{i(t-s)(\Delta-V)} F^+(u^{+})ds  \right \rVert _{ L^4_I L^4},\\
\leq & C\epsilon +C(C\epsilon)^3,\\
\leq & 2C\epsilon.
\end{aligned}
\end{equation}
and
\begin{equation}              \label{(iii)2}
\begin{aligned}
\left \lVert e^{it\Delta}  u^{-}(0) \right \rVert _{L^4_I L^4}
\leq & \left \lVert u^- \right \rVert _{L^4_I L^4}+\left \lVert \int_0^t e^{i(t-s)\Delta} F^-(u^{-})ds  \right \rVert _{ L^4_I L^4},\\
\leq & 2C\epsilon.
\end{aligned}
\end{equation}
Since $(\psi^{+},\ \psi^-)$ satisfies (\ref{system}) and $u^+(0)=e^{i2\theta}\psi^{+}_0,\ u^-(0)=\psi^{-}_0$, apply the Duhamel formula, (\ref{(iii)1}) and (\ref{(iii)2}) to obtain
\begin{equation}                   \label{(iii)3}
\begin{aligned}
 \left \lVert \psi^+ \right \rVert _{L^4_I L^4}
\leq & \left \lVert e^{it(\Delta-V)} u^+(0) \right \rVert _{L^4_I L^4}+ C\left \lVert \psi^{\pm}  \right \rVert _{ L^4_I L^4}^3,\\
\leq & 2C\epsilon +C\left \lVert \psi^{\pm}  \right \rVert _{ L^4_I L^4}^3.\\
\end{aligned}
\end{equation}
and
\begin{equation*}
\begin{aligned}
\left \lVert \psi^- \right \rVert _{L^4_I L^4}
\leq & \left \lVert e^{it\Delta} u^-(0) \right \rVert _{L^4_I L^4}+ C\left \lVert \psi^{\pm}  \right \rVert _{ L^4_I L^4}^3+C\left \lVert V  \right \rVert _{ L^2_I L^2} \left \lVert \psi^{-}  \right \rVert _{ L^4_I L^4},\\
\leq & 2C\epsilon +C\left \lVert \psi^{\pm}  \right \rVert _{ L^4_I L^4}^3+C\eta \left \lVert \psi^{-}  \right \rVert _{ L^4_I L^4}.\\
\end{aligned}
\end{equation*}
Choose $\eta$ sufficiently small such that $C\eta<\frac{1}{3}$, which yields
\begin{equation}              \label{(iii)5}
\left \lVert \psi^{-}  \right \rVert _{ L^4_I L^4} \leq 3C\epsilon +\frac{3}{2}C \left \lVert \psi^{\pm}  \right \rVert^3 _{ L^4_I L^4},
\end{equation}
Combining (\ref{(iii)3}) and (\ref{(iii)5}), we have
\begin{equation*}
\left \lVert \psi^{\pm}  \right \rVert _{ L^4_I L^4} \leq 5C\epsilon +3C \left \lVert \psi^{\pm}  \right \rVert^3 _{ L^4_I L^4}.
\end{equation*}
Then by continuity argument, we get
\begin{equation}              \label{(iii)6}
\left \lVert \psi^{\pm}  \right \rVert _{ L^4_I L^4} \leq 7C\epsilon.
\end{equation}
which, together with Strichartz estimates gives
\begin{equation}              \label{(iii)7}
\begin{aligned}
\left \lVert \psi^{\pm}  \right \rVert _{ L^{\infty}_I L^2} \leq & \left \lVert e^{it(\Delta-V)} e^{i2\theta}\psi^{+}_0  \right \rVert _{ L^2} + \left \lVert e^{it\Delta} \psi^{-}_0  \right \rVert _{ L^2},\\
&  + 2C\left \lVert \psi^{\pm}  \right \rVert^3 _{ L^4_I L^4} +C \left \lVert V  \right \rVert _{ L^2_I L^2} \left \lVert \psi^{-}  \right \rVert _{ L^4_I L^4},  \\
\leq & 2\epsilon +2C(7C\epsilon)^3 +C\eta 8C\epsilon,\\
<& 4\epsilon.
\end{aligned}
\end{equation}
Therefore (\ref{estimate for psi}) is obtained.

Then from the system (\ref{system}), we have energy conservation $E(\psi^{\pm})=E(\psi^{\pm}_0)$. Since the $T$ depends only on $E(\psi^{\pm}_0)$ and $V=2\frac{\ch r-1}{\sh^2 r}$, by (ii) and energy conservation, it will follow that $(\psi^{+},\psi^-)$ is a global solution with $\left \lVert \psi^{\pm}  \right \rVert _{ L^4_J L^4} \leq C(\left \lVert \psi^{\pm}_0  \right \rVert _{ L^2},|J|)$ for any compact interval $J \subset \R$.

(v)Applying $(-\Delta)^{\frac{s}{2}}$ for $s=1,\ 2$ to both sides of system (\ref{system}), we obtain
\begin{equation*}
\left\{\begin{aligned}
&(i\partial_t+\Delta)(-\Delta)^{\frac{s}{2}}R_+\psi^+=(-\Delta)^{\frac{s}{2}}F^++(-\Delta)^{\frac{s}{2}}(VR_+\psi^+),\\
&(i\partial_t+\Delta)(-\Delta)^{\frac{s}{2}}R_-\psi^-=(-\Delta)^{\frac{s}{2}}F^--(-\Delta)^{\frac{s}{2}}(VR_-\psi^-).
\end{aligned}
\right.
\end{equation*}
The nonlinearities $F^{\pm}$ can be written as 
\begin{align*}
F^{\pm}=&\Big(-\frac{1}{2}|R_{\pm}\psi^{\pm}|^2+\int_r^{\infty}\frac{\cosh s}{\sinh s}\Re(\psi^+\bar{\psi}^-)ds\Big)R_{\pm}\psi^{\pm}\pm\frac{\cosh r}{2\sinh^2 r}\int_0^r(|\psi^+|^2-|\psi^-|^2)\sinh sdsR_{\pm}\psi^{\pm},\\
\triangleq & F_1^{\pm}\pm F_2^{\pm}.
\end{align*}
Let $\varphi(r)\in C_c^{\infty}$ be a bump function with $0\leq \varphi \leq 1$, $\varphi\big|_{B_1(0)}=1$ and $\varphi\big|_{B_2^c(0)}=0$, $F_2^{\pm}$ can be rewritten as 
\begin{equation*}
F_2^{\pm}=\varphi F_2^{\pm} +(1-\varphi)F_2^{\pm}\triangleq {\rm I}^{\pm}+{\rm II}^{\pm}.
\end{equation*}
Since $\left\lVert \psi^{\pm}\right\rVert_{L^4_IL^4}\leq M$, Strichartz estimates imply $\left\lVert \psi^{\pm}\right\rVert_{L^3_IL^6}\lesssim 1$. Then we split the interval into $I=\bigcup I_j$ such that $\left\lVert \psi^{\pm}\right\rVert_{L^3_{I_j}L^6}<\epsilon$, $\left\lVert(-\Delta)^{1/2} \psi^{\pm}_0\right\rVert_{L^2}\left\lVert \psi^{\pm}\right\rVert_{L^3_{I_j}L^6}\lesssim1$, $\left\lVert \partial_r^k V\right\rVert_{L^3_{I_j}L^6}\ll 1$ and $\left\lVert \frac{\cosh r}{\sinh r}\partial_r V\right\rVert_{L^3_{I_j}L^6}\ll 1$.
By Duhamel's formula and Strichartz estimates, we have 
\begin{equation}                 \label{regularity estimate}
\left\lVert (-\Delta)^{\frac{s}{2}}R_{\pm}\psi^{\pm}\right\rVert_{L^{\infty}_{I_j}L^2\bigcap L^3_{I_j}L^6}\lesssim \left\lVert (-\Delta)^{\frac{s}{2}}R_{\pm}\psi^{\pm}_0\right\rVert_{L^2}+\left\lVert (-\Delta)^{\frac{s}{2}}F^{\pm}\pm (-\Delta)^{\frac{s}{2}}(VR_{\pm}\psi^{\pm})\right\rVert_{L^1_{I_j}L^2}.
\end{equation}
Now we estimate the second term of the right hand side of (\ref{regularity estimate}). Define
\begin{align*}
&A=\left \lVert \psi^{\pm} \right \rVert_{L^3_{I_j}L^6},\ \   B=\left \lVert \partial_r \psi^{\pm} \right \rVert_{L^3_{I_j}L^6}+\left \lVert \frac{\psi^+}{\sinh r} \right \rVert_{L^3_{I_j}L^6},\\
&C=\left \lVert \partial_r^2 \psi^{\pm} \right \rVert_{L^3_{I_j}L^6}+\left \lVert \frac{\cosh r}{\sinh r} \partial_r \psi^{\pm} \right \rVert_{L^3_{I_j}L^6}+\left \lVert \frac{\psi^+}{\sinh^2 r} \right \rVert_{L^3_{I_j}L^6}.
\end{align*}
For $s=1$, from (\ref{Sobolev inequality 5}) we easily obtain
\begin{equation}
\begin{aligned}              \label{s=1 easy term}
&\left\lVert (-\Delta)^{\frac{1}{2}}(F_1^{\pm} \pm {\rm II}^{\pm} \pm VR_{\pm}\psi^{\pm})\right\rVert_{L^1_{I_j}L^2} \\
\lesssim & B(A^2+\left\lVert V\right\rVert_{L^{3/2}_{I_j}L^3})+A^3+\left\lVert \partial_r V\right\rVert_{L^{3/2}_{I_j}L^3}A.
\end{aligned}
\end{equation}
Since the operator $\frac{1}{r^2}\int_0^r \cdot sds$ keeps the two dimensional frequency localization, one could use Littlewood-Paley decomposition to deal with ${\rm I}^{\pm}$. To estimate ${\rm I}^{\pm}$, we claim that for $f$ radial, $p\geq 2$ the following estimate holds
\begin{equation}         \label{frequency localization}
\left\lVert (-\Delta_{\R^2})^{\frac{s}{2}}\frac{1}{r^2}\int_0^r f\ sds\right\rVert_{L^p(\R^2)}\lesssim \left\lVert (-\Delta_{\R^2})^{\frac{s}{2}} f \right\rVert_{L^p(\R^2)}.
\end{equation}
Then we have 
\begin{equation}             \label{Ipm hard term}
\begin{aligned}
&\left\lVert (-\Delta)^{\frac{1}{2}}{\rm I}^{\pm}\right\rVert_{L^1_{I_j}L^2}\\
\lesssim &\left\lVert \partial_r\big( \varphi(r)\frac{\cosh r}{\sinh^2 r}\int_0^r (|\psi^+|^2-|\psi^-|^2)\sinh sds  \big) R_{\pm}\psi^{\pm}\right\rVert_{L^1_{I_j}L^2}\\
& +\left\lVert  \varphi(r)\frac{\cosh r}{\sinh^2 r}\int_0^r (|\psi^+|^2-|\psi^-|^2)\sinh sds  \right\rVert_{L^{3/2}_{I_j}L^3}\cdot B,\\
\lesssim & \left\lVert \partial_r\big( \varphi(r)\frac{r^2\cosh r}{\sinh^2 r} \big)\frac{1}{r^2}\int_0^r (|\psi^+|^2-|\psi^-|^2)\frac{\sinh s}{s}\varphi(\frac{r}{2}) sds  \right\rVert_{L^{3/2}_{I_j}L^3}\cdot A\\
&+\left\lVert\big( \varphi(r)\frac{r^2\cosh r}{\sinh^2 r} \big) (-\Delta_{\R^2})^{1/2}\big( \frac{1}{r^2}\int_0^r (|\psi^+|^2-|\psi^-|^2)\frac{\sinh s}{s}\varphi(\frac{r}{2}) sds\big)  \right\rVert_{L^{3/2}_{I_j}L^3}\cdot A+BA^2,\\
\lesssim & A^3+BA^2.
\end{aligned}
\end{equation}
Hence, (\ref{regularity estimate}), (\ref{s=1 easy term}) and (\ref{Ipm hard term}) imply
\begin{equation}
\left\lVert (-\Delta)^{1/2}R_{\pm}\psi^{\pm} \right\rVert_{L^{\infty}_{I_j}L^2\bigcap L^3_{I_j}L^6} \lesssim \left\lVert (-\Delta)^{1/2}R_{\pm}\psi^{\pm}_0 \right\rVert_{L^2}+\left\lVert \psi^{\pm} \right\rVert_{L^3_{I_j}L^6}.
\end{equation}
We repeat the above procedure for $I_{j+1}$ to obtain the similar estimate in $I_{j+1}$. Thus, (\ref{regularity thm}) valid for $s=1$.

For $s=2$, similarly, we also easily have
\begin{equation}\label{s=2 easy term}
\begin{aligned}      
&\left\lVert (-\Delta)(F_1^{\pm}\pm{\rm II}^{\pm}\pm VR_{\pm}\psi^{\pm})\right\rVert_{L^1_{I_j}L^2} \\
\lesssim & C(A^2+\left\lVert V\right\rVert_{L^{3/2}_{I_j}L^3})+B(A^2+\left\lVert \partial_r V \right\rVert_{L^{3/2}_{I_j}L^3})+B^2A+A^3+\left\lVert \Delta V \right\rVert_{L^{3/2}_{I_j}L^3}A.
\end{aligned}
\end{equation}
Then for ${\rm I}^{\pm}$, which can be rewritten as 
\begin{align*}
(-\Delta){\rm I}^{\pm}=&(-\Delta)\big( \varphi(r)\frac{r^2\cosh r}{\sinh^2 r}R_{\pm}\psi^{\pm}\cdot \frac{1}{r^2}\int_0^{r} (|\psi^+|^2-|\psi^-|^2)\varphi(\frac{s}{2})\frac{\sinh s}{s}\ sds  \big),\\
=& (-\Delta)\big( \varphi(r)\frac{r^2\cosh r}{\sinh^2 r}R_{\pm}\psi^{\pm}\big)\cdot \varphi(\frac{r}{2}) \frac{1}{r^2}\int_0^{r} (|\psi^+|^2-|\psi^-|^2)\varphi(\frac{s}{2})\frac{\sinh s}{s}\ sds\\
&+\big( \varphi(r)\frac{r^2\cosh r}{\sinh^2 r}R_{\pm}\psi^{\pm}\big)\cdot (-\Delta)\big(\frac{1}{r^2}\int_0^{r} (|\psi^+|^2-|\psi^-|^2)\varphi(\frac{s}{2})\frac{\sinh s}{s}\ sds  \big)\\
& -2\partial_r\big( \varphi(r)\frac{r^2\cosh r}{\sinh^2 r}R_{\pm}\psi^{\pm}\big)\cdot \partial_r\big( \frac{1}{r^2}\int_0^{r} (|\psi^+|^2-|\psi^-|^2)\varphi(\frac{s}{2})\frac{\sinh s}{s}\ sds  \big),\\
\triangleq &{\rm I}_1^{\pm}+{\rm I}_2^{\pm}+{\rm I}_3^{\pm}.
\end{align*} 
By (\ref{basic inequality6}) and (\ref{frequency localization}) we have
\begin{equation}       \label{s=2 I1 I3 hard term}
\left\lVert{\rm I}_1^{\pm}+{\rm I}_3^{\pm} \right\rVert_{L^1_{I_j}L^2}\lesssim CA^2+BA^2+A^3+B^2A,
\end{equation}
for ${\rm I}_2^{\pm}$, from (\ref{frequency localization}) we obtain
\begin{equation}     \label{s=2 I2 hard term}
\begin{aligned}
\left\lVert{\rm I}_2^{\pm} \right\rVert_{L^1_{I_j}L^2}\lesssim & A\left\lVert\Delta(|\psi^+|^2-|\psi^-|^2)\varphi(\frac{r}{2})\frac{\sinh r}{r} \right\rVert_{L^{3/2}_{I_j}L^3},\\
\lesssim & CA^2+BA^2+B^2A+A^3.
\end{aligned}
\end{equation}
Thus, by (\ref{regularity estimate}), (\ref{s=2 easy term}), (\ref{s=2 I1 I3 hard term}) and (\ref{s=2 I2 hard term}) we have 
\begin{equation}
\left\lVert (-\Delta)R_{\pm}\psi^{\pm}\right\rVert_{L^{\infty}_{I_j}L^2\bigcap L^3_{I_j}L^6}\lesssim \left\lVert (-\Delta)R_{\pm}\psi^{\pm}_0\right\rVert_{L^2}+\left\lVert (-\Delta)^{1/2}R_{\pm}\psi^{\pm}_0\right\rVert_{L^2}+\left\lVert \psi^{\pm}\right\rVert_{L^3_{I_j}L^6}.
\end{equation}
Hence, (\ref{regularity thm}) follows for $s=2$.

Finally, we prove (\ref{frequency localization}). Denote $B_r=B_r(0)$ and $m_k(r)=\varphi(2^{-k}r)-\varphi(2^{-k+1}r)$. Since $f$ is radial, we have
\begin{equation}
\frac{1}{r^2}\int_0^r f\ sds=C\frac{1}{m(B_r)}\int_{\R^2} f\cdot \textbf{1}_{B_r}(y)dy=\frac{1}{m(B_r)}f\ast \textbf{1}_{B_r}(0).
\end{equation}
Then for $P_k f=\mathcal{F}^{-1}(m_k(\xi)\widehat{f}(\xi))$ , we have
\begin{align*}
&\int_{\R^2}\frac{1}{m(B_{|x|})}(P_kf\ast \textbf{1}_{B_{|x|}})(0)e^{-ix\eta}dx\\
=& \int_{\R^2} \frac{1}{m(B_{|x|})}\int \widehat{P_k f}(\xi)\widehat{\textbf{1}_{B_{|x|}}}(\xi)d\xi e^{-ix\eta}dx,\\
=& \int \widehat{P_k f}(\xi)\widehat{\textbf{1}_{B_{1}}}(|x|\xi) e^{-ix\eta}dxd\xi,\\
=& \int \widehat{P_k f}(\xi)\textbf{1}_{B_{1}}(\frac{\eta}{|\xi|}) |\xi|^{-2}d\xi,
\end{align*}
which implies
\begin{equation}            \label{frequency local}
\frac{1}{r^2}\int_0^r P_kf\ sds= P_{\leq k}\big(\frac{1}{r^2}\int_0^r P_k f\ sds\big).
\end{equation}
Hence, by Littlewood-Paley decomposition and (\ref{frequency local}) we have
\begin{align}          \nonumber
&\left\lVert (-\Delta_{\R^2})^{s/2}  \frac{1}{r^2}\int_0^r P_kf\ sds\right\rVert_{L^p}\\ \nonumber
\lesssim  & \left\lVert\Big[\sum\limits_j |\sum\limits_k 2^{sj} P_j\big(  \frac{1}{r^2}\int_0^r P_kf\ sds\big)|^2 \Big]^{1/2}\right\rVert_{L^p},\\              \label{frequency local final} 
\lesssim & \left\lVert\Big[\sum\limits_j \big(\sum\limits_k 2^{s(j-k)}\textbf{1}_{\leq 0}(j-k) \big| \frac{1}{r^2}\int_0^r 2^{sk} P_kf\ sds \big|\big)^2 \Big]^{1/2}\right\rVert_{L^p},  
\end{align}
from (\ref{basic inequality6}) we obtain
\begin{align*}
(\ref{frequency local final})
\lesssim & \left\lVert\Big[\sum\limits_j \big(\sum\limits_k 2^{s(j-k)}\textbf{1}_{\leq 0}(j-k) \big|  2^{sk} P_kf\big|\big)^2 \Big]^{1/2}\right\rVert_{L^p},\\
\lesssim & \left\lVert \{ 2^{sk}\textbf{1}_{\leq 0}(k)\ast  |2^{sk} P_kf|\}_{l^2} \right\rVert_{L^p},\\
\lesssim & \left\lVert \{2^{sk} P_kf\}_{l^2} \right\rVert_{L^p},\\
\lesssim & \left\lVert (-\Delta_{\R^2})^{s/2} f \right\rVert_{L^p}.
\end{align*}
Thus, (\ref{frequency localization}) follows.
\end{proof}

The above theorem is only concerned with the general solutions of (\ref{system}). Since the system of $(\psi^+,\ \psi^-)$ is derived from the Schr\"{o}dinger map (\ref{Schrodinger map}), if we want to reconstructed the map $u$ by $\psi^{\pm}$, the solution $\psi^{\pm}$ of (\ref{system}) must satisfies the compatibility condition (\ref{compatibility psi+psi-}).
\begin{theorem}                 \label{thm 4.7}   %ËµÃ÷½âÂú×ãÏàÈÝÐÔÌõ¼þ¡£
If $\psi^{\pm}_0\in L^2$ satisfies the compatibility condition, then $\psi^{\pm}(t)$ satisfies the compatibility condition for any $t\in I$. If, in addition, $R_{\pm}\psi^{\pm}_0\in H^3$, then (\ref{compatibility conditions}) and (\ref{curvature}) are satisfied.
\end{theorem}

\begin{proof}     %Ö¤Ã÷¶¨Àí
Given $\psi_1=\frac{\psi^+ +\psi^-}{2}$, $\frac{\psi_2}{\sinh r}=\frac{\psi^+- \psi^-}{2i}$, $A_1=0$. To prove the compatibility condition (\ref{compatibility psi+psi-}), it suffices to show that $D_1 \psi_2=D_2 \psi_1$ is preserved for $t\in I$. For this we need to derive the equation for
\begin{equation*}
F=D_2 \psi_1-D_1\psi_2.
\end{equation*}

Before deriving the equation for $F$, we give some identities from (\ref{system}). First, (\ref{A2 with psi+ psi-}) gives
\begin{equation}              \label{thm 4.7 no.1}
\partial_1 A_2-\partial_2 A_1= \Im (\psi_1 \bar{\psi}_2),
\end{equation}
Second, the system of $(\psi^+,\psi^-)$ (\ref{system}) and (\ref{A2 with psi+ psi-}) imply that
\begin{equation*}
\begin{aligned}
\partial_0 A_2- \partial_2 A_0
=&  \frac{1}{2} \int_0^r (\Re (\partial_t \psi^+  \bar{\psi}^+)-\Re (\partial_t \psi^- \bar{\psi}^-))\sinh sds,\\
=&  -\frac{1}{2} \int_0^r \partial_s (\Im(\partial_s \psi^+  \bar{\psi}^+)\sinh s) -\partial_s (\Im(\partial_s \psi^-  \bar{\psi}^-)\sinh s)ds,\\
=&  \Im(\psi_0 \bar{\psi}_2)+\Re (F \bar{\psi}_1),
\end{aligned}
\end{equation*}
where $\psi_0$ is given by (\ref{psi0})
Third, (\ref{A0 with psi+ psi-}) implies
\begin{equation*}
\partial_1 A_0-\partial_0 A_1=  \Im(\psi_1 \bar{\psi}_0) -\Re (\frac{F\bar{\psi}_2}{\sinh ^2 r}).
\end{equation*}
Finally, we obtain the following two equations from (\ref{system}) by algebraic computation and $A_2^2 +|\psi_2|^2=1$,
\begin{equation*}
\begin{aligned}
&D_0 \psi_1=i[D_1(D_1+ \coth r) \psi_1+\frac{D_2D_2 \psi_1}{\sinh^2 r}-2i \frac{\cosh r A_2}{\sinh^2 r}\frac{\psi_2}{\sinh r}+i \Im(\psi_1 \frac{\bar{\psi}_2}{\sinh r})\frac{\psi_2}{\sinh r}],\\
&D_0 \psi_2=i[(D_1+ \coth r)D_1 \psi_2+\frac{D_2D_2 \psi_2}{\sinh^2 r}+2 \frac{\cosh r }{\sinh r}F-i \Im(\psi_1 \bar{\psi}_2)\psi_1].
\end{aligned}
\end{equation*}
Then combining the above two equations with (\ref{differentialpsi0}), we have
\begin{equation}               \label{thm 4.7 no.2}
\begin{aligned}
&D_1\psi_0- D_0 \psi_1=\frac{-i}{\sinh^2 r} D_2 F,\\
&D_2\psi_0- D_0 \psi_2=i(D_1-\coth r)  F.
\end{aligned}
\end{equation}
Apply the operator $D_0$ to $F$, by (\ref{thm 4.7 no.1})-(\ref{thm 4.7 no.2}), we have
\begin{equation*}
\begin{aligned}
D_0 F
=& D_0 D_2 \psi_1 -D_0D_1 \psi_2,\\
=& D_2 D_0 \psi_1- D_1 D_0\psi_2 +i\Im(\psi_0 \bar{\psi}_2)\psi_1+i\Re (F\bar{\psi}_1)\psi_1 \\
 & - i\Im (\psi_0 \bar{\psi}_1)\psi_2-i\Re(\frac{F\bar{\psi}_2}{\sinh ^2 r})\psi_2, \\
=& D_2 D_1 \psi_0+D_2(\frac{i}{\sinh^2 r}D_2 F)- D_1 (D_2\psi_0-i(D_1-\coth r)F)\\
 &   +i\Im(\psi_0 \bar{\psi}_2)\psi_1-i\Im (\psi_0 \bar{\psi}_1)\psi_2+ i\Re (F\bar{\psi}_1)\psi_1 -i\Re(\frac{F\bar{\psi}_2}  {\sinh ^2 r})\psi_2,\\
=&  D_2 D_1 \psi_0- D_1 D_2\psi_0 +i\Im(\psi_0 \bar{\psi}_2)\psi_1  +i\Im (\psi_1 \bar{\psi}_0)\psi_2 -\frac{iA_2^2}{\sinh^2 r}F \\
& +i\partial_r (\partial_r -\coth r)F+ i\Re (F\bar{\psi}_1)\psi_1 -i\Re(\frac{F\bar{\psi}_2}{\sinh ^2 r})\psi_2,\\
=& -\frac{iA_2^2}{\sinh^2 r}F+i\partial_r (\partial_r -\coth r)F+ i\Re (F\bar{\psi}_1)\psi_1 -i\Re(\frac{F\bar{\psi}_2}{\sinh ^2 r})\psi_2.
\end{aligned}
\end{equation*}
So we derive equation for F:
\begin{equation*}
(i\partial_t+\partial_r^2 -\coth r \partial_r) F=(A_0+\frac{A_2^2}{\sinh^2 r}+\partial_r(\coth r))F-\Re(F\bar{\psi}_1)\psi_1+\Re(\frac{F\bar{\psi}_2}{\sinh^2 r})\psi_2,
\end{equation*}
namely
\begin{equation}     \label{F equation}
(i\partial_t+\Delta-\frac{1}{\sinh^2 r})\frac{F}{\sinh r}=A_0 \frac{F}{\sinh r}+\frac{A_2^2-1}{\sinh^2 r}\frac{F}{\sinh r}-\Re(\frac{F}{\sinh r}\bar{\psi}_1)\psi_1+\Re(\frac{F}{\sinh r} \frac{\bar{\psi}_2}{\sinh r})\frac{\psi_2}{\sinh r}.
\end{equation}

                                               %ÀûÓÃpsiµÄÕýÔòÐÔµÃµ½F/sinhr µÄÕýÔòÐÔ
If $R^{\pm}\psi^{\pm}\in H^1$, we can write
\begin{equation*}
\begin{aligned}
\frac{F}{\sinh r}
=&\frac{1}{\sinh r}(iA_2 \psi_1-\partial_r \psi_2),\\
=& \frac{i}{2}\big[ (A_2 +1)\frac{\psi^+}{\sinh r}+\frac{A_2 -1}{\sinh r}\psi^- +\partial_r( \psi^+ -\psi^-)+\frac{\cosh r-1}{\sinh r}(\psi^+ -\psi^-)\big],
\end{aligned}
\end{equation*}
Due to the boundedness of  $A_2$ and $\frac{\cosh r-1}{\sinh r}$, we get $\frac{F}{\sinh r}\in L^2$.

If $R^{\pm}\psi^{\pm}\in H^2$, we using $A_2^2+|\psi_2|^2=1$ and Sobolev embedding, yields $\frac{F}{\sinh r}\in \dot{H}^1$ by the representation
\begin{equation*}
\begin{aligned}
-2i\frac{F}{\sinh^2 r}
=& (A_2+1)\frac{\psi^+}{\sinh^2 r}+\frac{A_2-1}{\sinh r} \frac{\psi^-}{\sinh r}+\frac{\cosh r}{\sinh r}(\partial_r \psi^+ -\partial_r \psi^-)\\
 & +\frac{\cosh r-1}{\sinh^2 r}(\psi^+- \psi^-)+\frac{\cosh r-1}{\sinh r}\partial_r(\psi^-- \psi^+),\\
-2i\partial_r(\frac{F}{\sinh r})
=& \frac{\partial_r A_2}{\sinh r}(\psi^++\psi^-) -\frac{\cosh r-1}{\sinh^2 r}A_2 (\psi^++ \psi^-) -\frac{A_2+1}{\sinh^2 r}\psi^+ -\frac{A_2-1}{\sinh^2 r}\psi^- \\
 & \frac{A_2}{\sinh r}\partial_r(\psi^+ +\psi^-) +\partial_{rr}(\psi^+-\psi^-) +\frac{\cosh r}{\sinh r} \partial_r(\psi^+ -\psi^-).
\end{aligned}
\end{equation*}

Let $P_{\epsilon}$ for $\epsilon>0$ be the smoothing operator defined by the Fourier multiplier $\lambda\rightarrow e^{-\epsilon^2 \lambda^2}$. Denote $N$ is the nonlinearity of (\ref{F equation}). Applying $P_{\epsilon}$ to both sides of (\ref{F equation}), we obtain
\begin{equation}
(i\partial_t+\Delta_{\H^2})P_{\epsilon}(e^{i\theta} \frac{F}{\sinh r})=P_{\epsilon}(e^{i\theta}N).
\end{equation}
Since $P_{\epsilon}(e^{i\theta} \frac{F}{\sinh r})$, $\partial_r P_{\epsilon}(e^{i\theta} \frac{F}{\sinh r})$ and $\frac{1}{\sinh r}P_{\epsilon}(e^{i\theta} \frac{F}{\sinh r})\in L^2$, which implies 
\begin{equation}
\partial_r P_{\epsilon}(e^{i\theta} \frac{F}{\sinh r})\cdot\sinh r,\ \ P_{\epsilon}(e^{i\theta} \frac{F}{\sinh r})\rightarrow 0, \ \ {\rm as}\ \ r\rightarrow 0,
\end{equation} 
and 
\begin{equation}
\partial_r P_{\epsilon}(e^{i\theta} \frac{F}{\sinh r})\sinh^{1/2}r,\ \ P_{\epsilon}(e^{i\theta} \frac{F}{\sinh r})\sinh^{1/2}r\rightarrow 0,\ \  {\rm as}\ \ r\rightarrow \infty.
\end{equation}
Hence, by integration by parts and (\ref{Sobolev inequality 4}), we get
\begin{align*}
\partial_t \left\lVert P_{\epsilon}(e^{i\theta} \frac{F}{\sinh r})\right\rVert_{L^2}^2=& 2\Re(i\partial_r P_{\epsilon}(e^{i\theta} \frac{F}{\sinh r})\cdot P_{\epsilon}(e^{i\theta} \frac{F}{\sinh r})\sinh r)\Big|_0^{\infty}\\
& -2\int\Re(iP_{\epsilon}(e^{i\theta}N)P_{\epsilon}(e^{i\theta} \frac{F}{\sinh r})){\rm dvol}_g,\\
\leq & 2\left\lVert P_{\epsilon}(e^{i\theta} \frac{F}{\sinh r})\right\rVert_{L^2}\left\lVert P_{\epsilon}(e^{i\theta}N)\right\rVert_{L^2},\\
\leq & 2\left\lVert P_{\epsilon}(e^{i\theta} \frac{F}{\sinh r})\right\rVert_{L^2}^2 \left\lVert \psi^{\pm}\right\rVert_{H^2}^2,
\end{align*}
which further gives
\begin{align*}
\left\lVert P_{\epsilon}(e^{i\theta} \frac{F}{\sinh r})\right\rVert_{L^2}^2(t)\leq &\left\lVert P_{\epsilon}(e^{i\theta} \frac{F}{\sinh r})\right\rVert_{L^2}^2(0)+2\int_0^t \left\lVert \frac{F}{\sinh r}\right\rVert_{L^2}^2\left\lVert \psi^{\pm}\right\rVert_{H^2}^2ds,\\
\leq &\left\lVert  \frac{F}{\sinh r}\right\rVert_{L^2}^2(0)+2\int_0^t \left\lVert \frac{F}{\sinh r}\right\rVert_{L^2}^2\left\lVert \psi^{\pm}\right\rVert_{H^2}^2ds.
\end{align*}
Then let $\epsilon\rightarrow 0$, we obtain
\begin{equation}
\left\lVert  \frac{F}{\sinh r}\right\rVert_{L^2}^2(t)\leq \left\lVert  \frac{F}{\sinh r}\right\rVert_{L^2}^2(0)+2\int_0^t \left\lVert \frac{F}{\sinh r}\right\rVert_{L^2}^2\left\lVert \psi^{\pm}\right\rVert_{H^2}^2ds.
\end{equation}
By using Gronwall inequality and $F(0)=0$, we get $F(t)=0$ for all $t\in I$.

In general, if $\psi^{\pm}_0\in L^2$ only, there exists $R_{+}\psi^+_{0,n}\in H^2$ such that $\left\lVert\psi^+_0-\psi^+_{0,n}\right\rVert_{L^2}\leq \frac{1}{n}$. By Lemma \ref{construct psi2 A2 from psi+}, we obtain compatible pair $R_{\pm}\psi^{\pm}_{0,n}\in H^2$ and $\left\lVert\psi^-_0-\psi^-_{0,n}\right\rVert_{L^2}\lesssim \frac{1}{n}$. By the above argument, the solutions $\psi^{\pm}_n$ with initial data $\psi^{\pm}_{0,n}$ satisfy compatibility condition. Then the compatibility condition for $\psi^{\pm}_n$ can be written as
\begin{equation*}
\psi^+_n-\psi^-_n=\int_r^{\infty} \frac{A_2(\psi^+_n+\psi^-_n)}{\sinh s}+ \frac{\cosh s}{\sinh s}(\psi^+_n-\psi^-_n)ds.
\end{equation*}
Hence, by Theorem \ref{Cauchy result} (iv), Lemma \ref{basic inequality} and the expression of $A_2$ (\ref{A2 with psi+ psi-}), we have
\begin{equation*}
\begin{aligned}
&\left\lVert\psi^+-\psi^--\int_r^{\infty} \frac{A_2(\psi^++\psi^-)}{\sinh s}+ \frac{\cosh s}{\sinh s}(\psi^+-\psi^-)ds \right\rVert_{L^2}\\
\leq &\left\lVert(\psi^+-\psi^-)-(\psi^+_n-\psi^-_n)\right\rVert_{L^2}\\
& +\left\lVert\int_r^{\infty} \frac{A_2(\psi^++\psi^-)-A_{2,n}(\psi^+_n+\psi^-_n)}{\sinh s}+ \frac{\cosh s}{\sinh s}[(\psi^+-\psi^-)-(\psi^+_n-\psi^-_n)]ds \right\rVert_{L^2},\\
\lesssim &\left\lVert  \psi^{\pm}-\psi^{\pm}_n \right\rVert_{L^2}+ \left\lVert  A_2-A_{2,n} \right\rVert_{L^{\infty}}\left\lVert  \psi^{\pm} \right\rVert_{L^2}+\left\lVert  A_{2,n} \right\rVert_{L^{\infty}}\left\lVert  \psi^{\pm}-\psi^{\pm}_n \right\rVert_{L^2},\\
\lesssim &\frac{1}{n}+\left\lVert  \int_0^r |\psi^+|^2-|\psi^-|^2-|\psi^+_n|^2+|\psi^-_n|^2 ds\right\rVert_{L^{\infty}}+(1+\left\lVert \psi^{\pm}_n \right\rVert_{L^2}^2)\frac{1}{n},\\
\lesssim & \frac{1}{n},
\end{aligned}
\end{equation*}
which complete the proof of Theorem \ref{thm 4.7}.
\end{proof}

\begin{proof}[Proof of Theorem \ref{main result 2}]
First, we claim: Given $R_{\pm}\psi^{\pm}_0 \in H^2$, $\psi^{\pm}(t)$ is the solution of (\ref{system of psi+ and psi- version}), then the map $u(t)$ constructed in Proposition \ref{construct u proposition} is a Schr\"{o}dinger map. Indeed, by Proposition \ref{construct u proposition}, we construct $u_0\in \mathfrak{H}^3$. Then by Theorem \ref{main result 1}, there exists a unique solution $u(t)\in L^{\infty}(I;\mathfrak{H}^3)$ with data $u_0$. As in Section 4.1, we construct Coulomb gauge and its field component such that they satisfy (\ref{system of psi+ and psi- version}) with initial data $\psi^{\pm}_0$. The uniqueness of the solution of (\ref{system of psi+ and psi- version}) implies $\psi^{\pm}(t)$ are the gauge representation of $\mathcal{V}^{\pm}(t)$. Therefore the map reconstructed in Proposition \ref{construct u proposition} is the Schr\"{o}dinger map $u(t)$.

Next we begin to prove the Theorem \ref{main result 2}. Given initial data $u_0\in \mathfrak{H}^3$, by Theorem 3.2 we obtain a unique local solution on $[0,T]$ for some $T>0$. In particular, if in addition $E(u_0)<\epsilon^2$ for sufficiently small $\epsilon$, we can construct the fields $\psi^{\pm}$ on interval $[0,\ T]$ satisfying (\ref{system of psi+ and psi- version}) and $\left\lVert \psi^+\right\rVert_{L^2}=\left\lVert \psi^-\right\rVert_{L^2}<\epsilon$ as in Section 4.1. By Theorem \ref{Cauchy result} (iii), the solution $\psi^{\pm}$ is defined on $J\subset \R$ for any compact interval $J$ and with $\left\lVert \psi^{\pm}\right\rVert_{L^4_JL^4}\leq C(J,\ \left\lVert \psi^{\pm}_0\right\rVert_{L^2})$. Then by Theorem \ref{Cauchy result} (v) and Proposition \ref{construct u proposition}, we construct a map $u(t)\in \mathfrak{H}^3$ coincide with the Schr\"{o}dinger map on $[0,\ T]$ from $\psi^{\pm}(t)$, moreover, $E(u(T))<\epsilon^2$. Then repeat the procedure the map reconstructed from $\psi^{\pm}(t)$ is in fact a Schr\"{o}dinger map.

For initial data $u_0\in \mathfrak{H}^1$, there exists $u_{0,n}\in\mathfrak{H}^3$ such that $\left\lVert u_0-u_{0,n}\right\rVert_{\mathfrak{H}^1}<\frac{1}{n}$. By (\ref{Lipschitz continuity of psi+-with u}), we obtain the Lipschitz continuity of $\psi^{\pm}_{0,n}$, i.e $\left\lVert \psi^{\pm}_{0,n}-\psi^{\pm}_0\right\rVert_{L^2} \lesssim \left\lVert u_{0,n}-u_0\right\rVert_{\mathfrak{H}^1}$. Then from Theorem \ref{Cauchy result} (iv), the solution of (\ref{system of psi+ and psi- version}) is Lipschitz continuous with respect to initial data, we have $\left\lVert \psi^{\pm}_n-\psi^{\pm}\right\rVert_{L^2}\lesssim \left\lVert\psi^{\pm}_{0,n}-\psi^{\pm}_0\right\rVert_{L^2}\lesssim \left\lVert u_{0,n}-u_0\right\rVert_{\mathfrak{H}^1}$ for any $t\in I$. From Proposition \ref{construct u proposition}, we get $\left\lVert u(t)-u_n(t)\right\rVert_{\mathfrak{H}^1}\lesssim \left\lVert \psi^{\pm}_n-\psi^{\pm}\right\rVert_{L^2}\lesssim \left\lVert u_{0,n}-u_0\right\rVert_{\mathfrak{H}^1}$ for any $t\in I$. Hence, we obtain the desired result.
\end{proof}

\section*{Acknowledgments}
The first author thanks Dr. Ze Li for helpful discussions.

\begin{tabular}{@{}r@{}p{16cm}@{}}
&Jiaxi Huang, {\small{Wu Wen-Tsun Key Laboratory of Mathematics, Chinese Academy of Sciences and Department of Mathematics, University of Science and Technology of China, Hefei 230026, \ Anhui, \ China}}.\\
&E-mail: jiaxih@mail.ustc.edu.cn;\\
&Youde Wang, {\small Academy of Mathematics and Systems Sciences, Chinese Academy of Sciences, Beijing 100080, P.R. China.}\\
&E-mail: wyd@math.ac.cn\\
&Lifeng Zhao, {\small{Wu Wen-Tsun Key Laboratory of Mathematics, Chinese Academy of Sciences
 and Department of Mathematics, University of Science and Technology of China, Hefei 230026, \ Anhui, \ China}}.\\
&E-mail: zhaolf@ustc.edu.cn.
\end{tabular}

\end{document}